\documentclass[12pt]{amsart}

\usepackage{amsmath,amsfonts,amssymb,graphicx,hyperref, amsthm,mathabx,amsrefs,geometry}
\hypersetup{colorlinks=true,citecolor=blue,linkcolor=blue,urlcolor=blue,pdfstartview=FitH,bookmarksdepth=3}
\usepackage[noabbrev, nameinlink]{cleveref}
\usepackage{enumitem}
\setlist[enumerate]{label=\arabic*.}

\usepackage{subcaption}
\usepackage{xspace}

\newcommand{\AdSq}{{\mathrm{Ad}}^2}

\newcommand{\specmu}{\mathbf{spec}}

\newcommand{\sinmu}{\mathbf{sin}}

\newcommand{\paren}[1]{\ensuremath{\left( #1 \right)}}
\newcommand{\set}[1]{\ensuremath{\left\{ #1 \right\}}}
\newcommand{\innerprod}[1]{\ensuremath{\left\langle #1 \right\rangle}}
\newcommand{\abs}[1]{\ensuremath{\left| #1 \right|}}
\newcommand{\norm}[1]{\ensuremath{\left\| #1 \right\|}}
\newcommand{\setdiv}{\,\middle|\,}
\newcommand{\summod}[1]{\ensuremath{\,(\mathrm{mod}\,#1)}}

\newcommand{\Matrix}[1]{\begin{pmatrix}#1\end{pmatrix}}
\newcommand{\SmallMatrix}[1]{\left(\begin{smallmatrix}#1\end{smallmatrix}\right)}

\newcommand\revdots{\mathinner{\mkern1mu\raise1pt\vbox{\kern7pt\hbox{.}}\mkern2mu\raise4pt\hbox{.}\mkern2mu \raise7pt\hbox{.}\mkern1mu}}

\newcommand{\piecewise}[1]{\left\{\begin{matrix}#1\end{matrix}\right.}
\newcommand{\If}{\mbox{if }}
\newcommand{\Otherwise}{\mbox{otherwise}}

\renewcommand{\Re}{{\mathop{\mathgroup\symoperators Re}}}
\renewcommand{\Im}{{\mathop{\mathgroup\symoperators Im}}}
\newcommand{\sgn}{{\mathop{\mathgroup\symoperators \,sgn}}}

\newcommand{\Max}[1]{\ensuremath{\max \set{#1}}}

\newcommand{\Z}{\mathbb{Z}}
\newcommand{\R}{\mathbb{R}}
\newcommand{\N}{\mathbb{N}}

\newcommand{\C}{\mathbb{C}}

\newcommand{\wbar}[1]{\overline{#1}}
\newcommand{\wtilde}[1]{\widetilde{#1}}
\newcommand{\what}[1]{\widehat{#1}}
\newcommand{\BigO}[2][]{O_{#1}\paren{#2}}
\newcommand{\e}[1]{e\paren{#1}}
\newcommand{\trans}[1]{{#1}^T}

\newcommand{\Poch}[2]{\paren{#1}_{#2}}
\DeclareMathOperator{\Tr}{Tr}
\DeclareMathOperator*{\res}{res}
\DeclareMathOperator{\diag}{diag}
\DeclareMathOperator{\Span}{span}

\theoremstyle{plain} 
\newtheorem{thm}{Theorem}
\newtheorem*{thm*}{Theorem}

\newtheorem*{cor*}{Corollary}
\newtheorem{lem}[thm]{Lemma}
\newtheorem*{lem*}{Lemma}
\newtheorem{prop}[thm]{Proposition}
\newtheorem*{prop*}{Proposition}
\newtheorem{conj}[thm]{Conjecture}
\newtheorem*{conj*}{Conjecture}

\crefname{thm}{Theorem}{Theorems}
\crefname{lem}{Lemma}{Lemmas}
\crefname{prop}{Proposition}{Propositions}
\crefname{cor}{Corollary}{Corollaries}
\crefname{conj}{Conjecture}{Conjectures}

\newcommand{\IoI}{\hyperlink{conj:IoI}{Interchange of Integrals}\xspace}
\newcommand{\IoIC}{\hyperlink{conj:IoI}{Interchange of Integrals Conjecture}\xspace}
\newcommand{\StrongIoI}{\hyperlink{conj:StrongIoI}{Strong Interchange of Integrals}\xspace}
\newcommand{\StrongIoIC}{\hyperlink{conj:StrongIoI}{Strong Interchange of Integrals Conjecture}\xspace}
\newcommand{\AnCont}{\hyperlink{conj:AnCont}{Analytic Continuation}\xspace}
\newcommand{\AnContC}{\hyperlink{conj:AnCont}{Analytic Continuation Conjecture}\xspace}
\newcommand{\DEPS}{\hyperlink{conj:DEPS}{Differential Equations and Power Series}\xspace}
\newcommand{\DEPSC}{\hyperlink{conj:DEPS}{Differential Equations and Power Series Conjecture}\xspace}
\newcommand{\Asymps}{\hyperlink{thm:Asymps}{Asymptotics Theorem}\xspace}

\newcommand{\JWDCC}{\hyperlink{conj:JWDC}{Jacquet-Whittaker Direct Continuation Conjecture}\xspace}

\newcommand{\MBIntsC}{\hyperlink{conj:MBInts}{Mellin-Barnes Integrals Conjecture}\xspace}
\newcommand{\BesselExpand}{\hyperlink{conj:BesselExpand}{Bessel Expansion}\xspace}
\newcommand{\BesselExpandC}{\hyperlink{conj:BesselExpand}{Bessel Expansion Conjecture}\xspace}
\newcommand{\TestFuns}{\hyperlink{conj:TestFuns}{Test Functions}\xspace}

\newcommand{\dspec}{d_\text{spec}}

\newcommand{\mlb}{\\ \hspace*{0.3in}}
\newcommand{\mlbb}{\\ \hspace*{0.6in}}

\title[Bessel functions on $GL(n)$]{Bessel functions on $GL(n)$, I}
\author{Jack Buttcane}
\date{12 December 2022}
\address{Department of Mathematics \& Statistics, 5752 Neville Hall, Orono, ME 04469, USA}
\email{jack.buttcane@maine.edu}

\begin{document}

\begin{abstract}
In the context of the Kuznetsov trace formula, we outline the theory of the Bessel functions on $GL(n)$ as a series of conjectures designed as a blueprint for the construction of Kuznetsov-type formulas with given ramification at infinity.
We are able to prove one of the conjectures at full generality on $GL(n)$ and most of the conjectures in the particular case of the long Weyl element; as with previous papers, we give some unconditional results on Archimedean Whittaker functions, now on $GL(n)$ with arbitrary weight.
We expect the heuristics here to apply at the level of real reductive groups.
A forthcoming paper will address the initial conjectures up to Mellin-Barnes integral representations in the case of $GL(4)$ Bessel functions.
\end{abstract}


\maketitle

\section{Introduction}

Bessel functions appeared in a scattering of topics in analysis and physics prior to their systematic study by Bessel \cite{Watson}.
Their introduction to number theory also follows several avenues:
The oldest the author is aware of is the work of Voronoi where the Bessel functions $K_0$ and $Y_0$ appear in his summation formula \cite{Voronoi}.
The Bessel functions occur in the theory of automorphic forms as the Fourier coefficients of Maass forms for $SL(2,\Z)$ are classical Whittaker functions and in the spherical case, these are just the $K$-Bessel functions.
Lastly, they are found in the trace formulas of Petersson \cite{Petersson} and Kuznetsov \cite{Kuz01}.
These functions are characterized either by integral formulas which happen to produce the functions that Bessel studied or by asymptotics and differential equations that can be manipulated into the Bessel differential equation.

The generalization of Maass forms to different groups leads to new meanings of the term ``Fourier coefficient'', but the (properly defined) Fourier coefficients of automorphic forms typically still satisfy asymptotics and differential equations whose solutions we also call ``Whittaker functions''.
The Voronoi formula can be generalized to any set of generating functions having some sort of functional equation and the kernel functions occurring in such a formula are typically referred to as Bessel functions.
The Petersson and Kuznetsov formulas can be viewed as special cases of a Kuznetsov formula over Maass forms with weight, and this can be generalized to Kuznetsov-type formulas on other groups (or more broadly, to Jacquet's relative trace formulas \cite{Jac03}); again, we will refer to the (conjectural) kernel functions in such formulas as Bessel functions.
For automorphic forms on a given group, the spherical Whittaker function should turn out to be among the Kuznetsov kernel functions, as should the Voronoi Bessel functions.
The functions that we propose to study here are the Kuznetsov kernel functions for the group $GL(n, \R)$.

Recent developments in number theory frequently use deep asymptotics for the Voronoi and Kuznetsov-type Bessel functions and stationary phase arguments in the integral transforms.
Currently, we have excellent understanding of the Bessel functions on $GL(2)$ (i.e. the classical Bessel functions) for example, in Olver's work \cite{Olver01}.
Even for $GL(3)$, construction of an Olver-type asymptotic expansion (which holds across a wide range of parameters) is dizzying to think about, but eventually one would like to characterize the behavior of the functions in every region of the parameter space.
However, the author's work on $GL(3)$ does give us some understanding of the Bessel functions on that group, and the purpose of the current paper is to provide a blueprint for extending this to other groups.
In particular, we certainly expect that these ideas continue to hold at the level of real reductive groups, but we restrict our attention to the case of $GL(n,\R)$ because the author is simply not well-versed in that language.
The paper is reasonably self-contained without reliance on the author's previous works and effort has been made to isolate the interactions with representation theory.

The process developed by the author for the $GL(3)$ Kuznetsov formula was:
\begin{enumerate}
\item Show that the integral transforms are given by kernel integrals; this is the \IoIC.

\item Show that the Bessel functions satisfy certain differential equations and symmetries, have predictable asymptotics and these properties uniquely characterize them; these are the \DEPSC and \Asymps.

\item Solve the differential equations to obtain power series solutions and use the power series solutions and/or the defining properties to obtain useful integral representations of the Bessel functions.
\end{enumerate}

One may regard this paper as studying a class of Bessel functions and the \IoIC would imply these are the kernel functions in the $GL(n)$ Kuznetsov formulas.
Crucially, we won't prove the \IoIC here, but we do prove the basic properties of the Bessel functions, especially in the case of the long Weyl element and in a forthcoming paper, we will construct the $GL(4)$ Bessel functions.

The full detail of the conjectures and what we can prove is laid out in \cref{sect:ConjConseq}.
We generally limit the scope to those properties associated with the construction of spectral-type Kuznetsov formulas as in \cref{sect:SpecKuz}, but we do briefly discuss arithmetic-type Kuznetsov formulas as well as integral representations of the Bessel functions and other applications in \cref{sect:Apps}.

One new ingredient in this paper, over the author's work on $GL(3)$, is the realization that a slightly stronger version of the \IoIC would imply the Bessel functions at higher weight (i.e. attached to strict subrepresentations of principal series representations attached to characters of the minimal parabolic) are given by analytic continuation from the principal series representations, and this removes the need to explicitly construct series and integral representations for the higher-weight cases.
The \StrongIoIC is somewhat technical to state and is given in \cref{sect:StrongIoI}; along the way we prove some new results about $GL(n)$ Whittaker functions in \cref{sect:WhittFunResults}, and both of these things follow from matrix decompositions given in \cref{sect:MatrixDecomps}.
Mathematica code for the decompositions is provided in \cref{sect:AppMD}.

Assuming the \IoI and \DEPS Conjectures, we are able to determine the ``small $y$'' asymptotics (i.e. the leading terms of the power series) of the Bessel functions in general; this is the \Asymps, which is proved in \cref{sect:Asymptotics}.
Finally, under the \IoIC, we are able to say quite a bit about the Bessel function attached to the long Weyl element, and these results are proved in \cref{sect:LE}.

The corresponding theory at the finite places is laid out in general in \cite[Section 6]{Baruch04}; the key distinction is that the defining integral converges (or rather, can be made to converge) absolutely for the finite places, but not the Archimedean place.
Baruch and Mao \cite{BaruchMao02} have tackled the Archimedean case from that perspective on $SL(2,\R)$, and more recent work on $GL(n)$ at the finite places can be found in \cite{Miao}.

\subsection{A technical note}
Obviously, one would hope to express the $GL(n)$ Bessel function as an integral transform of the corresponding $GL(n-1)$ Bessel function via matrix decompositions in the defining integral; certainly this trick works very well for the Whittaker functions.
(In fact, we use that idea in \cref{sect:WhittFunResults}.)
There is a technical point that seems to arise in every viewpoint the author has explored, having to do with conjugating an upper triangular matrix past a lower triangular matrix in the presence of characters of both the lower triangular matrix and the $x$-component of the Iwasawa decomposition of the product.
The conjugation results in terms that are complex exponentials of the bottom row of the lower triangular matrix, and for $n > 3$ this is no longer given by a character, preventing the descent to $GL(n-1)$.
Precisely, this can be seen in the \StrongIoIC, which requires additional variables $\alpha$ that are zero for the Bessel function, but generally not in the lower-rank integrals of the inductive construction.
If one is looking for the exact step in the matrix decompositions that this occurs, it is the conjugations \eqref{eq:BadHatter1} (on the other side of the Weyl element, $x_a$ would be lower triangular, while $x_b^*$ remains upper triangular) as well as \eqref{eq:BadHatter2} which moves a function of the top row into the first super-diagonal; this directly translates to a need for the $\alpha$ vector in \cref{sect:BesselMatrixDecomps}.

\section{Some Notation}
We take $G = GL(n,\R)/\R^+$ and $K=O(n)$; please note that this is different from previous papers where we used $GL(n,\R)/\R^\times$ and $SO(n)$.
The distinction allows us to better induct on the rank, e.g. even and odd $GL(2)$ Maass forms are distinguished by the $\det$ representation of $O(2)$ (and hence correspond to different representations of $SO(n)$ when embedded into higher rank groups), but both correspond to the trivial representation of $SO(2)$.
We regularly identify representatives of cosets with the cosets themselves; in particular, we treat elements of $G$ as matrices instead of cosets modulo the positive scalar matrices.

Take $Y$ to be the diagonal matrices in $G$, then characters of $Y$ are parameterized by two tuples $\mu \in \C^n$ and $\delta \in \set{0,1}^n$ (or $\delta \in \Z^n$ modulo 2) where $\mu_1+\ldots+\mu_n=0$ (for consistency mod $\R^+$) and the characters $p_\mu$ and $\chi_\delta$ on the diagonal matrices are given by
\[ p_\mu(\diag(a_1,\ldots,a_n)) = \prod_{i=1}^n \abs{a_i}^{\mu_i}, \qquad \chi_\delta(\diag(a_1,\ldots,a_n)) = \prod_{i=1}^n \sgn(a_i)^{\delta_i}. \]
We generally normalize the character $p_\mu$ by translating to $p_{\rho+\mu}$ where $\rho$ is the half-sum of the positive roots with coordinates $\rho_i = \frac{n+1}{2}-i$, $i=1,\ldots,n$; this places the tempered almost-spherical automorphic forms at $\Re(\mu)=0$, see the next section.

Set $V = Y \cap K$ and take $Y^+$ to be the positive diagonal matrices in $G$ so that $Y=Y^+V$.
Note that $p_\mu$ is really a character of $Y^+$ and $\chi_\delta$ a character of $V$.

Take $U(\R)$ to be the upper-triangular unipotent (i.e. diagonal entries equal to 1) matrices in $G$, then the Iwasawa decomposition of $G$ is $G=U(\R) Y^+ K$.
Characters of $U(\R)$ have the form
\[ \psi_m\Matrix{1&x_{1,2}&&&*\\&1&x_{2,3}\\&&\ddots&\ddots\\&&&1&x_{n-1,n}\\&&&&1} = \e{m_1 x_{1,2}+\ldots+m_{n-1} x_{n-1,n}}, \qquad \e{t} = e^{2\pi i t}, \]
for some $m \in \R^{n-1}$.
We say $\psi_m$ is non-degenerate if $m_1 \cdots m_{n-1} \ne 0$.

On $Y$, we take coordinates
\begin{align}
\label{eq:YCoords}
	y =& \Matrix{y_1 y_2\cdots y_n\\&y_2 \cdots y_n\\&&\ddots\\&&&y_n},
\end{align}
and we generally do not distinguish between $y \in Y$ and $y=(y_1,\ldots,y_n) \in (\R^\times)^{n-1} \times \set{\pm 1} \cong Y$ as the multiplication is the same.
In this notation, we have $\psi_y(x) = \psi_I(yxy^{-1})$.
(Note the change $y_j \mapsto y_{n-j}$ compared to \cite{ArithKuzII}, etc.)
We will denote the $n\times n$ identity matrix by $I$ or $I_n$.

We take the Weyl group $W$ to be the set of matrices with exactly one 1 on every row and column and the remaining entries all zero.
Of particular interest will be the set $W^\text{rel}$ of ``relevant'' Weyl elements (as in \cite{Jac03} or \cite[Section 7.5]{KnightlyLi01}) of the form
\begin{align}
\label{eq:RelevantWeylEles}
	w_{r_1,\ldots,r_\ell} = \Matrix{&&I_{r_1}\\&\revdots\\I_{r_\ell}}, \qquad r_1+\ldots+r_\ell = n,
\end{align}
and we denote the long Weyl element $w_{1,\ldots,1}$ by $w_l$.
The Weyl elements can also be thought of as permutation matrices, and we occasionally use cycle notation, e.g. if $n=3$, then
\[ w_{(1\,2\,3)} = w_{2,1} = \Matrix{&1\\&&1\\1}. \] 

The Weyl group acts on $\mu$ and $\delta$ by permutations using
\begin{align*}
	p_{\mu^w}(y) :=& p_\mu(w y w^{-1}), & \chi_{\delta}(y) = \chi_\delta(w y w^{-1}), & w \in W, y \in Y,
\end{align*}
and in particular, $\mu^{w_l} = (\mu_n, \mu_{n-1}, \ldots, \mu_1)$.
(Note that, e.g., $p_{(\mu^w)^{w'}}(y) = p_{\mu^w}(w' y {w'}^{-1}) = p_\mu(w (w' y {w'}^{-1}) w^{-1}) = p_{\mu^{ww'}}(y)$, so this really is a right action.)
Similarly, we can define the action on the indices; using $\mu^w_i := (\mu^w)_i$, we take $\mu_{i^w} := (\mu^w)_i$.

Suppose $w \in W$ can be written $w = \trans{\paren{e_{i_1} \, \ldots \, e_{i_n}}}$ with $e_j$ the $j$-th standard basis element (as a column vector), so $e_1 = \trans{\paren{1 \, 0 \, \ldots \, 0}}$.
Then $w e_j = e_{i_j}$.
If we define $w(j) = i_j$, then we have $\mu^w_j = (\mu^w)_j = \mu_{w^{-1}(j)}$, and we defined the right action so that $j^w = w^{-1}(j)$.
Note that we could very well have defined the left action $p_{w(\mu)}(y) = p_\mu(w^{-1} y w)$; this is clearly related to the right action by $\mu^w = w^{-1}(\mu)$, and we would have $w(\mu)_j = \mu^{w^{-1}}_j = \mu_{w(j)}$.

For any $w \in W$, we have the decomposition $U=U_w \, \wbar{U}_w$ where $U_w = U \cap w^{-1} U w$ and $\wbar{U}_w = U \cap w^{-1} \trans{U} w$.
(These are orthogonal subspaces of $U$ under the logarithm map and the Killing form.)
We will also write $U_m(\R)$ for the upper triangular unipotent matrices in $GL(m,\R)$, $1 \le m \le n$.

Unless stated otherwise, the variable $g$ will denote an element of $G$, $x$ an element of $U(\R)$, $y$ an element of $Y$, $k$ an element of $K$ and $w$ an element of $W$.
We generally embed $GL(n-1,\R) \subset G$ in the upper left corner
\[ g \mapsto \Matrix{g\\&1}, \]
and we use prime notation to distinguish objects of a similar type, not as derivatives.

\section{Spectral Kuznetsov formulas on $GL(n)$}
\label{sect:SpecKuz}
We are interested in studying Fourier-Whittaker coefficients/Hecke eigenvalues of automorphic forms (i.e. a spectral basis) in $L^2(\Gamma\backslash G)$ where $\Gamma \subset G$ is a discrete subgroup of finite covolume.
One key tool in this area is the spectral Kuznetsov formula (aka. ``relative trace formula''), which equates a spectral and a geometric/arithmetic decomposition of the Fourier coefficient
\[ \int_{U(\Z)\cap\Gamma\backslash U(\R)} P(xg) \wbar{\psi_\beta(x)} dx \]
 of a Poincar\'e series $P$, i.e. a function of the form
\[ P(g) = \sum_{\gamma \in U(\Z)\cap\Gamma\backslash \Gamma} F(\gamma g) \]
with a kernel function $F:G \to \C$ that satisfies $F(xg) = \psi_a(x) F(g)$ for some characters $\psi_a,\psi_b$ of $U(\R)$.

Automorphic forms can be arranged into representations of $G$ according to the decomposition of the right-regular representation on the $L^2$-space, and these can be further decomposed into their $K$-types, i.e. representations of $K$.
Automorphic representations can be embedded as subrepresentations of principal series representations, i.e. representations induced from characters of the group of upper triangular matrices, and also collected into families according to the smallest weight $K$-type occurring inside the representation (up to conjugation).
The principal series representations correspond to characters $p_{\rho+\mu} \chi_\delta$ of $Y$.

Take $\Lambda = \paren{\frac{k_1-1}{2},-\frac{k_1-1}{2},\ldots,\frac{k_m-1}{2},-\frac{k_m-1}{2},0,\ldots,0} \in \R^n$ with $2m\le n$, $2 \le k_i \in \N$, and consider the family of automorphic representations occurring as subrepresentations of principal series representations $\pi_{\Lambda,\Lambda+\wtilde{\mu},\delta}$ with characters of the form $p_{\rho+\Lambda+\wtilde{\mu}} \chi_\delta$ for some $\wtilde{\mu},\delta$ with $\wtilde{\mu}_1+\ldots+\wtilde{\mu}_n=0$ and $\Re(\wtilde{\mu})=0$.
These share a minimal $K$-type and are called ``tempered'' for the condition $\Re(\wtilde{\mu})=0$.
By abusive of terminology, we refer to either $\Lambda$ or the tuple $(k_1,\ldots,k_m)$ as the weight of the automorphic representation (as opposed to the weight of the representation of $K$) and we can assume that $k_{i+1} \le k_i$, $i=1,\ldots,m-1$ (up to conjugation of the representation) as well as $\wtilde{\mu}_{2i}=\wtilde{\mu}_{2i-1}$, $i=1,\ldots,m$ (for unitary representations).
Let $i \mathfrak{a}_0^*(\Lambda)$ be the set of all such $\Lambda+\wtilde{\mu}$, i.e.
\[ i \mathfrak{a}_0^*(\Lambda) = \set{\Lambda+\wtilde{\mu} \setdiv \mu \in \C^n, \Re(\wtilde{\mu})=0, \wtilde{\mu}_1+\ldots+\wtilde{\mu}_n=0, \wtilde{\mu}_{2j}=\wtilde{\mu}_{2j-1}, j=1,\ldots,m}. \]

We note that for a given $\Lambda \ne 0$, not all values of $\delta$ are possible -- i.e. yield nontrivial representations; we have the parity constraints
\begin{align}
\label{eq:LambdadeltaParity}
	\delta_{2j-1}+\delta_{2j} \equiv k_i \pmod{2}, \qquad j=1,\ldots,m,
\end{align}
see \eqref{eq:Sigmasigmad} -- and not all $\delta$ are inequivalent (i.e. if there is some Weyl element which fixes $\Lambda$ and also permutes $\delta \to \delta'$, then the data $(\Lambda,\delta)$ and $(\Lambda,\delta')$ yield the same set of forms).
(Beware of drawing a line between $\Lambda$ and the minimal $K$-type; for $GL(3)$ with $\Lambda=(0,0,0)$ and $\delta = (1,1,0)$, the minimal $K$-type has dimension 3, i.e. highest weight 1.)

In order to localize the Kuznetsov formula to such a family of automorphic representations, we take the kernel of the Poincar\'e series to be an inverse Whittaker transform
\[ F(g) = \int_{i \mathfrak{a}_0^*(\Lambda)} f(\mu) W_\sigma(g,\mu,\delta) \dspec\mu, \]
where $\sigma$ is the minimal $K$-type corresponding to the weight $\Lambda$ and $W_\sigma(g,\mu,\delta)$ is the Jacquet-Whittaker function at $K$-type $\sigma$ in the principal series representation with character $p_{\rho+\mu} \chi_\delta$.
The measure $\dspec\mu$ occurring in Wallach's Whittaker inversion formula \cite{Wallach}
\[ \int_{U(\R)\backslash G} \int_{i\mathfrak{a}_0^*(\Lambda)} f(\mu) W_\sigma(g,\mu,\delta) \dspec\mu \, \wbar{W_\sigma(g,\mu',\delta)} dg = f(\mu') \]
can be computed by a method of Goldfeld and Kontorovich \cite{GoldKont} if one has the appropriate generalization of Stade's formula \cite{Stade02}, which is a direct computation of
\begin{align}
\label{eq:StadesFormula}
	L(s,\mu,\mu',\delta) :=& \int_{U(\R)\backslash GL(n-1,\R)} \Tr\paren{W_\sigma(g,\mu,\delta) \trans{\wbar{W_\sigma(g,\mu',\delta)}}} \det(g)^s dg,
\end{align}
and one expects this to be the gamma factors of the Rankin-Selberg $L$-function \cite{Jac02}.
(There are ways to temporarily get around the use of Stade's formula, but each has its own downside.)
A solution of this problem in the principal series case can be found in \cite{IshiiMiyazaki}.

If we write the Langlands spectral expansion \cite{Langlands02} over this family of automorphic forms (we are not assuming the absence of complementary series forms) as
\[ P(g) = \int_{\mathcal{B}(\Lambda,\delta)} \innerprod{P,\Xi} \Xi(g) \, d\Xi, \]
then applying Wallach's Whittaker inversion and using Stade's formula to get rid of the spare Whittaker function gives the spectral Kuznetsov formula (see \cite[Section 8]{ArithKuzII})
\begin{align}
\label{eq:SpecKuz}
	\int_{\mathcal{B}(\Lambda,\delta)} \frac{\lambda_\Xi(b) \wbar{\lambda_\Xi(a)}}{L(1,\AdSq \Xi)} f(\mu_\Xi) d_\text{H}\Xi = \sum_{w \in W} \sum_c \frac{S_w(a,b,c)}{p_{-\rho}(c)} H_w(f, acw b^{-1} w^{-1}).
\end{align}
On the left side, $f(\mu)$ is a reasonably arbitrary test function and we've included the Rankin-Selberg constants from converting to Hecke eigenvalues $\lambda_\Xi$ in the measure $d_\text{H}\Xi$.
The right side has a sum over the Weyl group $W$ and a sum over tuples of non-zero integers $c \in \Z^{n-1}$ which we arrange into matrices of the form
\[ c=\diag(1/c_1,c_1/c_2,\ldots,c_{n-2}/c_{n-1},c_{n-1}). \]
(Note the change $c_j \mapsto c_{n-j}$ compared to \cite{ArithKuzII}, etc.)
The variables $a$ and $b$ are tuples of nonzero integers $a,b \in \Z^{n-1}$, which we treat as elements of $Y$.
We have the Kloosterman sum
\[ S_w(a,b,c) = \sum_{xcwx' \in U(\Z)\cap\Gamma/\backslash\wbar{U}_w(\Z)\cap\Gamma} \psi_a(x) \psi_b(x'), \]
defined in terms of the Bruhat decomposition $G=U(\R)YWU(\R)$ (the Kloosterman sum is taken to be zero when it is not well-defined), and finally, we have the integral transform
\begin{equation}
\label{eq:HwDef}
\begin{aligned}
	p_\rho(y) H_w(f,y) =& \int_{U(\R)\backslash GL(n-1,\R)} \int_{\wbar{U}_w(\R)} \int_{i \mathfrak{a}_0^*(\Lambda)} \frac{f(\mu)}{L(1,\mu,\mu,\delta)} \\
	& \qquad \times \Tr\paren{W_\sigma(ywug,\mu,\delta) \trans{\wbar{W_\sigma(g,\mu,\delta)}}} \dspec\mu \, \wbar{\psi_I(u)} du \, \det(g) dg.
\end{aligned}
\end{equation}

In this paper, we are interested in the special functions occurring in the integral transform $H_w(f,y)$.
The final integral over $g$ simply serves to remove a spare Whittaker function, so we will ignore it.
Also, the Whittaker function $\wbar{W_\sigma(g,\mu,\delta)}$ can be expressed in terms of $W_\sigma(g,2\Lambda-\mu,\delta)$, which is entire in $\mu$ and reasonably bounded, so we can simply include that in the test function $f(\mu)$.
Similarly, the weight $\frac{\dspec\mu}{L(1,\mu,\mu,\delta)}$ is not especially relevant as it is (conjecturally) holomorphic near $\Re(\mu)=0$, so we restrict our attention to the integral transform
\begin{align}
\label{eq:MainIntegral}
	\int_{\wbar{U}_w(\R)} \int_{i \mathfrak{a}_0^*(\Lambda)} f(\mu) W_\sigma(ywug,\mu,\delta) \dspec\mu \, \wbar{\psi_I(u)} du.
\end{align}
The spectral measure $\dspec\mu$ has some zeros on $i\mathfrak{a}_0^*(\Lambda)$ which will be relevant.
Lastly, only the relevant Weyl elements $w \in W^\text{rel}$ occur in the spectral Kuznetsov formula, by a theorem of Friedberg \cite{Friedberg01}.

We should mention two very common, equivalent constructions of the spectral Kuznetsov formulas:
\begin{enumerate}
\item One can take the inner product of two Poincar\'e series $P_1,P_2$, applying the unfolding to one of the series and proceeding precisely as above; this yields a spare weight function, which can be sent to 1 by a density argument.
\item One can take the double Fourier coefficient
\[ \int_{U(\Z)\cap\Gamma\backslash U(\R)} \int_{U(\Z)\cap\Gamma\backslash U(\R)} F(xy,x'y') \wbar{\psi(x)} dx \, \wbar{\psi'(x')} dx' \]
of the kernel function
\[ F(g,g') = \sum_{\gamma \in \Gamma} f(g^{-1} \gamma g). \]
One then takes $f$ such that on the spectral side an application of Harish-Chandra's (zonal) spherical inversion reduces to some reasonably arbitrary test function $\hat{f}$, while on the arithmetic side, one needs to show that the Fourier transform of a zonal spherical function is a product of two Whittaker functions, but the remainder of the process is identical.
\end{enumerate}
One should likely always handle the convergence issues by starting with an inner product of a Poincar\'e series with itself and apply Plancherel's identity, but having some knowledge of the convergence on the spectral side \cite{JanaNelson}, it is slightly simpler to directly consider the Fourier coefficient of a Poincar\'e series.

\section{The Conjectures and Their Consequences}
\label{sect:ConjConseq}
Let $w=w_{r_1,\ldots,r_\ell} \in W^\text{rel}$.
We define the ``compatibility condition'' for $y \in Y$ to be $\psi_y(w u w^{-1}) = \psi_I(u)$ for all $u \in U_w(\R)$, and it turns out that the condition for the Kloosterman sum above to be well-defined is exactly that $y=acwb^{-1}w^{-1}$ satisfies the compatibility condition.
Furthermore, if we treat $\wbar{U}_w$ as the quotient $U_w\backslash U$, then the compatibility condition is precisely what \eqref{eq:MainIntegral} requires to be well-defined.
Define $\hat{r}_i = r_1+\ldots+r_i$ and
\begin{align}
\label{eq:YwDef}
	Y_w = \set{\pm\Matrix{y_{\hat{r}_1} \cdots y_{\hat{r}_{\ell-1}} I_{r_1}\\&\ddots\\&& y_{\hat{r}_\ell} I_{r_{\ell-1}}\\&&&I_{r_\ell} } \setdiv y_i \in \R^\times},
\end{align}
or equivalently
\[ Y_w =\set{y \in Y \setdiv y_i = 1, i \notin \set{\hat{r}_1, \ldots, \hat{r}_\ell}}. \]
Then $Y_w$ is exactly the set of $Y$ matrices which satisfy the compatibility condition (but only for $w \in W^\text{rel}$).
We also take $\wbar{Y}_w$ to be the orthogonal complement of $Y_w$ (under the logarithm map) so that $Y_w \wbar{Y}_w = Y$.

Define $G_w := U(\R)Y_w w \wbar{U}_w(\R) w^{-1}$.
(This is slightly uncommon; more typical would be $U(\R)Y_w w \wbar{U}_w(\R)$, but our definition is more convenient here.)
For $F:G \to \C$ nice enough, $w \in W^\text{rel}$ and $y,g \in G$, define
\begin{align}
\label{eq:TwDef}
	T_w(F_g)(y) = \int_{\wbar{U}_w(\R)} F(ywug) \wbar{\psi_I(u)} du.
\end{align}

\begin{conj*}[Interchange of Integrals]
\hypertarget{conj:IoI}{}If $f(\mu)$ is holomorphic with rapid decay on an open tube domain containing $\Re(\mu)=0$, $y \in G_{w_l}$, $t\in G$, $\Lambda$ the weight of a family of automorphic representations and $\sigma$ a $K$-type occurring such representations, then for 
\[ F(g) = \int_{i \mathfrak{a}_0^*(\Lambda)} f(\mu) W_\sigma(yw x t,\mu,\delta) \dspec\mu, \]
we have
\begin{align*}
	T_w(F_t)(y) &= \int_{i \mathfrak{a}_0^*(\Lambda)} f(\mu) \wtilde{K}_w(y,t,\Lambda,\mu,\delta,\sigma) \dspec\mu
\end{align*}
for some function $\wtilde{K}_w(y,t,\Lambda,\mu,\delta,\sigma)$.
Furthermore, this function is smooth in the coordinates of $t$ and $y$ and polynomially bounded in $\norm{\mu}$.
\end{conj*}

One would like to treat the function $\wtilde{K}_w(y,t,\Lambda,\mu,\delta,\sigma)$ in the \IoIC as being defined by the (non-convergent) integral
\begin{align}
\label{eq:IoIBadDef}
	\int_{\wbar{U}_w(\R)} W_\sigma(yw x t,\mu,\delta) \wbar{\psi_I(x)} dx
\end{align}
This can be thought of (see the discussion in \cref{sect:StrongIoI}) as an example of an oscillatory integral
\[ \int_{\R^m} e^{i\phi(x)} f(x) dx \]
which just fails to converge, i.e.
\[ \int_{\R^m} \abs{f(x)} dx = \infty, \]
but for any $\epsilon>0$,
\[ \int_{\R^m} \abs{f(x)} \prod_{i=1}^m \paren{1+x_i^2}^{-\epsilon} dx < \infty. \]
In the one-dimensional case, failure of such an oscillatory integral to converge requires that, for some combination of the parameters, the oscillatory factor $e^{i\phi(x)}$ fails to oscillate at the point $\abs{x}=\infty$; for a two-dimensional integral, there must be a line along which the oscillation cancels, and so on, until a failure of convergence in the general $n$-dimensional case requires a grand conspiracy among the oscillating factors.

\begin{prop}
\label{prop:IoIConsequence}
Assume the \IoIC, then the restriction of $\wtilde{K}_w(y,t,\Lambda,\mu,\delta,\sigma)$ to $y \in G_w$ satisfies
\[ \wtilde{K}_w(y,t,\Lambda,\mu,\delta,\sigma) = K_w(y,\Lambda,\mu,\delta) W_\sigma(t,\Lambda,\delta) \]
for some function $K_w(y,\Lambda,\mu,\delta)$ which is independent of $\sigma$.
\end{prop}

Applying the proposition to the statement of the conjecture gives
\begin{equation}
\label{eq:KwDef}
\begin{aligned}
	&\int_{\wbar{U}_w(\R)} \int_{i \mathfrak{a}_0^*(\Lambda)} f(\mu) W_\sigma(yw x t,\mu,\delta) \dspec\mu \, \wbar{\psi_I(x)} dx \\
	&= \int_{i \mathfrak{a}_0^*(\Lambda)} f(\mu) K_w(y,\Lambda,\mu,\delta) W_\sigma(t,\mu,\delta) \dspec\mu.
\end{aligned}
\end{equation}

\cref{prop:IoIConsequence} implies that the Bessel function is an invariant of the irreducible representation providing the Whittaker function; i.e. we may vary the particular vector of the Whittaker model without changing the Bessel function.
Informally, the \AnContC is the stronger statement \textit{if $\pi$ is a subrepresentation of a principal series representation $\pi'$, then the Bessel functions attached to $\pi$ are the same as that of $\pi'$.}
Justification for this conjecture is simply that the same is true for the Whittaker function from which the Bessel function is defined.
\begin{conj*}[Analytic Continuation]
\hypertarget{conj:AnCont}{}Assume the \IoIC, then the function $K_w(y,\Lambda,\mu,\delta)$ is entire in $\mu$ and satisfies
\[ K_w(y,\Lambda,\mu,\delta) = K_w(y,0,\mu,\delta). \]
\end{conj*}
When assuming the \AnContC, we drop the extraneous argument 0.
Assuming the conjecture in an argument typically means we may suppose $\Re(\mu)=0$ and the case of general $\mu$ follows by analytic continuation; moreover, one may typically assume the Whittaker function in the definition is at the minimal weight by \cref{prop:IoIConsequence}.

Some simple consequences of \cref{prop:IoIConsequence} are
\begin{prop}
\label{prop:IoIConsequences2}
Assume the \IoIC and, for convenience, the \AnContC.
Define the involution $\iota:G \to G$ by $g^\iota = w_l \trans{\paren{g^{-1}}} w_l$ and let $w' \in W$.
Then the function $K_w(y,\mu,\delta)$ satisfies
\begin{align*}
	K_w(y,\mu,\delta) =& K_w(\tilde{v}y^\iota,-\mu^{w_l},\delta^{w_l}) = K_w(y,\mu^{w'},\delta^{w'}),
\end{align*}
where $\tilde{v} = v w v w^{-1}$ with $v \in Y$ having all coordinates $v_i = -1$, $i=1,\ldots,n$.
Furthermore, $K_I(y,\mu,\delta)$ is identically one.
\end{prop}

We can also formulate a more precise version of the \IoIC; informally, the \StrongIoIC states: \textit{For all $w,\delta,\mu,y,t$, the (non-convergent) integral \eqref{eq:IoIBadDef} can be rearranged through change of variables, integration by parts and a smooth, dyadic partition of unity so that it converges rapidly.}
The full statement is rather technical, so we postpone it to \cref{sect:StrongIoI}.
For the \StrongIoIC, we are reaching rather farther than in the Interchange of Integrals and assuming that the integral oscillates rapidly at infinity.
We offer two pieces of justification here:
first, that this is the case for $GL(2)$ and $GL(3)$, and second that number of oscillating factors in \eqref{eq:StrongIoI} grows faster in $n$ than the number of integrals, so the equations required to ensure a lack of oscillation likely form an overconstrained system.

\begin{prop}
\label{prop:StrongIoIConsequences}
The \StrongIoIC implies the \IoI and \AnCont Conjectures.
\end{prop}

Note: In a forthcoming paper, the author hopes to demonstrate the \StrongIoI and \DEPS (below) Conjectures in the case of $GL(4)$.
From early computations it appears that, while the \IoIC remains true in all cases, the \StrongIoIC fails for precisely one Weyl element, $w_{2,2}$, on $GL(4)$.
The obstructions for that element appear to relate to a (literal) symmetry in $\wbar{U}_{w_{2,2}}$ and a small presence on the first super-diagonal, and suggest that some elements $GL(n)$, $n \ge 5$ will fail the stronger conjecture.
(Again, the weaker conjecture should continue to hold.)
To overcome this, one can replace Shalika's multiplicity one theorem in \cref{sect:IoIConsequencePf} with the \DEPSC (i.e. multiplicity one for Bessel functions), but then one must manually compute the \Asymps (below) for the higher-weight cases $\Lambda \ne 0$ in order to justify the \AnContC.

Now we begin to consider representations of these functions.
For any $w \in W^\text{rel}$, define $W_w$ to be the set of Weyl elements fixing every element of $Y_w$
\[ W_w = \set{w' \in W\setdiv w' y {w'}^{-1} = y \, \forall y \in Y_w}. \]

\begin{conj*}[Differential Equations and Power Series]
\hypertarget{conj:DEPS}{}Suppose the \IoIC and, for convenience, the \AnContC, then
\begin{enumerate}
\item $K_w(y,\mu,\delta)$ inherits $\ell-1$ partial differential equations from the action of the Casimir operators on $W_\sigma(yw x t,\mu,\delta)$.
The degree of each partial differential equation does not exceed $n!$ and the equations do not depend on $\delta$.

\item Let $v \in V_w := Y_w \cap K$ and define $\mathcal{K}_w(\mu,v)$ to be the space of smooth functions $f(y)$ on $y \in G_w, \sgn(y)=v$ satisfying those $\ell-1$ differential equations and
\[ f\paren{uyv \paren{wu'w^{-1}}} = \psi_I(uu') f(yv), \qquad u \in U(\R), y \in Y_w^+, u' \in \wbar{U}_w \]
then $\dim \mathcal{K}_w(\mu,v) = \abs{W/W_w}$.

\item There exists a Frobenius series of the form
\[ J_w^*(y,\mu) = p_{\rho+\mu}(y) \sum_{m \in \N_0^{\ell-1}} a_{w,m}(\mu) \prod_{i=1}^{\ell-1} y_{\hat{r}_i}^{m_i}, \qquad y \in Y_w, \]
with $a_{w,m}(\mu)$ independent of $\sgn(y)$, $a_{w,m}(\mu^{w'})=a_{w,m}(\mu)$ for all $w' \in W_w$, and $a_{w,0}(\mu)=1$, such that the restriction of $J_w^*(y,\mu)$ to $\sgn(y)=v$ is in $\mathcal{K}_w(\mu,v)$.

\item Provided $\mu_i-\mu_j \notin \Z$ for $i \ne j$, the set
\[ \set{J_w^*(y,\mu^{w'})\setdiv w' \in W/W_w} \]
spans $\mathcal{K}_w(\mu,v)$.

\item Define $J_w(y,\mu) = J_w^*(y,\mu)/\Lambda_w(\mu)$, where
\[ \Lambda_w(\mu^{w_l}) = \prod_{\substack{j<k\\ k^w < j^w}} (2\pi)^{\mu_j-\mu_k} \Gamma\paren{1+\mu_k-\mu_j}. \]
When $\mu_i - \mu_j \in \Z$ for one or more pairs $i \ne j$, we have
\[ \dim\Span\set{J_w(y,\mu^{w'}), J_w(y,\mu^{w_{(i\,j)}w'})}=1 \]
where $w_{(i\,j)}$ transposes $\mu_i$ and $\mu_j$.
The remaining solutions in that case are given by the analytic continuation of, e.g.,
\[ \frac{J_w(y,\mu)-J_w(y,\mu^{w_{(i\,j)}})}{\sin \pi(\mu_i-\mu_j)}, \]
provided $w_{(i\,j)} \notin W_w$.

\item The coefficients $a_{w,m}(\mu)$ are entire in $\mu$ except for possible poles at points where $\mu_i - \mu_j \in \Z$ for some $i \ne j$, and there exists some $C>0$ such that $a_{w,m}(\mu) \ll \frac{(3+\norm{\mu})^{C M}}{(M!)^2}$ where $M=\sum_{i=1}^{\ell-1} m_i$.
In particular, the series converges absolutely and uniformly on compact subsets in $y,\mu$.
\end{enumerate}
\end{conj*}
Outside of the exceptional cases $\mu_i - \mu_j \in \Z$, $i \ne j$, this is proved for the long element in \cref{sect:LEDiffEqs}; the bound on the long-element coefficients in that case follows from work of Hashizume \cite{Hashi}.
For later usage, we also define
\[ J_w^*(y,\mu,\delta) = \chi_\delta(y) J_w^*(y,\mu), \qquad J_w(y,\mu,\delta) = \chi_\delta(y) J_w(y,\mu). \]
By dropping the condition $\sgn(y)=v$, we can define $\mathcal{K}_w(\mu)$ as a space of functions on $G_w$, then the conjecture implies $\dim \mathcal{K}_w(\mu) = 2^\ell \abs{W/W_w}$ and it is spanned by these $J_w(y,\mu,\delta)$.

Take $\lambda_i=\lambda_i(\mu)$ to be the eigenvalue of $p_{\rho+\mu}$ under the Casimir operator $\Delta_i$ and let $H$ be the space of all linear combinations
\[ \sum_{i=2}^n a_i X_i (\Delta_i-\lambda_i) \]
with $a_i = a_i(\lambda_2,\ldots,\lambda_n) \in \C[y_1,\ldots,y_{n-1}]$ and $X_i$ an element of the commutative algebra
\[ \C[\Delta_2,\ldots,\Delta_n,E_{1,1},\ldots,E_{n-1,n-1}], \]
using the usual basis $E_{i,j}$ of the Lie algebra of $G$.
Then $\wtilde{K}_w(g,t,\Lambda,\mu,\delta,\sigma)$ as a function of $g \in G_{w_l}$ is killed by the operators $\Delta_i-\lambda_i$, so also by any operator in $H$.
We expect to find a basis of $\ell-1$ operators $\Delta^*_1,\ldots,\Delta^*_{\ell-1} \in H$ with the property that restricting to $y \in Y_w$ removes all derivatives with respect to the coordinates of $\wbar{Y}_w$ and $w U_w(\R) w^{-1}$; in other words, the operators $\Delta^*_i$ properly restrict to functions of $G_w$.
It would then follow that $K_w(g,\mu,\delta)$ is killed by each $\Delta^*_i$.

If one believes such a basis (i.e. minimal set of elements that generate this left ideal in the universal enveloping algebra) of operators exists, then the restrictions of these $\Delta^*_i$ to $y \in Y_w$ are composed of multiples of the coordinates of $y$ and the operators $y_i \partial_{y_i}$ and having $\ell-1$ independent operators is sufficient that solutions should be uniquely determined by their asymptotics near zero.
Since each $\Delta^*_i$ is invariant under the action of the Weyl group, this justifies parts 2 and 3 of the conjecture.

A trivial corollary of the conjecture is:
\begin{prop}[Uniqueness]
Assume the \DEPSC.
The function $K_w(y,\Lambda,\mu,\delta)$ is uniquely determined by its first-term asymptotics, bi-$U(\R)$-invariance and differential equations.
\end{prop}
Hence, with \cref{prop:IoIConsequences2}, we may define coefficients $C_w(\mu,\delta)$ by 
\begin{align}
\label{eq:KwPS}
	K_w(y,\mu,\delta) = \sum_{w' \in W/W_w} C_w(\mu^{w'},\delta^{w'}) J_w(y,\mu^{w'}, \delta^{w'}).
\end{align}
We also define $C_w^*(\mu,\delta) = C_w(\mu,\delta)/\Lambda_w(\mu)$.
Note:  The quotient by $W_w$ is on the right because $p_{\mu^{w_a w_b}}(y) = p_{\mu^{w_a}}(w_b y w_b^{-1}) = p_{\mu^{w_a}}(y)$ if $y\in Y_w$ and $w_b \in W_w$.

Informally, the Asymptotics Conjecture would state that \textit{dropping the $\mu$ integral and ignoring issues of convergence, replacing the Whittaker function in the definition of $K_w(y,\mu,\delta)$ with its first-term asymptotics as $Y_w \ni y \to 0$ yields the first-term asymptotics of $K_w(y,\mu,\delta)$.}

We prove the Asymptotics Conjecture and furthermore compute the $C_{w,w'}^*(\mu,\delta)$ using Shahidi's theory of local coefficients \cite[Chapter 5]{Shah01}:
\begin{thm*}[Asymptotics]
\hypertarget{thm:Asymps}{}Assume the \IoI, \AnCont and \DEPS Conjectures, then for all $w,\mu,\delta$,
\begin{enumerate}
\item $C_w(\mu,\delta)$ is meromorphic in $\mu$ and for all $w' \in W$, we have $C_w(\mu^{w'},\delta^{w'}) = C_w(\mu,\delta)$.

\item We have
\[ C_w^*(\mu^{w_l},\delta^{w_l}) = \prod_{\substack{1\le j < k \le n\\k^w < j^w}} (-1)^{\delta_k} \pi^{\frac{1}{2}+\mu_j-\mu_k} i^{m_{j,k}} \frac{\Gamma\paren{\frac{m_{j,k}-\mu_j+\mu_k}{2}}}{\Gamma\paren{\frac{1+m_{j,k}+\mu_j-\mu_k}{2}}}, \]
where $\set{0,1} \ni m_{j,k} \equiv \delta_j+\delta_k \pmod{2}$.
\item Define
\[ \sinmu(\mu,\delta) = \prod_{1 \le j < k \le n} \sin \frac{\pi}{2}\paren{\mu_j+\delta_j-\mu_k-\delta_k}, \]
then we have
\[ \sinmu(\mu,\delta) C_{w_l}(\mu^w,\delta^w) = (-\pi)^{\frac{n(n-1)}{2}} i^{-(n-1)(\delta_1+\ldots+\delta_n)}\sgn(w). \]
\end{enumerate}
\end{thm*}
This is proved in \cref{sect:Asymptotics}; we are using the \AnCont non-trivially to reduce to $\Lambda=0$ where we have holomorphy of the Whittaker functional equations, and the zeros of the spectral measure $\dspec\mu$ are needed here to cancel the poles of $C_w^*(\mu,\delta)$.
The $C_w^*(\mu,\delta)$ are (up to normalization) the local coefficients of Shahidi \cite[Chapter 5]{Shah01}.

It is a standard fact from complex analysis that for a function having a Mellin transform in addition to a power series (Frobenius series) expansion, the poles of the Mellin transform are simple and occur precisely at the locations corresponding to terms of the power series.
So there is a strong connection between power series and inverse Mellin transforms.
For hypergeometric functions, one can easily go in the opposite direction; given a series
\[ f(x) = \sum_{k=0}^\infty \frac{\Gamma(a_1+k)\cdots\Gamma(a_r+k)}{\Gamma(b_1+k)\cdots\Gamma(b_\ell+k)} \frac{(-x)^k}{k!}, \]
we have an inverse Mellin integral representation on a ``hook contour''
\[ f(x) = \int_{\mathcal{C}_1} \frac{\Gamma(a_1-s)\cdots\Gamma(a_r-s)}{\Gamma(b_1-s)\cdots\Gamma(b_\ell-s)} \Gamma(s) x^{-s} \frac{ds}{2\pi i}, \]
with $\mathcal{C}_1$ as in \cref{pic:HookContour} (the dots there are the poles of $\Gamma(s)$, which correspond to the terms of the power series), provided the poles of the numerator lie outside the contour (and the original series converged).
If the gamma factors in this integral representation have sufficient decay at infinity, we might hope to straighten the contour into something of the form $\mathcal{C}_2$ where the the contour passes to the left of the poles of the gamma factors $\Gamma(a_i-s)$ and the real part of the contour (i.e. $\set{\Re(z)\setdiv z \in \mathcal{C}_2}$) remains bounded; we call this a vertical contour.
Having an integral representation as an inverse Mellin transform over a vertical contour implies that $x^{-r} f(x)$ is bounded as $x \to \infty$ for some $r$, and conversely, a combination of differentiability and boundedness implies good decay for the Mellin transform, hence an integral representation over a vertical contour.

\begin{figure}[t!]
\centering
\begin{subfigure}[t]{3in}
\centering
\includegraphics[width=3in]{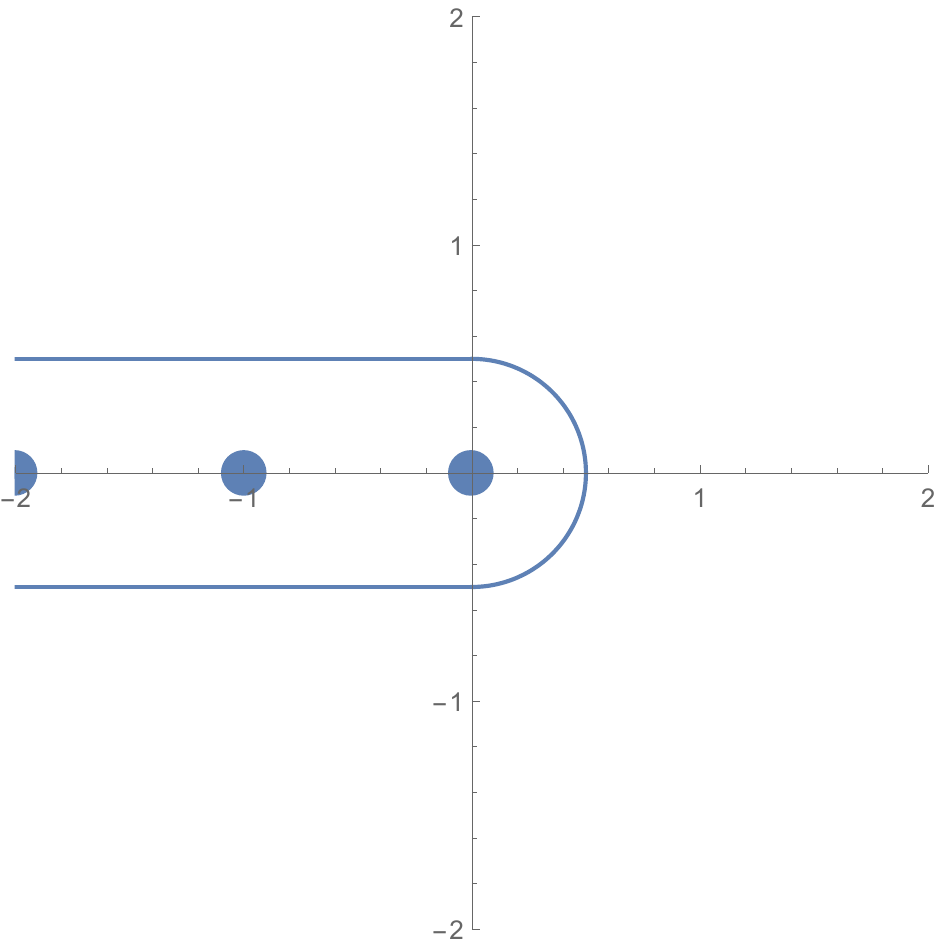}
\caption{The contour $\mathcal{C}_1$.}
\label{pic:HookContour}
\end{subfigure}
~
\begin{subfigure}[t]{3in}
\centering
\includegraphics[width=3in]{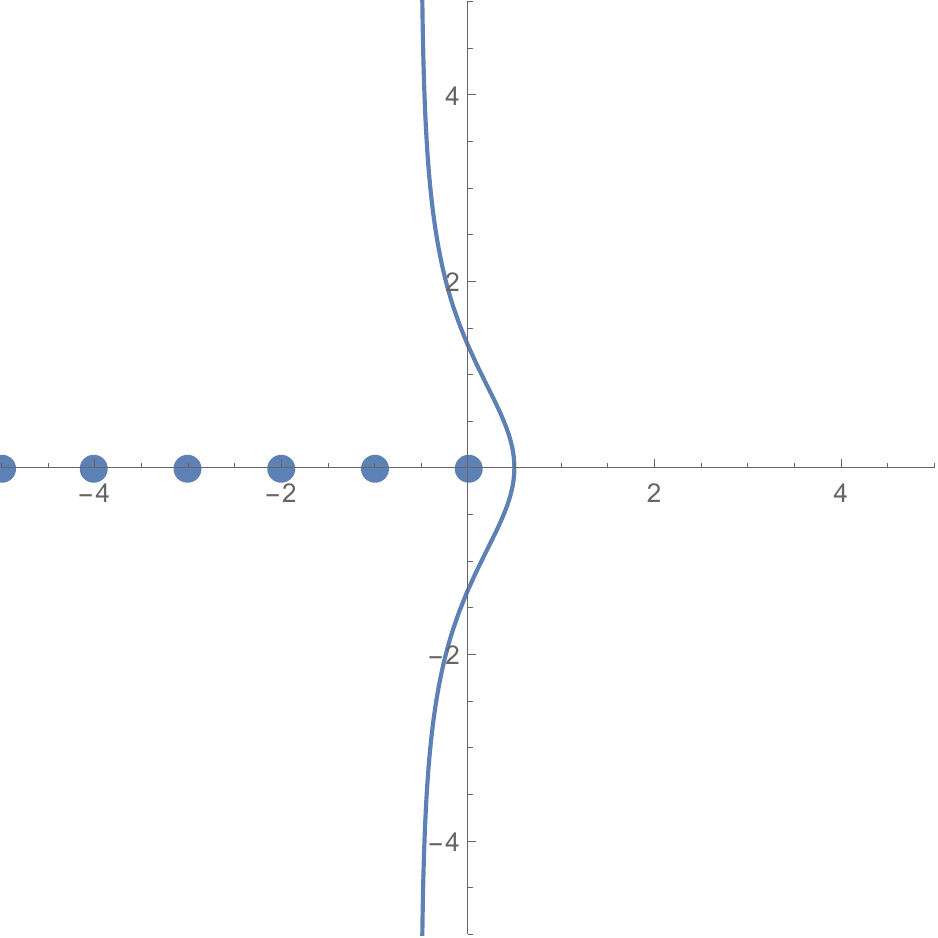}
\caption{The contour $\mathcal{C}_2$.}
\end{subfigure}
\caption{Contours for inverse Mellin transforms.}
\end{figure}

We allow a Mellin-Barnes integral (in this paper) to be a multi-dimensional inverse Mellin transform of a (finite) sum of quotients of gamma functions.
We require the contours to be vertical and follow the Barnes integral convention that the argument of every gamma function remains in $\C \setminus (-\infty,0]$ throughout the contour.
(The unbounded parts of each contour may well need to pass to the left of the imaginary axis for absolute convergence.)
If one believes the \DEPS, it should be the case that the coefficients of the power series are given by finite term recursions whose solutions are multi-dimensional hypergeometric functions.
In that case, we again have a trivial inverse Mellin integral representation over a hook contour, but the original definition of $K_w$ suggests it is polynomially bounded in $y$ and the differential equations lead us to believe it should have a Mellin-Barnes integral (using vertical contours).

\begin{conj*}[Mellin-Barnes Integrals]
\hypertarget{conj:MBInts}{}
Assume the \IoIC, and for convenience the \AnContC.
Fix any $v \in V_w$ and $\delta$, then as a function of $y \in Y^+_w$, $K_w(vy,\mu,\delta)$ has a Mellin-Barnes integral.
\end{conj*}
Note this conjecture is true in the case $w=w_l$ and $v=I$ by a theorem of Stade \cite{Stade01}, as $K_{w_l}(y,\mu,\delta)$ will turn out to be the Whittaker function for $y \in Y^+$; we prove this in \cref{sect:LEDiffEqs}.

For $y \in Y^+$, let $v\in V$ and define
\begin{align*}
	W_v =& \set{w' \in W \setdiv w'v=vw'}, \\
	K^v_{w_l}(y,\mu) =& \sum_{w \in W_v} \sgn(w) J_{w_l}(vy,\mu^w).
\end{align*}
For $y \in Y_w^+$ and $\eta \in \Z^n$, also define
\begin{align*}
	J^\dagger_w(y,\mu,\eta) =& p_{\rho+\mu}(y) \sum_{\substack{m \in \N_0^{\ell-1}\\ m_i \equiv \eta_{\hat{r}_i}\summod{2}}} a_{w,m}(\mu) \prod_{i=1}^{\ell-1} y_{\hat{r}_i}^{m_i}. \\
	K^\dagger_w(y,\mu,\eta,\delta) =& \sum_{w' \in W/W_w} C^*_w(\mu^{w'},\delta^{w'}) J^\dagger_w(y,\mu^{w'},\eta-\delta^{w'}).
\end{align*}

\begin{prop}
\label{prop:MBConsequences}
Assume the \MBIntsC.
Then $K^v_{w_l}(y,\mu)$ and \\ $K^\dagger_w(y,\mu,\eta,\delta)$ also have Mellin-Barnes integrals.
\end{prop}
Remarks:
\begin{enumerate}
\item If $\sgn(y_i)=+$, then we expect $J_{w_l}(y,\mu)$ to grow like $I_\nu(4\pi\sqrt{y_i})$ as $y_i \to \infty$, but if $\sgn(y_i)=\sgn(y_j)=-$, there is a diagonal term which likely grows exponentially along $y_i=y_j \to -\infty$.
So we need at least the sum over $W_v$.

\item For $y_1 < 0$ on $GL(2)$, $\SmallMatrix{-1\\&1}$ is not fixed by the long Weyl element and in that case we have $J_{w_l}(y,\mu) =\sqrt{\abs{y_1}} J_{2\mu_1}(4\pi\sqrt{\abs{y_1}})$, which has a Mellin-Barnes integral.

\item The group $W_v$ necessarily factors as product of two permutation groups -- the permutations of the positive sign coordinates and the permutations of the negative sign coordinates.

\item The dimension of the space of moderate-growth long-element Bessel functions of sign $v$ should be exactly $\abs{W/W_v}$.

\item We can certainly consider functions $K^v_w(y,\mu)$ for $w \ne w_l$, but the correct linear combinations are not obvious to the author.
\end{enumerate}

From the analysis of \cref{sect:LEDiffEqs} and an integral representation of Ishii \cite[Theorem 3]{Ishii} for the Frobenius series solutions of the Whittaker differential equations, we will show
\begin{prop}
\label{prop:IshiiRepns}
Assume the \DEPSC.
Define $Z^+_\alpha(x) = I_\alpha(x)$ and $Z^-_\alpha(x) = J_\alpha(x)$ and make the $n$- and $v$-dependence explicit by writing
\[ J_{n,w_l}^\sharp(v, y, \mu) = p_{-\rho}(y) J_{n,w_l}^*(vy,\mu).\]
for $v \in V$ and $y \in Y^+$.
Then we have the integral representation
\begin{align*}
	& J_{n,w_l}^\sharp(v, y, \mu) = \\
	& \prod_{j=1}^{n-1} y_j^{\frac{n-2j}{2(n-2)}(\mu_1+\mu_n)} \frac{1}{(2\pi i)^{n-2}} \int_{\abs{u_1}=1} \cdots \int_{\abs{u_{n-2}}=1} \prod_{j=1}^{n-1} Z^{v_j}_{\mu_1-\mu_n}\paren{4\pi\sqrt{y_j(1+u_{j-1})(1+1/u_j)}} \\
	& \qquad J_{n-2,w_l}^\sharp\paren{v', \paren{y_2 u_1/u_2,\ldots,y_{n-2} u_{n-3}/u_{n-2}},\mu'} \prod_{j=1}^{n-2} u_j^{-\frac{n}{2(n-2)} (\mu_1+\mu_n)} \frac{du_j}{u_j},
\end{align*}
where $v'=(v_2,\ldots,v_{n-2})$, $u_0=1/u_{n-1}=0$ and  $\mu' \in \C^{n-2}$ defined by $\mu'_j = \mu_{j+1}+\frac{\mu_1+\mu_n}{n-2}$.
\end{prop}
One should interpret the integral by analytic continuation from the region having $\mu \in \R^n$ with $\mu_i - \mu_j > 0$ for all $i < j$.

Taking the appropriate sum over $W_v$ immediately gives an inductive integral representation for $K^v_{w_l}(y,\mu)$:
Define
\begin{align*}
	K^{v \sharp}_{w_l}(y,\mu) :=& \sum_{w \in W_v} \sgn(w) J^\sharp_{w_l}(vy,\mu^w),
\end{align*}
then
\begin{align*}
	K_{n,w_l}^{v \sharp}(v, y, \mu) =& \sum_{w \in W_v/W_{v'}} \sgn(w) \prod_{j=1}^{n-1} y_j^{\frac{n-2j}{2(n-2)}(\mu^w_1+\mu^w_n)} \frac{1}{(2\pi i)^{n-2}} \int_{\abs{u_1}=1} \cdots \int_{\abs{u_{n-2}}=1} \\
	& \qquad \prod_{j=1}^{n-1} Z^{v_j}_{\mu^w_1-\mu^w_n}\paren{4\pi\sqrt{\abs{y_j}(1+u_j)(1+1/u_{j-1})}} \\
	& \qquad K_{n-2,w_l}^{v' \sharp}\paren{v', \paren{y_2 u_2/u_1,\ldots,y_{n-2} u_{n-2}/u_{n-3}},(\mu^w)'} \prod_{j=1}^{n-2} u_j^{\frac{n}{2(n-2)} (\mu^w_1+\mu^w_n)} \frac{du_j}{u_j},
\end{align*}
where we've embedded $W_{v'} \subset W_{n-2}$ in $W_n$ by acting on the coordinates of $\mu'$ (equivalently, by taking the permutations that fix $\mu_1$ and $\mu_n$).

These integral representations suffer from exponential growth in the integrands unless all $\abs{y_j}$ are small (compared to $\min_{i < j} \abs{\mu_i - \mu_j}$), so the question is now can we find a good representation (along the lines of \cite{SubConv}) when any of the $y_i$ are not small.
One might consider contour shifting in the $u$ variables, but also we can simply guess and check using Ishii's recurrence relation \eqref{eq:IshiiRecRel}.

\section{Applications}
\label{sect:Apps}
Consider the spectral Kuznetsov formula as in \eqref{eq:SpecKuz}.
We can get a bound on the Kloosterman sum terms by contour shifting just in the Mellin expansion of the first Whittaker function of $H_w(f,y)$ and the $\mu$ integrals (without using the Bessel functions).
This bound will certainly be weaker than even simply Mellin expanding the Bessel functions since we lose by Mellin expanding the Whittaker function and we don't save anything when we use absolute convergence on the $\mu$, $u$ and $g$ integrals in $H_w(f,y)$.
This weaker bound plus the compatibility relation might be good enough for nontrivial results except on the long Weyl element term, just by comparing the distance needed to shift the leading asymptotics of the Whittaker function.
(The shorter Weyl elements come with compatibility relations which affect the region of absolute convergence of the corresponding Kloosterman zeta function and so reduce the distance needed in the contour shifting.)
For the long Weyl element term, we might hope to prove the \IoIC at just that Weyl element, then we have everything but a good integral representation, and one would hope that follows by some contour shifting and combinatorics on \cref{prop:IshiiRepns}; this would likely be the largest benefit at the minimal effort.

In this section, we explore some possible applications of the above conjectures (at all Weyl elements), and what additional results would be required.

\subsection{Arithmetically-weighted Weyl laws}
Define $\tau_w > 0$ so that the Kloosterman zeta
\[ Z_w(a,b,\mu) = \sum_c S_w(a,b,c) p_{\rho+\mu}(c) \]
converges absolutely on the region $\Re(\mu_i-\mu_{i+1}) > \tau_w$ for all $1 \le i \le n-1$.

We have $\tau_w \le 1$ by D\c{a}browski-Reeder \cite[Theorem 0.3.i]{DabReeder} and $\tau_{w_l} \le \frac{7}{8}$ by Miao \cite{Miao}.
One might conjecture that $\tau_w = \frac{1}{2}$ is sufficient (as this is true for $n \le 3$, see the references in \cite{SpectralKuz}), but it is not clear that this should hold for $n \ge 4$.
Furthermore, we would expect that $\tau_w < \tau_{w_l}$ for $w \ne w_l$, due to the compatibility relation, though it might be necessary to take a finer-grained $\tau_{w,i}$, $i=1,\ldots,n-1$.
Some specific bounds for the Voronoi Weyl element can be found in \cite{Friedberg01} and for $n=4$ in Huang's appendix to \cite{GSW}.

Write $\dspec\mu = \specmu(\mu) d\mu$ where $d\mu=d\mu_1 \cdot d\mu_{n-1}$.
Note that we expect
\[ \abs{\specmu(\mu)} \asymp \prod_{i<j} \abs{\mu_i-\mu_j} \]
away from the points $\mu_i=\mu_j$, $i \ne j$.

If we construct a test function by convolving the characteristic function of a set against a Gaussian and apply contour shifting in the Mellin expansion of the Bessel functions to reach the region of absolute convergence of the Kloosterman zeta functions, we get an arithmetically-weighted Weyl law of the shape:
\begin{thm}
\label{thm:ArithWeyl}
Assume the \MBIntsC.
For $\Omega \subset i\mathfrak{a}^*_0(\Lambda)$ take $T = \sup_{\mu \in \Omega} \norm{\mu}$ and suppose $T \ge n^2$.
Take $E(\mu,\delta)$ as in \cref{sect:ArithWeyl} and assume that the relevant contour shifts converge, then
\begin{align*}
	\int\limits_{\substack{\mathcal{B}(\Lambda,\delta)\\ \mu_\xi\in \Omega}} \frac{1}{L(1,\AdSq \xi)} d_\text{H}\xi = \int_{\Omega} \dspec\mu + \BigO{\int_{\Omega\cup\partial \Omega^*} E(\mu,\delta) \abs{d\mu}+\int_{\partial \Omega^*} \abs{\specmu(\mu)} \abs{d\mu}},
\end{align*}
where $\partial \Omega^* = \partial \Omega+B(0,(\log T)^{-1/2+\epsilon})$ with $B(0,R)$ the ball of radius $R$ centered at $\mu=0$.
\end{thm}
(Note that we are expecting the Bessel functions to demonstrate some decay in $\norm{\mu}$; otherwise, the first error term would dominate.)

Nelson and Jana have also given a conductor-ordered arithmetically-weighted Weyl law \cite[Theorem 3]{JanaNelson}; it would be interesting to replicate this result using the Kuznetsov formula directly.

\subsection{Bessel expansions}
A paper of Baruch and Mao \cite{BaruchMao01} defines Bessel-like distributions in terms of the expansion of test functions into a basis of the Whittaker model of a representation.
We wish to connect these back to the Bessel functions studied here, using the methods of \cite{ArithKuzII}.

\begin{conj*}[Bessel Expansion]
\hypertarget{conj:BesselExpand}{}
Suppose $F:G \to \C$ is smooth and compactly supported in the Iwasawa coordinates, then for $g \in G$, the Bessel-like distribution
\[ \what{F}(g,\mu,\delta) := \sum_\sigma (\dim\sigma) \Tr\paren{W_\sigma(g,\mu,\delta) \int_G F(g') \trans{\wbar{W_\sigma(g',-\wbar{\mu},\delta)}} dg'} \]
is given by the integral
\[ \what{F}(g,\mu,\delta) = \frac{1}{\abs{V}} \int_{G_{w_l}} F(g'g) \wbar{K_{w_l}(g',-\wbar{\mu},\delta)} dg'. \]
\end{conj*}
Here we are using the term ``Bessel-like distribution'' by analogy with \cite[Appendix 4]{BaruchMao01}.
One should compare this to \cite[Theorem 2.5]{Baruch01}; if one could connect the Bessel-like distribution as it appears here to the Bessel distribution there, then that theorem would show $\what{F}$ is given by integration against some element of $\mathcal{K}_w(\mu)$, while the result above specifies which Bessel function appears.

Fundamentally, the \BesselExpandC follows from two facts, the first of which is that the Fourier transform of a zonal spherical function is a product of two Whittaker functions.
This is strongly tied to Wallach's Whittaker inversion theorem \cite{Wallach} and Comtat \cite{Comtat} has given a very general method for recovering this statement from Wallach's theorem.
Even though this method should certainly apply here, we consider the method of \cite{ArithKuzII}, which relies on a conjecture, as it is very nearly a special case of the \StrongIoIC and it is likely of independent interest:
Informally, the \JWDCC states \textit{for $\mu$ in a tube domain containing $\Re(\mu)=0$ and all $t$, the integral \eqref{eq:JWDef} defining the Jacquet-Whittaker function $W_\sigma(t,\mu,\delta)$ can be rearranged through change of variables, integration by parts and a smooth, dyadic partition of unity so that it converges absolutely.}
As with the \StrongIoIC, the full statement is rather technical, so we postpone it to \cref{sect:StrongIoI}.

The second required fact is a sort of Bessel inversion or Plancherel formula:
For $F(\mu)=F(\Lambda+\wtilde{\mu})$ Schwartz-class and holomorphic in $\wtilde{\mu}$ on a tube domain containing $\Re(\wtilde{\mu})=0$, we define
\[ \wtilde{F}_{\delta,\Lambda}(y) = \int_{i\mathfrak{a}^*_0(\Lambda)} F(\mu) K_{w_l}(y,\mu,\delta) \dspec\mu. \]
If also $F'(\mu)$ satisfies these properties for $\Lambda',\delta'$, define
\[ \mathcal{I}_{\delta,\Lambda,\delta',\Lambda'}(F,F') := \int_Y \wtilde{F}_{\delta,\Lambda}(y) \wbar{\wtilde{F'}_{\delta',\Lambda'}(y)} dy. \]
Then by analytic continuation, it is sufficient to show that $\mathcal{I}_{\delta,\Lambda,\delta',0}(F,F')$ is zero unless $\Lambda=0$ and $\delta=\delta'$, in which case
\begin{align}
\label{eq:InvBesPlancherel}
	\mathcal{I}_{\delta,0,\delta,0}(F,F') = \abs{V} \int_{i\mathfrak{a}^*_0(0)} F(\mu) \wbar{F'(-\wbar{\mu})} \dspec\mu.
\end{align}

To obtain the actual expansion of smooth, compactly supported functions on $Y_w$ into Bessel functions, we also require
\begin{conj*}[Test Functions]
\hypertarget{conj:TestFuns}{}
For $w \in W^\text{rel}$, let $f:Y_w \to \C$ be smooth and compactly supported.
Define $F:G \to \C$ to be zero outside $U(\R) Y \trans{U(\R)} w$ and
\begin{align}
\label{eq:TestFDef}
	F(x y y' \trans{z} w x') =& \psi_I(xx') f(y) u(y',z,x'),
\end{align}
$x\in U(\R), x' \in \wbar{U}_w(\R), y \in Y_w, y' \in \wbar{Y}_w, z \in U_{w^{-1}}(\R)$, where
\[ u:\wbar{Y}_w \times U_{w^{-1}}(\R) \times \wbar{U}_w(\R) \to \C \]
is smooth and compactly supported with
\[ \int_{\wbar{U}_w(\R)} u(I,I,x') dx' = 1. \]
Then $F$ is smooth and compactly supported in the Iwasawa coordinates when the support of $u$ is sufficiently small.
\end{conj*}
Note that such an $F$ has
\[ T_w(F)(y) = f(y), \qquad y \in Y_w. \]

If the \BesselExpand and \TestFuns Conjectures hold, it follows that for such $f$ and $F$, we have
\[ f(y) = \frac{1}{\abs{V}} \sum_{\delta,\Lambda} \int_{i\mathfrak{a}^*_0(\Lambda)} K_w(y,\mu,\delta) \int_G F(g) \wbar{K_{w_l}(g,-\wbar{\mu},\delta)} dg \, \dspec\mu. \]

Aside: One might be tempted to define a Bessel function on $G$ by $K(yw,\mu,\delta) = K_w(y,\mu,\delta)$, but applying these decompositions to the explicit representations on $GL(3)$, e.g. in \cite{ArithKuzI}, we see that such a function is not continuous (and the difficulty cannot be explained by any errors in the constants).

\subsection{Arithmetic Kuznetsov trace formulas}
Taking such an $F$ for the kernel of a Poincar\'e series, we get the arithmetic Kuznetsov formulas (see \cite[Section 8]{ArithKuzII}):
\begin{thm}
\label{thm:ArithKuz}
Assume the Bessel Expansion and \TestFuns Conjectures and suppose $\psi_a, \psi_b$ are non-degenerate characters.
Then for $f$ and $F$ as above,
\begin{align*}
	&\sum_{U_{w'}(\R) \supset U_w(\R)} \sum_c \frac{S_{w'}(a,b,c)}{p_{-\rho}(c)} T_{w'}(F)(acw' b^{-1} {w'}^{-1}) = \\
	& \qquad \frac{1}{\abs{V}} \sum_{\delta,\Lambda} \int_{\mathcal{B}(\Lambda,\delta)} \frac{\lambda_\Xi(b) \wbar{\lambda_\Xi(a)}}{L(1,\AdSq \Xi)} \int_G F(g) \wbar{K_{w_l}(g,-\wbar{\mu_\Xi},\delta)} dg \,d_\text{H}\Xi,
\end{align*}
and
\[ T_w(F)(y) = f(y), \qquad y \in Y_w. \]
\end{thm}

When $w=w_l$ only the long element appears in the Weyl element sum, but for the other Weyl elements, we have a conjecture:
\begin{conj*}[Isolation of Weyl Cells]
For $f$ and $F$ as above, if the support of $u(y',z,x')$ is sufficiently small compared to $a,b$, then
\[ T_{w'}(F)(acw' b^{-1} {w'}^{-1}) = \delta_{w'=w} f(acw b^{-1} w^{-1}). \]
\end{conj*}
The conjecture implies that individual Weyl elements can be isolated from the original Poincar\'e series, giving a Kuznetsov-type formula on just the Weyl cell:
\begin{align*}
	& \sum_c \frac{S_w(a,b,c)}{p_{-\rho}(c)} f(acw b^{-1} w^{-1}) = \\
	& \qquad \frac{1}{\abs{V}} \sum_{\delta,\Lambda} \int_{\mathcal{B}(\Lambda,\delta)} \frac{\lambda_\Xi(b) \wbar{\lambda_\Xi(a)}}{L(1,\AdSq \Xi)} \int_G F(g) \wbar{K_{w_l}(g,-\wbar{\mu_\Xi},\delta)} dg \,d_\text{H}\Xi,
\end{align*}

\subsection{Other applications}
Though \cite{BlomerDensity} encountered some unexpected difficulties with square-root cancellation in the level aspect, one might hope that \cite[Theorem 4]{ArithKuzII} would continue to hold:
\begin{conj}
For $f:Y \to \R$ smooth and compactly supported,
\begin{align*}
	\sum_c \frac{S_{w_l}(a,b,c)}{\abs{c_1 c_2 \cdots c_{n-1}}} f(yc) &\ll_{a,b,f,\epsilon} \prod_{i=1}^{n-1} y_i^{\theta+\epsilon}+\sum_{i=1}^{n-1} y_i^{-1-\epsilon},
\end{align*}
where $\theta=\theta(n) < \frac{1}{2}$ is the Ramanujan-Selberg parameter for $GL(n)$.
\end{conj}
In \cite{ArithKuzII} this was proved for $n=3$ using bounds for the Mellin transform of the Bessel functions.

In \cite{NelsonSubConv}, Nelson seems to have proven a very general subconvexity bound for $L$-functions of $GL(n)$ Maass forms in ``generic position'', i.e. when all of $\abs{\mu_i-\mu_j} \asymp T$ and $\abs{\mu_i} \asymp T$, using so-called ``soft'' methods.
The earlier paper \cite{SubConv}, took a direct approach to subconvexity using the explicit Kuznetsov formula and Bessel functions for $GL(3)$, and we expect that such methods will be useful for mathematicians studying subconvexity and moments of $L$-functions on $GL(n)$, provided the Bessel functions and Kloosterman sums are sufficiently well-understood.
The first challenges faced by such a researcher will likely be:
\begin{enumerate}
\item Optimal (square-root cancellation?) bounds for the Kloosterman sums.
There has be some progress on this front (as well as some negative results), e.g. in \cite{Miao} and \cite{BlomerMan}.

\item Proving a sharp cutoff for the sum of Kloosterman sums.
That is, given a reasonably nice test function $f$, we expect $H_w(f, acwb^{-1}w^{-1})$ (recall \eqref{eq:HwDef}) to decay rapidly for $c_i$ large compared to $a,b$; the goal should be to identify the precise boundary of the region where the rapid decay occurs.
This is Lemmas 8 and 9 in \cite{SubConv}, and similarly, Lemmas 3 and 4 in \cite{GPSSubConv}.

\item Computing the Fourier transform of the Bessel functions.
After applying the Kuznetsov formula to a moment or subconvexity problem, a typical next step involves Poisson or Voronoi summation, and hence the Fourier transform of the $H_w(f,y)$ function.
This is Lemmas 10 and 11 in \cite{SubConv} and Lemmas 5 and 6 in \cite{GPSSubConv}.
Sharma \cite{SharmaSubConv} has apparently improved this by applying stationary phase to the double-Bessel integral representations of the $GL(3)$ Bessel function.
The author wonders if directly applying the known Fourier transform of the $GL(2)$ Bessel functions in the double-Bessel integral representations would result in a better bound and/or a cleaner argument, and might offer some hope of generalization to $GL(n)$.

\item Diophantine analysis of the collapsed Kloosterman sums.
The Poisson summation step of \cite{SubConv} caused the Kloosterman sums to mostly collapse to congruences; one then has to have some understanding of the set of points satisfying those relations.
\end{enumerate}

Lastly, It would be very interesting to write out the Kuznetsov formulas in the case where one or both of the characters are degenerate (having one or more indices equal to zero).
Of course, we should be able to build the $GL(n,\R)$ spectral Kuznetsov formulas at a degenerate character from the Kuznetsov formulas of $GL(m,\R)$ with $m < n$, but the arithmetic Kuznetsov formulas have the possibility of connecting sums of degenerate $GL(n,\R)$ Kloosterman sums to sums of $GL(m,\R)$ Kloosterman sums in unexpected ways.

\subsection{Other integral representations of the Bessel functions}
\subsubsection{An explicit direct integral}
For any proposed integral representation of the Bessel functions, the conjectures reduce the question to showing it satisfies the appropriate differential equations and computing its first-term asymptotics.
One might then hope for a direct integral representation; for example, that the Jacquet-type integral of the function defined on $G_{w_l}$ by
\[ I^*_{\mu,\delta}\paren{x y \trans{x'}} = \psi_I(x) p_{\rho+\mu}(y) \chi_\delta(y). \]
would give the Bessel function up to a constant; i.e. $K_w(y,\mu,\delta) = C(\mu,\delta) \mathcal{I}_w(y,\mu,\delta)$ where
\begin{align*}
	\mathcal{I}_w(y,\mu,\delta) =& \int_{\wbar{U}_w(\R)} I^*_{\mu,\delta}(y w u) \wbar{\psi_I(u)} du.
\end{align*}
Unfortunately, this integral is far from absolutely convergent for any value of $\mu$.
The author is unaware of any direct, absolutely convergent integral for the Bessel functions.

One could formulate a version of the \StrongIoIC for this integral and (with a bit more difficulty) argue that it would satisfy the same differential equations as the Bessel function (for the same reasons), leaving us to check the first-term asymptotics as $y \to 0$, which naively appear to be independent of $\sgn(y)$ (except for the factor $\chi_\delta(y)$), so it is not unreasonable.
The advantage to applying this integral representation would be significantly fewer extraneous integrals and a notable lack of spare Whittaker functions.

\subsubsection{The Voronoi Bessel function}
On the other hand, Zhi Qi \cite{ZhiQi} has evaluated the iterated integral $\mathcal{I}_w(y,\mu,\delta)$ for the Weyl elements $w=w_{1,n-1},w_{n-1,1}$.
On $GL(3)$, the Bessel functions $K_w(y,\mu,0)$ attached to $w=w_{1,2}$ or $w=w_{2,1}$ coincide with the Bessel functions appearing in the Voronoi formula \cite{MillerSchmid02} and these are the functions Zhi Qi has investigated on $GL(n)$; compare the functions $K_{w_4}(y,\mu)$ and $K_{w_5}(y,\mu)$ in \cite{SpectralKuz} to the functions $g^\pm(y)$ in \cite[Lemma 6]{Bl02}.

\begin{conj}
Up to normalization, the function $K_{w_{n-1,1}}(y,\mu,\delta)$ is given by the Mellin-Barnes integral
\[ \abs{y_{n-1}}^{\frac{n-1}{2}} \sum_{\ell\in\set{0,1}} \paren{i\sgn(y_{n-1})}^\ell \int_{-i\infty}^{i\infty} \abs{\pi^n y_{n-1}}^{-s} \prod_{j=1}^n \frac{\Gamma\paren{\frac{\ell_j+s-\mu_j}{2}}}{\Gamma\paren{\frac{1+\ell_j-s+\mu_j}{2}}} \frac{ds}{2\pi i}, \]
where $\set{0,1} \ni \ell_j \equiv \ell+\delta_j \pmod{2}$.
\end{conj}
Of course, $K_{w_{1,n-1}}(y,\mu,\delta)$ has a similar representation by duality (i.e. \cref{prop:IoIConsequences2}).

Under the \StrongIoIC, we have
\begin{align*}
	& K_{w_{n-1,1}}(y,\mu,0) W_1(I,\mu,0) = \\
	& \int_{\R^{n-1}} \e{-x_1+\frac{y_{n-1} x_{n-1}}{(1+x_{n-1}^2)\prod_{i=1}^{n-2}\sqrt{1+x_i^2}}-\sum_{i=1}^{n-2} \frac{x_i x_{i+1}}{\sqrt{1+x_i^2}}} \\
	& \qquad W_1\paren{\paren{\frac{\sqrt{1+x_2^2}}{\sqrt{1+x_1^2}},\ldots,\frac{\sqrt{1+x_{n-1}^2}}{\sqrt{1+x_{n-2}^2}},\frac{1}{(1+x_{n-1}^2)\prod_{i=1}^{n-2}\sqrt{1+x_i^2}}},\mu,0} dx
\end{align*}
(for the correct interpretation of the conditionally-convergent Riemann integral), so a direct comparison to \cite{ZhiQi} is difficult, but the $GL(3)$ computations in \cite{GoldLi} seem to (vaguely) lend support for the general case.

\subsubsection{Zagier's construction}
Li's generalization to $GL(n)$ \cite{Gold01} of Zagier's $GL(2)$ construction follows the spherical inversion route of constructing the Kuznetsov formula, so the weight functions on the Kloosterman sum side are of the form
\[ H_w(f, y) = \int_{\wbar{U}_w(\R)} \int_{Y^+} \int_{U(\R)} F\paren{t^{-1} x^{-1} y w u t} \psi_I(x) dx\,(\det t) dt \,\wbar{\psi_I(u)}du, \]
where $F(g)$ is given by an inverse spherical transform
\[ F(g) = \int_{i\mathfrak{a}^*_0(0)} f(\mu) h_\mu(g) d_\text{SI}\mu. \]
One might consider continuing to follow Zagier's path to find an integral representation of the Bessel function from this perspective; if the \IoIC holds, the constructions may be used interchangeably with the same Bessel functions appearing in each, as per the discussion in \cref{sect:SpecKuz}.

To the author's knowledge, Zagier's construction was never published, but it can be found in Xiaoqing Li's Kloostermania notes \cite{LiKloostermania}.
In the $GL(2)$ long-element case, Zagier substitutes on $u$ to arrive at
\[ H_w(f, y) = \int_{\wbar{U}_w(\R)} \int_{Y^+} \int_{U(\R)} F\paren{t^{-1} x^{-1} y w u x t} dx\,(\det t) dt \,\wbar{\psi_I(u)}du, \]
and this can be done for most Weyl elements on $GL(n)$ (but obviously not all; consider $w=I$) by carefully choosing $v \in U(\R)$ and $v' \in U_w(\R)$ such that $x^{-1} yw v' (yw)^{-1} = v^{-1}$ and $v' u v \in \wbar{U}_w(\R)$.

Zagier's next move substitutes on $xt$ to diagonalize $ywu$ over $\C$; on $GL(n)$, this can be facilitated by a change of variables replacing $n-1$ of the coordinates of $u$ with eigenvalues $\lambda$ of $ywu$.
(There is also perhaps some benefit to viewing $F(g)$ as a function of the characteristic polynomial of $g\trans{g}$.)
Then Zagier faces cases based on the types (complex or real) of the eigenvalues.
When the eigenvalues are real, we can substitute $xt \mapsto tx$, which removes $t$ from the argument of $F$ so the $t$ integral can be computed directly, and this turns out to be independent of $x$ and $u$ (provided one brings the signs $x \mapsto vxv$, $v \in V$ with $t$).

In the complex eigenvalue case, things are somewhat more complicated.
Zagier's substitution is now a complex fractional linear transformation; multiplication by $i$ reverses the coordinates of $x+it$ and the author has no intuition as to how to directly implement it on $GL(n)$.
However, one might instead consider to block diagonalize $ywu$ so that in the complex case we have elements of $SO(2)$ on the diagonal.
In this way, on $GL(2)$, we can take the Cartan coordinates for $xt \mapsto h_1(\theta) t$ and apply the commutativity of $SO(2)$ to remove the $\theta$ integral, which can be evaluated and again turns out to be independent of $t$ and $u$.
On $GL(n)$, a combination of converting certain $x$ coordinates to polar and conjugating certain coordinates of $t$ across should enjoy similar success.

At this point, having reduced the number of integrals, Zagier uses properties of the $GL(2)$ spherical transform whose generalization to $GL(n)$ are not at all obvious.
In the real eigenvalue case, Zagier is simply writing the spherical transform as the Mellin transform of the Harish-Chandra transform, but in the complex eigenvalue case, the picture is much less clear.

In theory, one could push the argument through simply because the dimension of $U(\R)$ is growing faster than the degree of the aforementioned characteristic polynomial, but the argument is far from ready for $GL(n)$, and here is where this story ends, at least for now.

\subsubsection{\texorpdfstring{Via $\what{F}(g,\mu,\delta)$}{Via hat-F(g,mu,delta)}}
One might consider to replace the Whittaker function $W_\sigma(g,\mu,\delta)$ in $\what{F}(g,\mu,\delta)$ as in the \BesselExpandC with $W_\sigma(g,\mu',\delta)$, then shift $-\mu$ and $\mu'$ into the region of absolute convergence of the Jacquet integrals.
At that point, one can open the Jacquet integrals and push the sum over $\sigma$ inside to produce a somewhat complicated, but explicit, integral.
One might hope to produce an integral representation of the long-element Bessel function by applying the \BesselExpandC; compare \cite[Theorem 1.7]{Baruch04} at the finite places.
Clearly, this would only apply for the long Weyl element.

\section{Background}

\subsection{Measures}
\label{sect:Measures}
We decompose Haar measure on $G$ by the Iwasawa decomposition $dg=dx \, dy \, dk$, and we take $dk$ to be Haar probability measure.

On $Y^+$, we apply the measure
\[ dy := p_{-2\rho}(y) \prod_{i=1}^{n-1} \frac{dy_i}{y_i} = \prod_{i=1}^{n-1} \frac{dy_i}{y_i^{1+i(n-i)}}, \]
which is the measure inherited as a subspace of $G$ (not Haar measure).
For the measure on $Y$, we apply counting measure to $V$, so if $d\bar{k}$ is probability measure on $V\backslash K$, we have
\[ \int_{Y^+} \int_K f(yk) dy \, dk = \frac{1}{\abs{V}} \int_Y \int_{V\backslash K} f(y\bar{k}) dy \, d\bar{k}. \]
Haar measure on $U(\R)$ is simply Lebesgue measure on $\R^{\frac{n(n-1)}{2}}$:
\[ dx = \prod_{i=1}^{n-1} \prod_{j=i+1}^n dx_{i,j}. \]

We will several times use the fact that the determinant of the Jacobian for the change of variables $x \mapsto yxy^{-1}$ on $\wbar{U}_w(\R)$ with $y \in Y$ is $p_{\rho-\rho^w}(y)$.
To see this, note that the coordinate $x_{i,j}$, $j>i$ is nonzero in $\wbar{U}_w(\R)$ if $i^w>j^w$, the change of variables sends $x_{i,j} \mapsto y_i\cdots y_{j-1} x_{i,j}$, and the power of $y_k$ in $p_{\rho-\rho^w}(y)$ is
\[ \sum_{q=1}^k \paren{q^w-q} = \sum_{1 \le i \le k} \sum_{\substack{1 \le j \le n\\i^w \ge j^w}} 1-\sum_{\substack{1 \le i,j \le k\\i^w \ge j^w}} 1 = \sum_{\substack{1 \le i \le k < j \le n\\i^w \ge j^w}} 1. \]

For $i=1,\ldots,n-1$, define $h_i(\theta) \in K$ by
\[ h_i(\theta) = \Matrix{\cos\theta && -\sin\theta\\&I_{i-1}\\ \sin\theta&&\cos\theta\\&&&I_{n-1-i}} = w_{(2\;i+1)} h_1(\theta) w_{(2\;i+1)}. \]
Then one possible arrangement of the hyper-spherical coordinates of $SO(n)$ is
\[ h_1(\theta_{n,1}) h_2(\theta_{n,2})\cdots h_{n-1}(\theta_{n,n-1}) h_1(\theta_{n-1,1})\cdots h_{n-2}(\theta_{n-1,n-2})\cdots h_1(\theta_{2,1}), \]
with $\theta_{i,1} \in [0,2\pi)$, $\theta_{i,j} \in [0,\pi]$, $j>1$.
Haar probability measure on $SO(n) \subset K$ is
\[ d_{SO(n)}k = \prod_{i=2}^n \frac{\Gamma\paren{\frac{i}{2}}}{2 \pi^{\frac{n+2}{4}}} \prod_{j=1}^{i-1} \sin^{j-1} \theta_{i,j} \, d\theta_{i,j}, \]
while $K/SO(n) = \set{SO(n), \diag(1,\ldots,1,-1) SO(n)}$, so Haar probability measure on $K$ is $dk = \frac{1}{2} d_{SO(n)}k$ if we assign counting measure on $K/SO(n)$.

\subsection{The power function}
Set
\begin{align*}
	M_n=&\set{\mu \in \C^n \setdiv \sum_{i=1}^n \mu_i = 0}, & M_n^0=&\set{\mu\in M_n\setdiv \Re(\mu)=0}.
\end{align*}
For elements $\mu$ of $M_n$, we define the linear transformation $\mu \mapsto \what{\mu} \in \C^{n-1}$ by
\[ \what{\mu}_i = \sum_{j=1}^i \mu_j, \]
and we define a coordinate-wise (as in \eqref{eq:YCoords}) power function
\[ \hat{p}_s(y) = \prod_{i=1}^{n-1} \abs{y_i}^{s_i}, \]
then we have $p_\mu(y) = \hat{p}_{\hat{\mu}}(y)$.
In particular, $\what{\rho}_i = \frac{i(n-i)}{2}$, and we apply the conventions $\what{\mu}^w = \what{\mu^w}$, $\what{\mu}^w_i = (\what{\mu}^w)_i$.
We apply the same construction to $\delta \mapsto \hat{\delta}$.

\subsection{Representations of $K$}
\label{sect:KRepns}
Throughout the paper, $\sigma$ will always refer to an irreducible, unitary representation $\sigma:K \to GL(n,\C)$.
Note that $\sigma|_{O(m)}$, $2 \le m < n$ is isomorphic to a sum of irreducible, unitary representations of $O(m)$, and we assume equality in place of the isomorphism, while for the last step, we assume $\sigma|_{SO(2)}$ is diagonalized.
These assumptions are effectively choosing a representative of the isormorphism class of $\sigma$.

If $\what{K}$ is the set of all such $\sigma$, then we have the Peter-Weyl Theorem
\begin{align}
\label{eq:PeterWeyl}
	f(k) =& \sum_{\sigma \in \what{K}} (\dim \sigma) \Tr\paren{\sigma(k) \int_K f(k') \trans{\wbar{\sigma(k)}} dk},
\end{align}
for $f \in L^2(K)$.

\subsection{Whittaker functions}
The Jacquet-Whittaker function arises from applying a certain integral operator to elements of irreducible adminisible representations of $G$.
Since these can always be viewed as subrepresentations of principal series representations and we may decompose $L^2(K)$ using the Peter-Weyl theorem, it is sufficient to consider matrix-valued Whittaker functions of the form $W_\sigma(g,\mu, \delta) = W_\sigma(g,w_l,\mu,\delta,\psi_I)$ where
\begin{gather}
\label{eq:JWDef}
	W_\sigma(g,w,\mu,\delta,\psi) := \int_{\wbar{U}_w(\R)} I_{\mu,\delta,\sigma}(w x g) \wbar{\psi(x)} dx, \qquad \Re(\mu_i-\mu_j) > 0, i<j \\
\begin{aligned}
	I_{\mu,\delta,\sigma} =& \Sigma_{\delta,\sigma} I_{\mu,\sigma}, & I_{\mu,\sigma}(xyk) =& p_{\rho+\mu}(y) \sigma(k), & \Sigma_{\delta,\sigma} =& \frac{1}{\abs{V}} \sum_{v \in V} \chi_\delta(v) \sigma(v),
\end{aligned}\nonumber
\end{gather}
for an irreducible, unitary (finite-dimensional) representation $\sigma$ of $K$.
Hopefully the reader will forgive the nonstandard definition which gives the transformations
\[ W_\sigma(xgk,w,\mu,\delta,\psi) = \psi(x)W_\sigma(g,w,\mu,\delta,\psi)\sigma(k), \qquad x \in \wbar{U}_w(\R), k \in K. \]
Anyone looking for a left action by $\sigma$ can simply replace $\sigma$ with $\sigma^{-1}$ in the definition above.

We may generally reduce to $\psi=\psi_I$ through
\begin{align}
\label{eq:WhittArgToCharFE}
	W_\sigma(g,w,\mu,\delta,\psi_{yt}) =& p_{-\rho-\mu^w}(y) \chi_{\delta^w}(y) W_\sigma(y g,w,\mu,\delta,\psi_t)
\end{align}

The $(\mu,\delta) \mapsto (\mu^w,\delta^w)$ functional equation of the Whittaker function can be written
\begin{align}
\label{eq:WhitFE}
	W_\sigma(g,\mu,\delta) =& M_\sigma(w,\mu,\delta) W_\sigma(g,\mu^w,\delta^w),
\end{align}
where we assume $M_\sigma(w,\mu,\delta) = M_\sigma(w,\mu,\delta) \Sigma_{\delta^w,\sigma^w}$.
Note: More generally, one can drop $\Sigma_{\delta,\sigma}$ from the definition of the Whittaker function and write the functional equation as $W_\sigma(g,\mu) = M_\sigma(w,\mu) W_\sigma(g,\mu^w)$, then $M_\sigma(w,\mu,\delta) = M_\sigma(w,\mu) \Sigma_{\delta^w,\sigma^w}$.

To make later computations slightly clearer, we also define
\begin{align}
\label{eq:WstarDef}
	W^*_{m,\sigma}(g,\mu,\delta) =& p_{-\rho_m}(g) W_\sigma(g,w_{l,m},\mu,\delta,\psi_I);
\end{align}
in case $m \le n$, this is a block-diagonal matrix of $GL(m)$ Whittaker functions (renormalized slightly).

\subsubsection{The Mellin transform}
\label{sect:WhittakerMellin}
For $\abs{\Re(\mu_i)} < \frac{1}{n}$ and $y \in Y^+$, the Whittaker function has a Mellin expansion of the form
\begin{align}
\label{eq:WhittMellinDef}
	W^*_{n,\sigma}(y, \mu,\delta) =& \int_{\Re(s)=1} \what{W}_{n,\sigma}(s,\mu,\delta) \hat{p}_{-s}(y) \frac{ds}{(2\pi i)^{n-1}}.
\end{align}
Note: $\what{W}_{n,\sigma}(s,\mu,\delta)$ is matrix-valued here.

For $S \subset [n-1]:=\set{1,\ldots,n-1}$ set
\begin{align*}
	M_S =& \set{\mu\in M_n \setdiv \hat{\mu}_j=0, \forall j \in S}, \\
	M_S^0=&\set{\mu\in M_S\setdiv \Re(\mu)=0}.
\end{align*}

Let $S=\set{i_1,\ldots,i_q} \subset [n-1]$ in increasing order, i.e. $i_1 < \ldots < i_q$, and take $S^c :=[n-1]-S=\set{i'_1,\ldots,i'_{n-1-q}}$, $i'_1 < \ldots < i'_{n-1-q}$.
For such $S$ and $S^c$, we let $s_{S^c} = (s_{i'_1},\ldots,s_{i'_{n-1-q}})$, and define
\[ \what{W}_{n,\sigma}(s_{S^c},\mu,\delta) = \res_{s_{i_1}=\hat{\mu}_{i_1}}\cdots\res_{s_{i_q}=\hat{\mu}_{i_q}} \what{W}_{n,\sigma}(s,\mu,\delta), \qquad p_{S,\mu}(y) = \prod_{i \in S} y_i^{\hat{\mu}_i}, \qquad \hat{p}_{S^c,-s_{S^c}}(y) = \prod_{i\in S^c} y_i^{-s_i}. \]
Further define $W_S \subset W$ to be the set of all Weyl elements that fix the restricted power function, i.e. $w_c \in W$ such that $p_{S,\mu^{w_c}} = p_{S,\mu}$.
Then by contour shifting, we have
\begin{equation}
\label{eq:WhittContShift}
\begin{aligned}
	& W^*_{n,\sigma}(y, \mu,\delta) =\\
	&\qquad \sum_{S \subset [n-1]} \sum_{w_b \in W/W_S} p_{S,\mu^{w_b}}(y) \int_{\Re(s_{S^c})=-\epsilon} \what{W}_{n,\sigma}(s_{S^c},\mu^{w_b},\delta) \hat{p}_{S^c,-s_{S^c}}(y) \frac{ds_{S^c}}{(2\pi i)^{\abs{S^c}}},
\end{aligned}
\end{equation}
on $\abs{\Re(\mu_i)} < \frac{\epsilon}{n}$.

\subsection{The relevant Weyl elements}
We collect some notes and definitions for the relevant Weyl elements:
Let $w=w_{r_1,\ldots,r_\ell} \in W^\text{rel}$ where $r_1+\ldots+r_\ell=n$.

\[ w^{-1} = w_l w w_l = \Matrix{&&I_{r_\ell}\\&\revdots\\I_{r_1}}, \qquad w U_w w^{-1} = U_{w^{-1}} = \wbar{U}_{w w_l}. \]
\[ U_w = \Matrix{U_{r_\ell}\\&\ddots\\&&U_{r_1}} \qquad \wbar{U}_w = \Matrix{I_{r_\ell}&&*\\&\ddots\\&&I_{r_1}}, \]
where $U_m$ is the upper triangular unipotent group in $GL(m)$.

For relevant Weyl elements, we have $\psi_I(u) = \psi_I(w u w^{-1})$ for all $u \in U_w(\R)$ since the set of first super-diagonal elements is the same.
So we may express the compatibility condition for relevant Weyl elements as $\psi_y(u)=\psi_I(u)$ for all $u \in U_{w^{-1}}(\R)$.
Clearly this constrains some of the coordinates of $y$ to be 1, i.e. that $y$ is block-diagonal with scalar-matrix entries, and the condition is equivalent to $y u y^{-1} = u$ for all $u \in U_{w^{-1}}(\R)$, and this gives rise to the description of $Y_w$ in \eqref{eq:YwDef}.

\subsection{The $GL(2)$ functional equations, explicitly}
Set $\alpha=\diag(1,-1)$.
For $O(2)$, the irreducible, unitary representations are just $1,\det$ and
\[ \sigma_d(h_1(\theta)) = \Matrix{e^{i d\theta}\\&e^{-i d\theta}}, \qquad \sigma_d(\alpha)=\Matrix{&1\\1}, \]
for $d \ge 1$.

The $V_2 = V$ group is composed of
\[ I, \qquad \Matrix{-1\\&-1} = h_1(\pi), \qquad \Matrix{1\\&-1} = \alpha, \qquad \Matrix{-1\\&1} = h_1(\pi) \alpha, \]
and the $\Sigma_{\chi,\sigma}$ matrices can be computed as
\begin{gather}
	\Sigma_{\chi,1} = \piecewise{1 & \If \chi=1, \\ 0 & \Otherwise.}, \qquad \Sigma_{\chi,\det} = \piecewise{1 & \If \chi=\chi_{(1,1)}, \\ 0 & \Otherwise.} \\
\label{eq:Sigmasigmad}
	\Sigma_{\chi_\delta,\sigma_d} = \frac{1}{2} \Matrix{1&(-1)^{\delta_2}\\(-1)^{\delta_2}&1} \times \piecewise{1& \If d \equiv \delta_1+\delta_2 \pmod{2}, \\ 0 & \Otherwise.}
\end{gather}

We have an explicit evaluation of the Jacquet-type integral
\begin{equation}
\label{eq:classWhittDef}
\begin{aligned}
	\mathcal{W}(y,u,d) =& \int_{-\infty}^\infty \paren{1+x^2}^{-\frac{1+u}{2}} \paren{\frac{1-ix}{\sqrt{1+x^2}}}^d \e{-yx} dx \\
	=& \piecewise{\displaystyle \frac{(\pi y)^{\frac{1+u}{2}}}{y \, \Gamma\paren{\frac{1-d+u}{2}}} W_{-\frac{d}{2}, \frac{u}{2}}(4\pi y) & \If y > 0, \\[5pt]
\displaystyle \frac{2^{1-u}\pi \,\Gamma(u)}{\Gamma\paren{\frac{1+u+d}{2}} \Gamma\paren{\frac{1+u-d}{2}}} & \If y=0,}
\end{aligned}
\end{equation}
with $W_{\alpha,\beta}(y)$ the classical Whittaker function (see \cite[Section 2.3.1]{HWI}).
Then the Whittaker function for $GL(2,\R)/\R^+$ is
\begin{align*}
	W_\sigma(I, w_l, \mu, \delta,\psi_y) =& \int_{U(\R)} I_{\mu,\sigma}(w_l x) \wbar{\psi_y(x)} dx \\
	=& \Sigma_{\chi_\delta,\sigma} \sigma(\alpha h_1(\tfrac{3\pi}{2})) \int_{-\infty}^\infty \paren{1+x^2}^{-\frac{1+\mu_1-\mu_2}{2}} \sigma\paren{h_1\paren{\arg\frac{1-ix}{1+x^2}}} \e{-y x} dx,
\end{align*}
and we may write this as
\begin{equation}
\label{eq:GL2Whitt}
\begin{aligned}
	W_1(I, w_l, \mu, \delta,\psi_y) =& \Sigma_{\chi_\delta,1} \mathcal{W}(y_1,\mu_1-\mu_2,0), \\
	W_{\det}(I, w_l, \mu, \delta,\psi_y) =& -\Sigma_{\chi_\delta,\det} \mathcal{W}(y_1,\mu_1-\mu_2,0), \\
	W_{\sigma_d}(I, w_l, \mu, \delta,\psi_y) =& \Sigma_{\chi_\delta,\sigma_d} \sigma_d(\alpha h_1(\tfrac{3\pi}{2})) \Matrix{\mathcal{W}(y_1,\mu_1-\mu_2,d)\\&\mathcal{W}(y_1,\mu_1-\mu_2,-d)}.
\end{aligned}
\end{equation}

For $y > 0$, we have
\[ W_\sigma(y, \mu, \delta) = p_{\rho+\mu^{w_l}}(y) W_\sigma(I, w_l, \mu, \delta,\psi_y) = y^{\frac{1-\mu_1+\mu_2}{2}} W_\sigma(I, w_l, \mu, \delta,\psi_y), \]
and if $\mathcal{W}^*(y,u,d) = y^{\frac{1-u}{2}} \mathcal{W}(y,u,d)$, then we have the functional equation
\[ \mathcal{W}^*(y,u,d) = \Gamma_\mathcal{W}^*(d, u) \mathcal{W}^*(y,-u,d), \qquad \Gamma_\mathcal{W}^*(d, u) = (-1)^d \pi^u \frac{\Gamma\paren{\frac{1+d-u}{2}}}{\Gamma\paren{\frac{1+d+u}{2}}}. \]
Thus we have the functional equations
\begin{equation}
\label{eq:GL2WhittFE}
\begin{aligned}
	M_1(w_l,\mu,\delta) =& \Gamma_\mathcal{W}^*(0, \mu_1-\mu_2) \Sigma_{\chi_{\delta^{w_l}},1}, \\
	M_{\det}(w_l,\mu,\delta) =& \Gamma_\mathcal{W}^*(0, \mu_1-\mu_2) \Sigma_{\chi_{\delta^{w_l}},\det}, \\
	M_{\sigma_d}(y, \mu, \delta) =& (-1)^{\delta_1+\delta_2} \Matrix{\Gamma_\mathcal{W}^*(d, \mu_1-\mu_2)\\&\Gamma_\mathcal{W}^*(-d, \mu_1-\mu_2)} \Sigma_{\chi_{\delta^{w_l}},\sigma_d}.
\end{aligned}
\end{equation}

\section{The Little Proofs}
\subsection{\texorpdfstring{\cref{prop:IoIConsequence}}{Proposition \ref{prop:IoIConsequence}}} 
\label{sect:IoIConsequencePf}
This is Shalika's local multiplicity one theorem for Archimedean Whittaker functions.
Consider the operators $h_1:\sigma \mapsto \wtilde{K}_w(y,t,\Lambda,\mu,\delta,\sigma)$ and $h_2:\sigma \mapsto W_\sigma(t,\mu,\delta)$ with images in the Whittaker model; these extend to operators on an irreducible, admissible representation of $G$ via the Peter-Weyl theorem \eqref{eq:PeterWeyl} and $h_2$ is known to be continuous in the topology of seminorms defined by the Lie algebra.
But $h_1$ is actually defined in terms of $h_2$ by the Interchange of Integrals conjecture, so $h_1$ is also continuous in that topology, hence $h_1$ is a constant multiple of $h_2$ by Shalika's theorem.

\subsection{\texorpdfstring{\cref{prop:IoIConsequences2}}{Proposition \ref{prop:IoIConsequences2}}}
The function $I_{\mu,\delta,\sigma}$ (recall \eqref{eq:JWDef}) satisfies
\[ I_{\mu,\delta,\sigma}((xyk)^\iota) = p_{-\rho^{w_l}-\mu^{w_l}}(y) \Sigma_{\delta,\sigma} \sigma(w_l k w_l) = \sigma(w_l) I_{-\mu^{w_l},\delta^{w_l},\sigma}(xyk) \sigma(w_l) \]
(using $-\rho^{w_l} = \rho$), and $\psi_I(x^\iota) = \wbar{\psi_I(x)}$, so the Whittaker function satisfies
\[ W_\sigma(g,w,\mu,\delta,\psi_I) = \sigma(w_l) W_\sigma(g^\iota, w^\iota, -\mu^{w_l}, \delta^{w_l}, \wbar{\psi_I}) \sigma(w_l). \]
For $v$ as in the statement of the proposition, $\wbar{\psi(x)} = \psi(v x v)$, so the Whittaker function also satisfies
\[ W_\sigma(g,w,\mu,\delta,\wbar{\psi}) = \chi_{\delta^w}(v) W_\sigma(vg,w,\mu,\delta,\psi). \]
Applying this to \eqref{eq:KwDef}, we get the first equality; the second follows from the functional equation of the Whittaker function and the holomorphy (and any bound on the growth) of $M_\sigma(w,\mu,\delta)$ near $\Re(\mu)=0$.
The last statement follows from the fact $\wbar{U}_I = \set{I}$.


\subsection{\texorpdfstring{\cref{prop:MBConsequences}}{Proposition \ref{prop:MBConsequences}}}
For $K^v_{w_l}(y,\mu)$, we can isolate such a linear combination of $J_{w_l}$ by taking the appropriate linear combination
\[ K^v_{w_l}(y,\mu) = \frac{1}{2^n} \sum_\delta \chi_\delta(v) \frac{\sinmu(\mu,\delta)}{(-\pi)^{\frac{n(n-1)}{2}} i^{-(n-1)(\delta_1+\ldots+\delta_n)}} K_{w_l}(vy,\mu,\delta). \]
For $K^\dagger_w(y,\mu,\eta,\delta)$, have
\[ K^\dagger_w(y,\mu,\eta,\delta) = \frac{1}{2^n} \sum_{v \in V} K_w(vy,\mu,\delta) \prod_{i=1}^{\ell-1} v_{\hat{r}_i}^{\eta_i}. \]

\subsection{\texorpdfstring{\cref{thm:ArithWeyl}}{Theorem \ref{thm:ArithWeyl}}}
\label{sect:ArithWeyl}
We take a very crude bound on the Kloosterman sum side of the spectral Kuznetsov formula:
For $T \ge 3$, define
\[ f_1(\Lambda+\wtilde{\mu}) := Q(\wtilde{\mu}) T^{(\wtilde{\mu}_1^2+\ldots+\wtilde{\mu}_n^2)/n^5}, \]
where
\[ Q(\wtilde{\mu}) = \prod_{i<j} \prod_{m=1}^{n^2} \frac{\paren{\wtilde{\mu}_i-\wtilde{\mu}_j}^2-m^2}{\paren{\wtilde{\mu}_i-\wtilde{\mu}_j}^2+n^2} \]
In the Mellin expansion of $K_w(y,\mu,\delta)$, $\mu \in i\mathfrak{a}^*_0(\Lambda)$, consider shifting $\Re(s_j) \mapsto -(\tau_w+\epsilon)\hat{\rho}_j$ as in \cite[Section 8.1]{WeylI} (see also \cref{sect:WhittakerMellin}).
We will encounter some poles at $s_j+\what{\mu}^{w'}_j \in [-\frac{j(n-j)}{2},0]\cap\Z$, $w'\in W$, which we may assume are simple.
For some $S =\set{j_1,\ldots,j_q} \subset [n-1]$, $Z=\set{z_{j_1},\ldots,z_{j_q}}$ with $z_{j_i} \in [-\frac{j_i(n-j_i)}{2},0]\cap\Z$, consider the iterated residues
\begin{align*}
	& E(w,\alpha,\mu,\delta,S,Z) = \\
	& \sup_{s_j \in i\R, j \notin S} \paren{\prod_{i \notin S} (1+\abs{s_i})^{1+\epsilon}} \abs{f_1(\mu) \specmu(\mu)\res_{\alpha_{j_1}+s_{j_1}+\hat{\mu}_{j_1}= z_{j_1}} \cdots \res_{\alpha_{j_q}+s_{j_q}+\hat{\mu}_{j_q}=z_{j_q}} \what{K}_w(\alpha+s, \mu,\delta)}.
\end{align*}
Our final bound on the inverse Bessel transform will use the function
\[ E(\mu,\delta) = \max_{w,w' \in W} \max_{S,Z} E(w,(-\tau_w-\epsilon,\ldots,-\tau_w-\epsilon),\mu^{w'}+\beta_{S,Z},\delta,S,Z), \]
where $\beta_{S,Z}=\sum_{j \in S} ((\tau_w+\epsilon)\hat{\rho}_j+z_j)(e_{n-j}-e_{n+1-j})$ with $e_j$ the standard basis in $\C^n$.
(We are not claiming this is the optimal choice for $\alpha$ and $\beta_{S,Z}$, just that the sum of Kloosterman sums would converge there.)

We apply the spectral Kuznetsov formula \eqref{eq:SpecKuz} with $a=b=I$ and a test function of the form
\[ \left.\int_\Omega f_1(\mu-\mu') d\mu'\middle/\int_{i\mathfrak{a}^*_0(\Lambda)} f_1(\mu') d\mu'\right.. \]
Such a test function is very close to the characteristic function of $\Omega$ except near the boundary, so we also apply the spectral Kuznetsov formula with a test function
\[ \left.\int_{\partial\Omega+B(0,(\log T)^{-1/2+\epsilon})} f_2(\mu-\mu') d\mu'\middle/\int_{i\mathfrak{a}^*_0(\Lambda)} f_2(\mu') d\mu'\right., \quad f_2(\Lambda+\wtilde{\mu}) := Q(\wtilde{\mu}) T^{(\wtilde{\mu}_1^2+\ldots+\wtilde{\mu}_n^2)/n^5/(\log T)^\epsilon}. \]

\section{Matrix Decompositions}
\label{sect:MatrixDecomps}

The goal of this section is essentially to compute the Iwasawa decomposition of $w x$ for $w \in W^\text{rel}$ and $x \in \wbar{U}_w(\R)$.
Throughout this section, the letters $a,b,c,d$ as subscripts are used to distinguish matrices of a similar type, not as indices.

We proceed in stages and the general procedure for each stage is essentially the same:
(carefully) factor $w x=x_a x_b$ with $x_a,x_b \in \trans{U(\R)}$, compute the Iwasawa decomposition of $x_b=x^*_b y^*_b k^*_b$, (somehow) interchange $x_a x^*_b = x^\dagger_b x^\dagger_a$ with $x^\dagger_b \in U(\R)$ and $x^\dagger_a \in \trans{U(\R)}$, compute the Iwasawa decomposition of $y^{*-1}_b x^\dagger_a y^*_b = x^*_a y^*_a k^*_a$, then the final Iwasawa decomposition becomes
\[ wx = \paren{x^\dagger_b y^*_b x^*a y^{*-1}_b} \paren{y^*_b y^*_a} \paren{k^*_a k^*_b}. \]
To simplify the process as much as possible, we perform coordinate substitutions at each stage.

We then apply these matrix decompositions to the integral definitions of the Whittaker and Bessel functions to derive inductive constructions of a sort.
An implementation of the decompositions in Mathematica is given in \cref{sect:AppMD}.

\subsection{The Preliminary Decomposition}
For
\[ w=w_{r_1,\ldots,r_\ell} = \Matrix{&&I_{r_1}\\&\revdots\\I_{r_\ell}}, \]
and $1 \le m < \ell$, define $\ell'=\ell-m$, $n'=\hat{r}_{\ell'}$,
\begin{align*}
	w' =& w_{r_1,\ldots,r_{\ell'}} = \Matrix{&&I_{r_1}\\&\revdots\\I_{r_{\ell'}}}, & w'' =& w_{r_{\ell'+1},\ldots,r_\ell} = \Matrix{&&I_{r_{\ell'+1}}\\&\revdots\\I_{r_\ell}}, \\
	w_a =& \Matrix{w'\\&w''}, & w_b =& w_{n',n-n'} =\Matrix{&I_{n'}\\I_{n-n'}}.
\end{align*}
Notice that $w=w_a w_b$ and
\[ w_b \paren{\wbar{U}_w \cap U_{w_b}} w_b^{-1} = \paren{w_a^{-1} \trans{U} w_a \cap w_b U w_b^{-1}} \cap \paren{U \cap w_b U w_b^{-1}} = \wbar{U}_{w_a} \cap w_b U w_b^{-1} = \wbar{U}_{w_a}. \]

Take $x \in \wbar{U}_w(\R)$ and write $x = \paren{w_b^{-1} x_a w_b} x_b$ with $x_a \in \wbar{U}_{w_a}(\R)$ and
\[ x_b \in \wbar{U}_w(\R) \cap \wbar{U}_{w_b}(\R) = \wbar{U}_{w_b}(\R) = \Matrix{I_{n-n'}&\R^{(n-n')\times n'}\\&I_{n'}}. \]
Then
\[ wx = w_a x_a \paren{w_b x_b w_b^{-1}} w_b, \]
and we write $w_b x_b w_b^{-1} = x_b^* y_b^* k_b^*$ with $x_b^* \in U(\R)$, $y_b^* \in Y$, $k_b^* \in K$.

Lastly, we write $x_a x_b^* = u_b u_a$ with $u_a \in \wbar{U}_{w_a}$ and $u_b \in U_{w_a}$, so that
\[ wx = \paren{w_a u_b w_a^{-1}} \paren{w_a y_b^* w_a^{-1}} w_a \paren{{y_b^*}^{-1} u_a y_b^*} \paren{k_b^* w_b}. \]

Let's examine that last step in more detail using block matrices.
If we write
\[ x_a = \Matrix{A\\&B}, \qquad x_b^* = \Matrix{C & D\\0&E}, \]
with $A,C \in U_{n'}(\R)$, $B,E \in U_{n-n'}(\R)$ and $D \in \R^{n' \times (n-n')}$, then
\begin{align}
\label{eq:BadHatter1}
	x_a x_b^* = \Matrix{I_{n'} & A D E^{-1} B^{-1}\\&I_{n-n'}} \Matrix{A C\\&B E}.
\end{align}
Now let $F$ be the entries of $C$ in $U_{w'}(\R)$ so that $F^{-1} C \in \wbar{U}_{w'}(\R)$ and similarly for $H \in U_{w''}(\R)$ so that $H^{-1} E \in \wbar{U}_{w''}(\R)$, then
\[ u_b = \Matrix{I_{n'} & A D E^{-1} B^{-1}\\&I_{n-n'}}\Matrix{F\\&H}, \qquad u_a = \Matrix{F^{-1} A C\\&H^{-1} B E} \]
(in terms of block matrices, $F$ and $H$ are diagonal while $A$ and $B$ are upper unipotent, so conjugating by $F$ and $H$ leaves $A$ and $B$ upper unipotent), and
\begin{align}
\label{eq:BadHatter2}
	w_a u_b w_a^{-1} =& \Matrix{I_{n'} & w' A D E^{-1} B^{-1} {w''}^{-1}\\&I_{n-n'}}\Matrix{w'F{w'}^{-1}\\&w''H{w''}^{-1}}.
\end{align}

Consider an integral of the form
\[ \mathcal{I}_w(f_1,f_2) := \int_{\wbar{U}_w(\R)} f_1(wx) f_2(x) dx, \]
where $f_1(xgk) = \psi_y(x) f_1(g) \sigma(k)$.
If $\ell=1$, we have $w=I_{r_1}$ and $\wbar{U}_w(\R)=\set{I_{r_1}}$, so
\[ \mathcal{I}_w(f_1,f_2) = f_1(I_{r_1}) f_2(I_{r_1}).  \]
Otherwise, substituting $u_a \mapsto x_a$ (i.e. $F^{-1} A C \mapsto A$, $H^{-1} B E \mapsto B$) has Jacobian determinant 1 and gives
\begin{align*}
	& \mathcal{I}_w(f_1,f_2) = \\
	& \int_{\wbar{U}_{w_b}(\R)} \int_{\wbar{U}_{w_a}(\R)} f_1\paren{\paren{w_a y_b^* w_a^{-1}} w_a \paren{(y_b^*)^{-1} x_a y_b^*}} f_2\paren{\Matrix{H\\&F} \paren{w_b^{-1} x_a w_b} \Matrix{E^{-1}\\&C^{-1}} x_b} \\
	& \qquad \psi_y\Matrix{I_{n'} & w' F A C^{-1} D B^{-1} H^{-1} {w''}^{-1}\\&I_{n-n'}} dx_a \, \psi_y\Matrix{w'F{w'}^{-1}\\&w''H{w''}^{-1}} \sigma(k_b^* w_b) dx_b.
\end{align*}
The conjugation $(y_b^*)^{-1} x_a y_b^* \mapsto x_a$ has Jacobian determinant $p_{\rho-\rho^{w_a}}(y_b^*)$ (see \cref{sect:Measures}), and we get
\begin{equation}
\label{eq:MainIoIDecomp}
\begin{aligned}
	& \mathcal{I}_w(f_1,f_2) = \\
	& \int_{\wbar{U}_{w_b}(\R)} \int_{\wbar{U}_{w_a}(\R)} f_1\paren{\paren{w_a y_b^* w_a^{-1}} w_a x_a} f_2\paren{\Matrix{H\\&F} \paren{w_b^{-1} y_b^* x_a (y_b^*)^{-1} w_b} \Matrix{E^{-1}\\&C^{-1}} x_b} \\
	& \qquad \psi_y\Matrix{I_{n'} & w' F A' C^{-1} D {B'}^{-1} H^{-1} {w''}^{-1}\\&I_{n-n'}} dx_a \, \psi_y\Matrix{w'F{w'}^{-1}\\&w''H{w''}^{-1}} \\
	& \qquad \times p_{\rho-\rho^{w_a}}(y_b^*) \sigma(k_b^* w_b) dx_b,
\end{aligned}
\end{equation}
where
\[ \Matrix{A'\\&B'} = y_b^* x_a (y_b^*)^{-1}. \]

\subsection{The Decomposition of $w_b x_b w_b^{-1}$}
For $i \ne j$, consider the matrix $X_{i,j} = X_{i,j}(x_{i,j}) = \exp\paren{x_{i,j} E_{i,j}} \in U(\R)$ which has all off-diagonal coordinates zero except for $x_{ij}$ at position $i,j$.
It is not too hard to see that for $j_1\ne i_1,j_2\ne i_2$ (or $i_1=i_2$ and $j_1=j_2$), $X_{i_1,j_1}$ and $X_{i_2,j_2}$ commute unless $i_1=j_2$ or $j_1=i_2$; in those cases, we have
\begin{align*}
	X_{i_1,j_1}(x_1) X_{i_2,i_1}(x_2) =& X_{i_2,i_1}(x_2) X_{i_2,j_1}(-x_1 x_2) X_{i_1,j_1}(x_1) \qquad i_2 \ne i_1,j_1, \\
	X_{i_1,j_1}(x_1) X_{j_1,j_2}(x_2) =& X_{j_1,j_2}(x_2) X_{i_1,j_2}(x_1 x_2) X_{i_1,j_1}(x_1) \qquad j_2 \ne i_1,j_1
\end{align*}
(the remaining case $i_1=j_2$ and $j_1=i_2$ being a bit more complicated).
The nonzero superdiagonal coordinates $x_{b,i,j}$ of $x_b$ have $i \le n-n'$ and $j > n-n'$, hence $\wbar{U}_{w_a}(\R)$ is isomorphic to $(\R^{(n-n') \times n'}, +)$ as a multiplicative group.
That is, we can freely factor elements of $\wbar{U}_{w_a}(\R)$ into products of $X_{i,j}$ matrices in any order.

We recursively compute $x^*_b$, $y^*_b$, $k^*_b$ and $x^{*-1}_b$ as follows:
First, we compute the Iwasawa decomposition of a single row.
For $x \in \R$, define
\begin{gather}
	\hat{y}_{n,i}(x) = \diag\paren{I_{i-1},\tfrac{1}{\sqrt{1+x^2}},I_{n-i-1},\sqrt{1+x^2}}, \\
	\hat{k}_{n,i}(x) = w_l h_{n+1-i}\paren{\arg\frac{1-i x}{\sqrt{1+x^2}}} w_l = \Matrix{I_{i-1}\\&\frac{1}{\sqrt{1+x^2}}&0&-\frac{x}{\sqrt{1+x^2}}\\&0&I_{n-i-1}&0\\&\frac{x}{\sqrt{1+x^2}}&0&\frac{1}{\sqrt{1+x^2}}}, \\
	\eta_{n,i,j} = \piecewise{-2 & \If i=j=n-1, \\ -1 & \If i=j\ne n-1 \text{ or } i=n-1\ne j, \\ 1 & \If i=j-1 \text{ or } i=n, \\ 0 & \Otherwise,}
\end{gather}
so that the $i$th $Y$-coordinate of $\hat{y}_{n,j}(x)$ is
\begin{align}
\label{eq:hatyYCoords}
	(\hat{y}_{n,j}(x))_i = \paren{1+x^2}^{\eta_{n,i,j}/2}.
\end{align}

\begin{lem}
\label{lem:RowIwa}
For a row vector $u = (u_1, \ldots, u_{n'}) \in \R^{n'}$,
\[ \Matrix{I_{n'} \\0&I_{n-n'-1}\\u&0&1} = x^\dagger y^\dagger k^\dagger, \qquad x^\dagger = \Matrix{u^*_c & 0 & u^*_d\\&I_{n-n'-1}&0\\&&1}, \qquad x^{\dagger-1} = \Matrix{\wtilde{u}^*_c & 0 & \wtilde{u}^*_d\\&I_{n-n'-1}&0\\&&1}, \]
where $u^*_c \in U_{n'}(\R)$, $u^*_d =\trans{(u^*_{d,1}, \ldots, u^*_{d,n'})} \in \R^{n'}$ as a column vector, $y^\dagger \in Y^+$ and $k^\dagger \in K$.
After performing the substitutions
\begin{align*}
	u_2 \mapsto& u_2 \sqrt{1+u_1^2}, &\ldots,&& u_{n'} \mapsto& u_{n'} \prod_{i=1}^{n'-1}\sqrt{1+u_i^2}
\end{align*}
(in that order) with Jacobian
\[ \prod_{i=1}^{n'-1} \paren{1+u_i^2}^{(n'-i)/2}, \]
the decomposition is given by:
\begin{align*}
	u^*_{c,i,j} =& -\frac{u_i u_j}{\prod_{l=i}^{j-1} \sqrt{1+u_l^2}}, & \wtilde{u}^*_{c,i,j} =& \frac{u_i u_j \prod_{l=i+1}^{j-1} \sqrt{1+u_l^2}}{\sqrt{1+u_i^2}}, \qquad j > i, \\
	u^*_{d,i} =& \frac{u_i}{\prod_{j=1}^{n'} \sqrt{1+u_j^2} \prod_{j=i}^{n'} \sqrt{1+u_j^2}}, & \wtilde{u}^*_{d,i} =& -\frac{u_i}{(1+u_i^2) \prod_{j=1}^{i-1} \sqrt{1+u_j^2}},\\
	y^\dagger =& \hat{y}_{n,1}(u_1) \cdots \hat{y}_{n,n'}\paren{u_{n'}}, & k^\dagger =& \hat{k}_{n,n'}\paren{u_{n'}} \cdots \hat{k}_{n,1}(u_1).
\end{align*}
\end{lem}
\begin{proof}
The proof is effectively identical to the previous decompositions, starting with the readily verifiable identity
\[ X_{n,i}(u_i) = X_{i,n}\paren{\frac{u_i}{1+u_i^2}} \hat{y}_{n,i}(u_i) \hat{k}_{n,i}(u_i). \]
We inductively separate the $u_i$ coordinate, starting at $i=1$, apply the identity and then conjugate the remaining coordinates
\begin{align*}
	&\paren{\prod_{j=i+1}^{n-n'} X_{n,j}(u_j)} X_{i,n}\paren{\frac{u_i}{1+u_i^2}} \hat{y}_{n,i}(u_i) = \\
	& X_{i,n}\paren{\frac{u_i}{1+u_i^2}} \paren{\prod_{j=i+1}^{n-n'} X_{i,j}\paren{-\frac{u_i u_j}{1+u_i^2}}} \hat{y}_{n,i}(u_i) \paren{\prod_{j=i+1}^{n-n'} \paren{\hat{y}_{n,i}(u_i)^{-1} X_{n,j}(u_j) \hat{y}_{n,i}(u_i)}}.
\end{align*}
Substituting $u_j \mapsto u_j\sqrt{1+u_i^2}$ sends $\hat{y}_{n,i}(u_i)^{-1} X_{n,j}(u_j) \hat{y}_{n,i}(u_i) \mapsto X_{n,j}(u_j)$, and it is not too hard to see (though it does require an induction) that $x^\dagger y^\dagger$, which is the product of the terms
\[ X_{i,n}\paren{\frac{u_i}{1+u_i^2}} \paren{\prod_{j=i+1}^{n-n'} X_{i,j}\paren{-\frac{u_i u_j}{\sqrt{1+u_i^2}}}} \hat{y}_{n,i}(u_i) \]
for $i=1,\ldots,n-n'$ in left-to-right order, has the given entries.
Similarly, $y^{\dagger-1} x^{\dagger-1}$ is the product of terms
\[ \hat{y}_{n,i}(u_i)^{-1} \paren{\prod_{j=i+1}^{n-n'} X_{i,j}\paren{\frac{u_i u_j}{\sqrt{1+u_i^2}}}} X_{i,n}\paren{-\frac{u_i}{1+u_i^2}} \]
for $i=1,\dots,n-n'$ in right-to-left order.
\end{proof}

Now we use the lemma to compute the Iwasawa decomposition of $w_b x_b w_b^{-1}$ in \eqref{eq:MainIoIDecomp}.
If we peel off the bottom row of $x_b$, we have
\[ x_b = \Matrix{I_{n-n'-1} & 0&x_c\\&1&u\\&& I_{n'}} = \Matrix{I_{n-n'-1} & 0&x_c\\&1&0\\&& I_{n'}} \Matrix{I_{n-n'-1} & 0&0\\&1&u\\&& I_{n'}}, \]
with $x_c \in \R^{(n-n'-1) \times n'}$ and $u \in \R^{n'}$ (as a row vector).
Conjugating by $w_b$ gives
\[ w_b x_b w_b^{-1} = \Matrix{I_{n'} \\x_c&I_{n-n'-1}\\0&0&1} \Matrix{I_{n'} \\0&I_{n-n'-1}\\u&0&1}. \]

Applying \cref{lem:RowIwa} gives
\[ w_b x_b w_b^{-1} = \Matrix{u^*_c& 0 & u^*_d\\x_c u^*_c&I_{n-n'-1}&x_c u^*_d\\0&0&1} y^\dagger k^\dagger \]
and some conjugating results in
\[ w_b x_b w_b^{-1} = \Matrix{u^*_c & 0 & u^*_d\\&I_{n-n'-1}&x_c u^*_d\\&&1} y^\dagger \, y^{\dagger-1} \Matrix{I_{n'} \\x_c u^*_c&I_{n-n'-1}\\0&0&1} y^\dagger k^\dagger. \]
Therefore, if
\[ y^{\dagger-1} \Matrix{I_{n'} \\x_c u^*_c&I_{n-n'-1}\\0&0&1} y^\dagger = x_c^* y_c^* k_c^*, \]
then we have
\begin{align*}
	x_b^* =& \Matrix{u^*_c & 0 & u^*_d\\&I_{n-n'-1}&x_c u^*_d\\&&1} y^\dagger x_c^* y^{\dagger-1}, &
	y_b^* =& y^\dagger y_c^*, &
	k_b^* =& k_c^* k^\dagger.
\end{align*}

We briefly use the entries $y^\dagger_{i,i}$ of the diagonal matrix $y^\dagger$ rather than the $Y$-coordinates.
In the integral \eqref{eq:MainIoIDecomp}, if one applies the substitutions of the lemma and then substitutes
\begin{align*}
	y^{\dagger-1} \Matrix{I_{n'} \\x_c u^*_c&I_{n-n'-1}\\0&0&1} y^\dagger \mapsto \Matrix{I_{n'} \\x_c&I_{n-n'-1}\\0&0&1},
\end{align*}
the combined Jacobian is
\[ \prod_{i=1}^{n'-1} \paren{1+u_i^2}^{(n'-i)/2}\paren{\prod_{i=1}^{n-n'-1} y^\dagger_{n'+i,n'+i}}^{n'} \paren{\prod_{j=1}^{n'} y^\dagger_{j,j}}^{-(n-n'-1)} = \prod_{i=1}^{n'} \paren{1+u_i^2}^{(n-1-i)/2}, \]
and using $u^{*-1}_c u^*_d = \wtilde{u}^*_c u^*_d = -\wtilde{u}^*_d$, the final result becomes
\begin{align}
\label{eq:xbConjDecomp}
	x_b^* =& \Matrix{u^*_c & 0 & u^*_d\\&I_{n-n'-1}& x_c^\dagger\\&&1} y^\dagger x_c^* y^{\dagger-1}, &
	y_b^* =& y^\dagger y_c^*, &
	k_b^* =& k_c^* k^\dagger,
\end{align}
where the column vector $x_c^\dagger \in \R^{n-n'-1}$ has entries
\[ x_{c,i}^\dagger = -\sum_{j=1}^{n'} x_{c,i,j} \frac{y^\dagger_{n'+i,n'+i}}{y^\dagger_{j,j}} \wtilde{u}^*_{d,j} = \sum_{j=1}^{n'} \frac{u_j x_{c,i,j}}{\prod_{l=1}^j \sqrt{1+u_l^2}}. \]

\subsection{The Whittaker function}
\label{sect:MDWhitt}
In this section, we apply \eqref{eq:MainIoIDecomp} to the Jacquet-Whittaker function.
Suppose $f_1=\Sigma_{\delta,\sigma} I_{\mu,\sigma}$ and $f_2=\wbar{\psi_t}$ so that $\psi_y=1$, then
\begin{align*}
	W_\sigma(I,w,\mu,\delta,\psi_t) =& \int_{\wbar{U}_{w_b}(\R)} \int_{\wbar{U}_{w_a}(\R)} I_{\mu,\delta,\sigma}\paren{w_a x_a} \wbar{\psi_{w_b t w_b^{-1} y_b^*}(x_a)} dx_a \, \psi_t\Matrix{H^{-1} E\\&F^{-1} C} \\
	& \qquad \times \wbar{\psi_t(x_b)} p_{\rho+\mu^{w_a}}(y_b^*) \sigma(k_b^* w_b) dx_b,
\end{align*}
since $\psi_I(w_b^{-1} x_a w_b) = \psi_I(x_a)$.

It is not too hard to see that
\begin{align}
\label{eq:tcharArg}
	\psi_t\Matrix{H^{-1} E\\&F^{-1} C} = \e{\sum_{i=\ell'+1}^{\ell-1} t_{n-\hat{r}_i} x^*_{b,n'+n-\hat{r}_i,n'+n-\hat{r}_i+1}+\sum_{i=1}^{\ell'-1} t_{n-\hat{r}_i} x^*_{b,n'-\hat{r}_i,n'-\hat{r}_i+1}},
\end{align}
and
\[ \wbar{\psi_t(x_b)} = \e{-t_{n-n'} x_{b,n-n',n-n'+1}}. \]

The unipotent group $\wbar{U}_{w_a}(\R)$ is composed of two commuting blocks, so we may write
\begin{align*}
	W_\sigma(I,w,\mu,\delta,\psi_t) =& \int_{\wbar{U}_{w_b}(\R)} W_\sigma(I,w',\mu,\delta,\psi_{w_b t w_b^{-1} y_b^*}) \sigma(w_b) W_\sigma(I,w'',\mu^{w_b},\delta^{w_b},\psi_{t w_b^{-1} y_b^* w_b}) \sigma(w_b^{-1}) \\
	& \qquad \times \psi_t\Matrix{H^{-1} E\\&F^{-1} C} \wbar{\psi_t(x_b)} p_{\rho+\mu^{w_a}}(y_b^*) \sigma(k_b^* w_b) dx_b \\
	=& \int_{\wbar{U}_{w_b}(\R)} \sigma(w_b) W_\sigma(I,w'',\mu^{w_b},\delta^{w_b},\psi_{t w_b^{-1} y_b^* w_b}) \sigma(w_b^{-1}) W_\sigma(I,w',\mu,\delta,\psi_{w_b t w_b^{-1} y_b^*}) \\
	& \qquad \times \psi_t\Matrix{H^{-1} E\\&F^{-1} C} \wbar{\psi_t(x_b)} p_{\rho+\mu^{w_a}}(y_b^*) \sigma(k_b^* w_b) dx_b,
\end{align*}
since $\sgn(y_b^*) = I$ and
\[ \Sigma_{\delta,\sigma} I_{\mu,\sigma}(w_b yk w_b^{-1}) = \Sigma_{\delta,\sigma} \sigma(w_b) I_{\mu^{w_b},\sigma}(yk) \sigma(w_b^{-1}) = \sigma(w_b) \Sigma_{\delta^{w_b},\sigma} I_{\mu^{w_b},\sigma}(yk) \sigma(w_b^{-1}). \]
Here we are embedding $w'$ and $w''$ into $G$ in the upper left, as usual.

Now take $w=w_l$.
Recalling \eqref{eq:WstarDef} and \eqref{eq:WhittArgToCharFE}, we see $w_b=w_{n-m,m}$, $w'=w_{l,n-m}$, $w''=w_{l,m}$,
\begin{align*}
	W_\sigma(I,w',\mu,\delta,\psi_{w_b t w_b^{-1} y_b^*}) =& p_{-{\mu'}^{w_{l,n-m}}}(t' y') \chi_{{\delta'}^{w_{l,n-m}}}(t') W^*_{n-m,\sigma}(t' y', \mu', \delta'), \\
	W_\sigma(I,w'',\mu^{w_b},\delta^{w_b},\psi_{t y_b^\dagger}) =& p_{-{\mu''}^{w_{l,m}}}(t'' y'') \chi_{{\delta''}^{w_{l,m}}}(t'') W^*_{m,\sigma}(t'' y'', \mu'', \delta''),
\end{align*}
where
\begin{gather}
\label{eq:WhittDecompParams1}
\begin{aligned}
	\delta' =& (\delta_1,\ldots,\delta_{n-m}), & \delta'' =& (\delta_{n-m+1},\ldots,\delta_n), \\
	t' =& (t_{m+1},\ldots,t_{n-1},1), & t'' =& (t_1,\ldots,t_{m-1},1), \\
	y' =& (y^*_{b,1},\ldots,y^*_{b,n-1-m},1), & y'' =& (y^*_{b,n-m+1},\ldots,y^*_{b,n-1},1), \\
	y_b^* =& \prod_{j=1}^m \prod_{k=1}^{n-m} \hat{y}_{n-m+j,k}\paren{x_{b,j,m+k}}, & y_{b,i}^* =& \prod_{j=1}^m \prod_{k=1}^{n-m} (1+x_{b,j,m+k}^2)^{\eta_{n-m+j,i,k}/2},
\end{aligned} \\
\label{eq:WhittDecompParams2}
\begin{aligned}
	\mu' =& (\mu_1,\ldots,\mu_{n-m})-\tfrac{\mu_1+\ldots+\mu_{n-m}}{n-m}(1,\ldots,1), \\
	\mu'' =& (\mu_{n-m+1},\ldots,\mu_n)-\tfrac{\mu_{n-m+1}+\ldots+\mu_n}{m}(1,\ldots,1), \\
	k_b^* =& \hat{k}_{n-m,n-m}(x_{b,1,m+1})\cdots\hat{k}_{n-m,1}(x_{b,1,n})\cdots\hat{k}_{n,n-m}(x_{b,m,m+1})\cdots\hat{k}_{n,1}(x_{b,m,n}),
\end{aligned}
\end{gather}
and if either $n-m=1$ or $m=1$, we take the corresponding Whittaker function to be 1.
Then
\begin{align*}
	p_{-{\mu'}^{w_{l,n-m}}}(t') p_{-{\mu''}^{w_{l,m}}}(t'') p_{\mu^{w_{l,n}}}(t) =& \abs{t_1^1\cdots t_m^m}^{\frac{\mu_{n-m+1}+\ldots+\mu_n}{m}} \abs{t_{m+1}^{n-m-1}\cdots t_{n-1}^1}^{\frac{\mu_{n-m+1}+\ldots+\mu_n}{n-m}}, \\
	\chi_{{\delta'}^{w_{l,n-m}}}(t') \chi_{{\delta''}^{w_{l,m}}}(t'') \chi_{\delta^{w_{l,n}}}(t) =& \sgn(t_m\cdots t_{n-1})^{\delta_{n-m+1}+\ldots+\delta_n} \sgn(t_n)^{\delta_1+\ldots+\delta_n}, \\
	p_{-{\mu'}^{w_{l,n-m}}}(y') p_{-{\mu''}^{w_{l,m}}}(y'') p_{\rho+\mu^{w_a}}(y_b^*) =& \prod_{i=1}^m \prod_{j=1}^{n-m} \paren{1+x_{b,i,n+1-j}^2}^{-\frac{n+1-i-j}{2}+\frac{n}{2m(n-m)}(\mu_{n-m+1}+\ldots+\mu_n)}, \\
\end{align*}
and the Jacobian of the change of variables applying \eqref{eq:xbConjDecomp} to each row $i=1,\ldots,m$ of $x_b$ is
\[  \prod_{i=1}^m \prod_{j=1}^{n-m} \paren{1+x_{b,i,n+1-j}^2}^{(n-i-j)/2}. \]
Thus for $t \in Y$, we have
\begin{equation}
\label{eq:WhittDecomp}
\begin{aligned}
	& W^*_{n,\sigma}(t,\mu,\delta) \abs{t_1^1\cdots t_m^m}^{-\frac{\mu_{n-m+1}+\ldots+\mu_n}{m}} \abs{t_{m+1}^{n-m-1}\cdots t_{n-1}^1}^{-\frac{\mu_{n-m+1}+\ldots+\mu_n}{n-m}} \\
	& \qquad \times \sgn(t_m\cdots t_{n-1})^{\delta_{n-m+1}+\ldots+\delta_n} \sgn(t_n)^{\delta_1+\ldots+\delta_n} \\
	&= \int_{\wbar{U}_{w_{n-m,m}}(\R)} W^*_{n-m,\sigma}(t' y', \mu', \delta') \sigma(w_{n-m,m}) W^*_{m,\sigma}(t'' y'', \mu'', \delta'') \sigma(w_{n-m,m}^{-1} k_b^* w_{n-m,m})  \\
	& \qquad \times \e{\sum_{i=1}^{m-1} t_i x^*_{b,n-m+i,n-m+i+1}+\sum_{i=m+1}^{n-1} t_i x^*_{b,i-m,i-m+1}-t_m x_{b,m,m+1}} \\
	& \qquad \times \prod_{i=1}^m \prod_{j=m+1}^n \paren{1+x_{b,i,j}^2}^{-\frac{1}{2}+\frac{n}{2m(n-m)}(\mu_{n-m+1}+\ldots+\mu_n)} dx_b \\
	&= \int_{\wbar{U}_{w_{n-m,m}}(\R)} \sigma(w_{n-m,m}) W^*_{m,\sigma}(t'' y'', \mu'', \delta'') \sigma(w_{n-m,m}^{-1}) W^*_{n-m,\sigma}(t' y', \mu', \delta') \sigma(k_b^* w_{n-m,m}) \\
	& \qquad \times \e{\sum_{i=1}^{m-1} t_i x^*_{b,n-m+i,n-m+i+1}+\sum_{i=m+1}^{n-1} t_i x^*_{b,i-m,i-m+1}-t_m x_{b,m,m+1}} \\
	& \qquad \times \prod_{i=1}^m \prod_{j=m+1}^n \paren{1+x_{b,i,j}^2}^{-\frac{1}{2}+\frac{n}{2m(n-m)}(\mu_{n-m+1}+\ldots+\mu_n)} dx_b,
\end{aligned}
\end{equation}
and the coordinates of $x_b^*$ can be worked out by induction using \eqref{eq:xbConjDecomp}.

\subsection{The Bessel function}
\label{sect:BesselMatrixDecomps}
Now suppose $f_1(xg) = \psi_I(x) f_1(g)$ and
\[ f_2(x) = f_{2,n}(x;r,t,\alpha) = \wbar{\psi_t(x)} \e{\alpha_1 x_{r_\ell,r_\ell+1}+\ldots+\alpha_{n'} x_{r_\ell,n}}. \]
We assume $m=1$ so $n'=n-r_\ell$, $w''=B=I_{r_\ell}$ and $H=E$:
\begin{align*}
	& \mathcal{I}_w(f_1(y\cdot),f_{2,n}(\cdot;r,t,\alpha)) = \\
	& \int_{\wbar{U}_{w_b}(\R)} \int_{\wbar{U}_{w_a}(\R)} f_1\paren{y\paren{w_a y_b^* w_a^{-1}} w_a x_a} f_2\paren{\Matrix{E\\&F} \paren{w_b^{-1} y_b^* x_a (y_b^*)^{-1} w_b} \Matrix{E^{-1}\\&C^{-1}} x_b} \\
	& \qquad \psi_y\Matrix{I_{n'} & w' F A' C^{-1} D E^{-1}\\&I_{r_\ell}} dx_a \, \psi_y\Matrix{w'F{w'}^{-1}\\&E} p_{\rho-\rho^{w_a}}(y_b^*) \sigma(k_b^* w_b) dx_b.
\end{align*}
We have
\[ \Matrix{E\\&F} \paren{w_b^{-1} y_b^* x_a (y_b^*)^{-1} w_b} \Matrix{E^{-1}\\&C^{-1}} = \Matrix{I_{r_\ell}\\&*}, \]
so this becomes
\begin{align*}
	& \mathcal{I}_w(f_1(y\cdot),f_{2,n}(\cdot;r,t,\alpha)) = \\
	& \int_{\wbar{U}_{w_b}(\R)} \int_{\wbar{U}_{w_a}(\R)} f_1\paren{y\paren{w_a y_b^* w_a^{-1}} w_a x_a} \psi_y\Matrix{I_{n'} & w' F A' C^{-1} D E^{-1}\\&I_{r_\ell}} \wbar{\psi_{w_b t w_b^{-1} (y_b^*)}(x_a)} dx_a \\
	& \qquad \psi_t\Matrix{I_{r_\ell}\\&F^{-1} C} \wbar{\psi_t(x_b)} \e{\sum_{i=1}^{n'} \alpha_i x_{b,r_\ell,r_\ell+i}} \psi_y\Matrix{w'F{w'}^{-1}\\&E} p_{\rho-\rho^{w_a}}(y_b^*) \sigma(k_b^* w_b) dx_b.
\end{align*}
Notice that
\[ \Matrix{C^{-1} & -C^{-1} D E^{-1}\\&E^{-1}} = {x_b^*}^{-1}, \]
so
\begin{align*}
	& \psi_y\Matrix{I_{n'} & w' F A' C^{-1} D E^{-1}\\&I_{r_\ell}} \\
	&= \e{y_{n'} (w' F A' C^{-1} D E^{-1})_{n',1}} \\
	&= \e{y_{n'} (F A' C^{-1} D E^{-1})_{r_{\ell-1},1}} \\
	&= \e{y_{n'} (A' C^{-1} D E^{-1})_{r_{\ell-1},1}} \\
	&= \e{-y_{n'} \sum_{i=r_{\ell-1}}^{n'} \frac{y^*_{b,r_{\ell-1},r_{\ell-1}}}{y^*_{b,i,i}} x_{a,r_{\ell-1},i} \paren{{x_b^*}^{-1}}_{i,n'+1}}\\
	&= \e{-y_{n'} \paren{{x_b^*}^{-1}}_{r_{\ell-1},n'+1} -y_{n'} \sum_{i=r_{\ell-1}+1}^{n'} \paren{\prod_{j=r_{\ell-1}}^{i-1}y^*_{b,j}} x_{a,r_{\ell-1},i} \paren{{x_b^*}^{-1}}_{i,n'+1}}.
\end{align*}
(Note this is $r_{\ell-1}$ not $r_\ell-1$.)
From the calculation on Whittaker functions, i.e. \eqref{eq:tcharArg}, we have
\[ \psi_t\Matrix{I_{r_\ell}\\&F^{-1} C} = \e{\sum_{i=1}^{\ell-2} t_{n-\hat{r}_i} x^*_{b,n'-\hat{r}_i,n'-\hat{r}_i+1}}, \]
and it is not too hard to see that
\[ \psi_y\Matrix{w'F{w'}^{-1}\\&E} = \e{\sum_{j=n'+1}^{n-1} y_j x^*_{b,j,j+1}+ \sum_{i=1}^{\ell-1} \sum_{j=n'-\hat{r}_i+1}^{n'-\hat{r}_{i-1}-1} y_{\hat{r}_i+\hat{r}_{i-1}-n'+j} x^*_{b,j,j+1}}. \]

Define
\begin{align*}
	r^\sharp =& (r_1,\ldots,r_{\ell-1}), \\
	y^\flat \Matrix{y^\sharp\\&I_{r_\ell}} :=& y w_a y_b^* w_a^{-1}, \\
	t^\sharp =& \paren{(w_b t w_b^{-1} y^*_b)_1, \ldots, (w_b t w_b^{-1} y^*_b)_{n'-1}, 1} \\
	=& \paren{t_{r_\ell+1} y^*_{b,1}, \ldots, t_{n-1} y^*_{b,n'-1}, 1}, \\
	\alpha^\sharp =& -y_{n'} \paren{y^*_{b,r_{\ell-1}} \paren{{x_b^*}^{-1}}_{r_{\ell-1}+1,n'+1},\ldots, y^*_{b,r_{\ell-1}} \cdots y^*_{b,n'-1} \paren{{x_b^*}^{-1}}_{n',n'+1}}, \\
	f_1^\sharp(g) =& f_1\paren{y^\flat\Matrix{g\\&I_{r_\ell}}},
\end{align*}
where the lower right entry of $y^\sharp$ is $y^\sharp_{n',n'} = 1$ and $y^\flat_i =1$ for all $i < n'$.
Then we have our first inductive form
\begin{align*}
	& \mathcal{I}_w(f_1(y\cdot),f_{2,n}(\cdot;r,t,\alpha)) = \\
	& \int_{\wbar{U}_{w_b}(\R)} \e{\sum_{i=1}^{n'} \alpha_i x_{b,r_\ell,r_\ell+i}-t_{r_\ell}x_{b,r_\ell,r_\ell+1}-y_{n'} \paren{x^{*-1}_b}_{r_{\ell-1},n'+1}} \\
	& \qquad \e{\sum_{i=1}^{\ell-2} t_{n-\hat{r}_i} x^*_{b,n'-\hat{r}_i,n'-\hat{r}_i+1}+\sum_{j=n'+1}^{n-1} y_j x^*_{b,j,j+1}+ \sum_{i=1}^{\ell-1} \sum_{j=n'-\hat{r}_i+1}^{n'-\hat{r}_{i-1}-1} y_{\hat{r}_i+\hat{r}_{i-1}-n'+j} x^*_{b,j,j+1}} \\
	& \qquad \mathcal{I}_{w'}\paren{f_1^\sharp\paren{y^\sharp \cdot},f_{2,n-r_1}\paren{\cdot; r^\sharp, t^\sharp, \alpha^\sharp}} p_{\rho-\rho^{w_a}}(y_b^*) \sigma(k_b^* w_b) dx_b.
\end{align*}
As usual, we are using the upper left embedding of $GL(n-r_1,\R)$ into $G$ in the functions $f_1^\sharp, f_{2,n-r_1}$.
Some care must be taken as $f_1^\sharp$ and $f_{2,n-r_1}$ are now functions on $GL(n-r_1,\R)$ instead of $GL(n-r_1,\R)/\R^+$, but our construction of $y^\sharp, y^\flat$ mostly resolves the issue; it is sufficient to treat the $Y_{n-r_1}$ component of the Iwasawa decomposition as a diagonal matrix rather than requiring its lower right entry to be $\pm1$.

Now we apply \eqref{eq:xbConjDecomp}; after the substitutions of \cref{lem:RowIwa}, if
\[ \Matrix{I_{n'} \\x_c&I_{n-n'-1}\\0&0&1} = x_c^* y_c^* k_c^*, \]
we have
\begin{align*}
	x_{b,r_\ell,r_\ell+i} \mapsto& u_i \prod_{j=1}^{i-1} \sqrt{1+u_i^2}, \\
	\paren{x_b^{*-1}}_{i,n'+1} =& (y^\dagger x_c^{*-1} y^{\dagger-1})_{i,n'+1} = \frac{y^\dagger_{i,i}}{y^\dagger_{n'+1,n'+1}} x^{*-1}_{c,i,n'+1} = \frac{1}{\sqrt{1+u_i^2}} x^{*-1}_{c,i,n'+1}, \\
	x^*_{b,i,i+1} =& y^\dagger_i x^*_{c,i,i+1}+u^*_{c,i,i+1} = y^\dagger_i x^*_{c,i,i+1}-\frac{u_i u_{i+1}}{\sqrt{1+u_i^2}}, \qquad i=1,\ldots,n'-1, \\
	y^*_{b,r_{\ell-1}} \cdots y^*_{b,i-1} =& \frac{y^\dagger_{r_{\ell-1},r_{\ell-1}}}{y^\dagger_{i,i}} = \frac{\sqrt{1+u_i^2}}{\sqrt{1+u_{r_{\ell-1}}^2}}, \qquad i=r_{\ell-1},\ldots,n'-1,\\
	x^*_{b,i,i+1} =& y^\dagger_i x^*_{c,i,i+1}, \qquad i=n',\ldots,n-2, \\
	x^*_{b,n-1,n} =& y^\dagger_{n-1} x^*_{c,n-1,n}+x^\dagger_{c,r_\ell-1} = y^\dagger_{n-1} x^*_{c,n-1,n}+\sum_{j=1}^{n'} \frac{u_j x_{c,r_\ell-1,j}}{\prod_{l=1}^j \sqrt{1+u_l^2}}, \\
	p_{\rho-\rho^{w_a}}(y_b^*) =& p_{\rho-\rho^{w_a}}(y_c^*) p_{\rho-\rho^{w_a}}(y^\dagger), \\
	\sigma(k_b^* w_b) =& \sigma(k_c^*) \sigma(k^\dagger w_b),
\end{align*}
when $r_\ell > 1$.
In case $r_\ell=1$, everything is the same except
\begin{align*}
	\paren{x_b^{*-1}}_{i,n'+1} =& \wtilde{u}^*_{d,i} = -\frac{u_i}{(1+u_i^2) \prod_{j=1}^{i-1} \sqrt{1+u_j^2}}, \\
	x^*_{b,n-1,n} =& u^*_{d,n-1} = \frac{u_{n-1}}{(1+u_{n-1}^2)\prod_{j=1}^{n-2}\sqrt{1+u_j^2}},
\end{align*}
and $x^*_{c,i,i+1}=0$.

For a moment, we will explicitly denote the dependence of $\rho$ on $n$ by writing $\rho_n$.
Define
\begin{align}
\label{eq:MatDecompIndw}
	\wtilde{w} =& w_{r_1,r_2,\ldots,r_\ell-1}, & \wtilde{w}_a =& \Matrix{w'\\&I_{r_\ell-1}}, & \wtilde{w}_b =& w_{n',n-1-n'} = \Matrix{&I_{n'}\\I_{n-1-n'}}
\end{align}
(dropping zero-dimensional matrices in case $r_\ell=1$), then $y_c^* \in GL(n-1,\R)$ implies that
\[ p_{\rho_n-\rho_n^{w_a}}(y_c^*) = p_{\rho_{n-1}-\rho_{n-1}^{\wtilde{w}_a}}(y_c^*), \]
and we have
\[ p_{\rho_n-\rho_n^{w_a}}(y^\dagger) \prod_{i=1}^{n'} \paren{1+u_i^2}^{(n-1-i)/2} = p_{-\rho_n^{w_a}}(y^\dagger) \prod_{i=1}^{n'} \paren{1+u_i^2}^{-1/2}. \]

Similarly, $w_a \hat{y}_{n,i}(u) w_a^{-1} = \hat{y}_{n,i^{w_a}}(u)$, so
\begin{align*}
	w_a y^\dagger w_a^{-1} =& \prod_{i=1}^{n'} \hat{y}_{n,i^{w_a}}(u_i), \\
	p_{\rho_n^{w_a}}(y^\dagger) =& \prod_{i=1}^{n'} p_{\rho_n}(\hat{y}_{n,i^{w_a}}(u_i)) = \prod_{i=1}^{n'} \paren{1+u_i^2}^{\frac{\rho_{n,n}-\rho_{n,i^{w_a}}}{2}} = \prod_{i=1}^{n'} \paren{1+u_i^2}^{\frac{i^{w_a}-n}{2}}.
\end{align*}

First, define
\[ \wbar{y} \Matrix{\wtilde{y}\\&1} := y \prod_{i=1}^{n'} \hat{y}_{n,i^{w_a}}(u_{n,i}), \]
where the lower right entry of $\wtilde{y}$ is $\wtilde{y}_{n-1,n-1} = 1$ and $\wbar{y}_i =1$ for all $i < n$.
This gives
\begin{align}
\label{eq:MatDecompIndybar}
	\wbar{y} =& \diag\paren{\wbar{y}_a,\ldots,\wbar{y}_a, \prod_{i=1}^{n'} \sqrt{1+u_{n,i}^2}}, \qquad \wbar{y}_a = y_{n-1} \times \piecewise{(1+u_{n,r_{\ell-1}}^2)^{-1/2} & \If r_\ell = 1, \\ 1 & \Otherwise,} \\
\label{eq:MatDecompIndytilde}
	\wtilde{y}_i =& y_i \prod_{j=1}^{n'} \paren{1+u_{n,j}^2}^{\eta_{n,i,n,j^{w_a}}/2}, \qquad i=1,\ldots,n-2
\end{align}

Now define
\begin{align}
\label{eq:MatDecompIndt}
	\wtilde{t} =& \paren{0,\ldots,0,t_{r_\ell+1} \tfrac{\sqrt{1+u_{n,2}^2}}{\sqrt{1+u_{n,1}^2}},\ldots,t_{n-1} \tfrac{\sqrt{1+u_{n,n'}^2}}{\sqrt{1+u_{n,n'-1}^2}}} \in \R^{n-2}, \\
\label{eq:MatDecompInda}
	\wtilde{\alpha} =& \piecewise{\paren{\frac{y_{n-1} u_1}{\sqrt{1+u_1^2}},\ldots,\frac{y_{n-1} u_{n'}}{\prod_{l=1}^{n'} \sqrt{1+u_l^2}}} \in \R^{n'} & \If r_\ell > 1, \\ \paren{\frac{y_{n'} u_{n,r_{\ell-1}+1}}{\sqrt{1+u_{n,r_{\ell-1}}^2}\prod_{i=1}^{r_{\ell-1}+1}\sqrt{1+u_{n,i}^2}},\ldots,\frac{y_{n'} u_{n,n'}}{\sqrt{1+u_{n,r_{\ell-1}}^2}\prod_{i=1}^{n'}\sqrt{1+u_{n,i}^2}}} \in \R^{n'-r_{\ell-1}} & \If r_\ell=1,}, \\
\label{eq:MatDecompIndr}
	\wtilde{r} =& \piecewise{(r_1,\ldots,r_{\ell-1},r_\ell-1) & \If r_\ell > 1, \\ (r_1,\ldots,r_{\ell-1}) & \If r_\ell=1,} \\
\label{eq:MatDecompIndf}
	\wtilde{f}_1(g) =& f_1\paren{\wbar{y}\Matrix{g\\&I_{r_\ell}}},
\end{align}
then we have our second inductive form
\begin{equation}
\label{eq:BesselInductiveForm2}
\begin{aligned}
	& \mathcal{I}_w(f_1(y\cdot),f_{2,n}(\cdot;r,t,\alpha)) = \\
	& \int_{\R^{n'}} \e{\sum_{i=1}^{n'} \alpha_i u_{n,i} \prod_{j=1}^{i-1} \sqrt{1+u_{n,j}^2}-t_{r_\ell} u_{n,1}+\delta_{r_\ell=1} \frac{y_{n'} u_{n,r_{\ell-1}}}{(1+u_{n,r_{\ell-1}}^2)\prod_{i=1}^{r_{\ell-1}-1}\sqrt{1+u_{n,i}^2}}} \\
	& \qquad \e{-\sum_{i=1}^{\ell-2} t_{n-\hat{r}_i} \frac{u_{n'-\hat{r}_i} u_{n'-\hat{r}_i+1}}{\sqrt{1+u_{n'-\hat{r}_i}^2}}-\sum_{i=1}^{\ell-1} \sum_{j=n'-\hat{r}_i+1}^{n'-\hat{r}_{i-1}-1} y_{\hat{r}_i+\hat{r}_{i-1}-n'+j} \frac{u_j u_{j+1}}{\sqrt{1+u_j^2}}} \\
	& \qquad \mathcal{I}_{\wtilde{w}}\paren{\wtilde{f}_1\paren{\wtilde{y} \cdot},f_{2,n-1}\paren{\cdot; \wtilde{r}, \wtilde{t}, \wtilde{\alpha}}} \sigma\paren{\wtilde{w}_b^{-1}k^\dagger w_b} \prod_{i=1}^{n'} \paren{1+u_{n,i}^2}^{\frac{n-i^{w_a}-1}{2}} du_{n,i}.
\end{aligned}
\end{equation}
As usual, we are using the upper left embedding of $GL(n-1,\R)$ into $G$ in the functions $\wtilde{f}_1,f_{2,n-1}$.

\section{Some results on Whittaker functions}
\label{sect:WhittFunResults}
Consider \eqref{eq:WhittDecomp}, which converges for
\[ \frac{n}{2m(n-m)}\Re(\mu_{n-m+1}+\ldots+\mu_n) > 2\Max{\max_i \abs{\Re(\mu'_i)},\max_i \abs{\Re(\mu''_i)}}. \]
The author anticipates that this integral representation will have applications, e.g. to the development of integral representations like those due to Stade \cite{Stade03}, and we explore some simple, relevant applications here.

The functional equation $\mu \mapsto \mu^w$ can be reduced to a sequence of transpositions of adjacent indices, and by two applications of \eqref{eq:WhittDecomp} (for $w_{(j\,j+1)}$, take $m=j-1$ and then on the Whittaker function at $w''$, take $m=2$), this can be reduced to the functional equation of a block diagonal matrix composed of the $GL(2)$ Whittaker functions (recall our assumptions on $\sigma$ in \cref{sect:KRepns}).
(Note that the sum $\mu_{n-m+1}+\ldots+\mu_n$ is invariant under the functional equations of both Whittaker functions in either form of \eqref{eq:WhittDecomp}.)
Therefore, $M_\sigma(w,\mu,\delta)$ is a product of $\sigma$ at fixed elements of the Weyl group and block diagonal matrices composed of the blocks \eqref{eq:GL2WhittFE}.
In particular, $M_\sigma(w,\mu,\delta)$ is holomorphic on the region described by $\abs{\Re(\mu_i-\mu_j)} < 1$ for all $i \ne j$, as this holds for each $\Gamma_{\mathcal{W}}^*(0,\mu_i-\mu_j)$.

Recalling \eqref{eq:WstarDef}, we now take $m=1,2$ in the first form of \eqref{eq:WhittDecomp}:
\begin{gather}
\begin{aligned}
\label{eq:WhitInd1}
	& W^*_{n,\sigma}(t,\mu,\delta) \abs{t_1^{n-1} t_2^{n-2}\cdots t_{n-1}^1}^{-\frac{\mu_n}{n-1}} \sgn(t_1\cdots t_{n-1})^{\delta_n} \sgn(t_n)^{\delta_1+\ldots+\delta_n} \\
	&= \int_{\R^{n-1}} W^*_{n-1,\sigma}(t' y', \mu', \delta') \sigma(\hat{k}_{n,n-1}(x_1)\cdots\hat{k}_{n,1}(x_{n-1}) w_{n-1,1}) \\
	& \qquad \times \e{-\sum_{i=1}^{n-2} t_{i+1} \frac{x_i x_{i+1}}{\sqrt{1+x_i^2}}-t_1 x_1} \prod_{j=1}^{n-1} \paren{1+x_j^2}^{-\frac{1}{2}+\frac{n}{2(n-1)}\mu_n} dx,
\end{aligned}\\
\begin{aligned}
	t' =& (t_2,\ldots,t_{n-1},1), & y' =& (y^*_1,\ldots,y^*_{n-2},1), & y^*_i =& \prod_{k=1}^{n-1} (1+x_k^2)^{\eta_{n,i,k}/2}, \\
	\delta' =& (\delta_1,\ldots,\delta_{n-1}), & \mu' =& (\mu_1,\ldots,\mu_{n-1})+\tfrac{\mu_n}{n-1}(1,\ldots,1),
\end{aligned}\nonumber
\end{gather}
\begin{gather}
\begin{aligned}
\label{eq:WhitInd2}
	& W^*_{n,\sigma}(t,\mu,\delta) \abs{t_1}^{-\frac{\mu_{n-1}+\mu_n}{2}} \abs{t_2^{n-2}\cdots t_{n-1}^1}^{-\frac{\mu_{n-1}+\mu_n}{n-2}} \sgn(t_2\cdots t_{n-1})^{\delta_{n-1}+\delta_n} \sgn(t_n)^{\delta_1+\ldots+\delta_n} \\
	&= \int_{\R^{2\times(n-2)}} W^*_{n-2,\sigma}(t' y', \mu', \delta') \sigma(w_{n-2,2}) W^*_{2,\sigma}((t_1 y^*_{n-1},1), \mu'', (\delta_{n-1},\delta_n)) \\
	& \qquad \times \sigma(w_{n-2,2}^{-1} \hat{k}_{n-1,n-2}(x_{1,1})\cdots\hat{k}_{n-1,1}(x_{1,n-2})\cdots\hat{k}_{n,n-2}(x_{2,1})\cdots\hat{k}_{n,1}(x_{2,n-2}) w_{n-2,2}) \\
	& \qquad \times \e{t_1 \sum_{i=1}^{n-2} \tfrac{x_{1,i} x_{2,i}}{\sqrt{1+x_{2,i}^2}}\prod_{j=1}^{i-1}\tfrac{\sqrt{1+x_{1,j}^2}}{\sqrt{1+x_{2,j}^2}}-\sum_{i=1}^{n-3} t_{i+2} \tfrac{x_{1,i} x_{1,i+1}\sqrt{1+x_{2,i+1}^2}}{\sqrt{1+x_{1,i}^2}\sqrt{1+x_{2,i}^2}}-\sum_{i=1}^{n-3} t_{i+2}\tfrac{x_{2,i} x_{2,i+1}}{\sqrt{1+x_{2,i}^2}} -t_2 x_{2,1}} \\
	& \qquad \times \prod_{i=1}^2 \prod_{j=1}^{n-2} \paren{1+x_{i,j}^2}^{-\frac{1}{2}+\frac{n(\mu_{n-1}+\mu_n)}{4(n-2)}} dx,
\end{aligned}\\
\begin{aligned}
	t' =& (t_3,\ldots,t_{n-1},1), & y' =& (y^*_1,\ldots,y^*_{n-3},1), \\
	y_i^* =& \prod_{j=1}^2 \prod_{k=1}^{n-2} (1+x_{j,k}^2)^{\eta_{n-2+j,i,k}/2}, & \delta' =& (\delta_1,\ldots,\delta_{n-2}), \\
	\mu' =& (\mu_1,\ldots,\mu_{n-2})+\tfrac{\mu_{n-1}+\mu_n}{n-2}(1,\ldots,1), & \mu'' =& (\mu_{n-1},\mu_n)-\tfrac{\mu_{n-1}+\mu_n}{2}(1,1).
\end{aligned}\nonumber
\end{gather}
In a neighborhood of $i\mathfrak{a}^*_0(\Lambda)$, the first integral representation is best in the case $\Lambda_n=0$, while the second is better for $\Lambda_n \ne 0$ since then $\Lambda_{n-1}+\Lambda_n=0$ so $\Re(\mu_{n-1}+\mu_n)$ is close to $0$.
Note that Stade and Ishii have given decompositions in \cite[Theorem 2.1]{Stade03} and \cite[Theorem 14]{IshiiStade} in the spherical case (in the latter paper by passing through the Mellin transform); it would be interesting to see if one can arrive at their formula by starting with this one.
This also recovers the integral representations of \cite[(3.22), (3.24)]{HWI}.

We use these expressions, starting with \eqref{eq:WhitInd1}, to obtain expressions for the Mellin transform of the Whittaker function.
Since we can conjugate by an element of the Weyl group to transform $\hat{k}_{n,i}(x)$ back to an element of $SO(2)$, our assumptions on $\sigma$ (see \cref{sect:KRepns}) imply that the entries of $\sigma(\hat{k}_{n,i}(x))$ are trigonometric polynomials in $\arg \frac{1-ix}{\sqrt{1+x^2}}$.
So we replace each $\sigma(\hat{k}_{n,n-j}(x_j))$ with $\paren{\frac{1-ix_j}{\sqrt{1+x_j^2}}}^{a_j}$ for some $a_j \in \Z$ and replace the Whittaker function $W^*_{n-1,\sigma}$ with its entry $W^*_{n-1,\sigma_{b_1},b_2,b_3}$ where $\sigma_{b_1}$ is an irreducible component of $\sigma|_{O(n-1)}$ and $1 \le b_2, b_3 \le \dim \sigma_{b_1}$.
Next, Mellin expand the $GL(n-1)$ Whittaker function and apply \cite[(6.1)]{MeThesis}
\begin{align*}
	\e{x} &= \lim_{\theta \to \frac{\pi}{2}^-} \exp\paren{-\abs{2\pi x} e^{-i\theta\sgn(x)}} = \lim_{\theta\to\frac{\pi}{2}^-} \int_{\Re(s) = 1} \abs{2\pi x}^{-s} e^{is\theta \sgn(x)} \Gamma\paren{s} \, \frac{ds}{2\pi i}, \qquad x \ne 0,
\end{align*}
to each term in the complex exponential, then use \cite[Section 2.3.2]{HWI}
\begin{gather*}
	\int_{\R} \sgn(x)^d \abs{x}^{a-1} (1+x^2)^{-\frac{a+b}{2}} \paren{\frac{1-ix}{\sqrt{1+x^2}}}^c dx = B_{d,c}(a,b), \qquad d \in \set{0,1} \\
\begin{aligned}
	B_{d,c}(a,b) =& (i \sgn(c))^d \sum_{j=0}^{(\abs{c}-d)/2} \binom{\abs{c}}{2j+d} (-1)^j B\paren{\frac{d+a}{2}+j,\frac{\abs{c}-d+b}{2}-j}, \\
	B(a,b) =& \frac{\Gamma(a)\Gamma(b)}{\Gamma(a+b)}, \qquad e^{it\theta \sgn(x)} = \sum_{d \in \set{0,1}} \cos\paren{d \frac{\pi}{2}-\theta t} \sgn(x)^d,
\end{aligned}
\end{gather*}
(which is the Mellin transform of the classical Whittaker function, compare \eqref{eq:classWhittDef}) and since the resulting integral converges absolutely, we may replace $\theta$ with $\frac{\pi}{2}$ in the integrand.
We conclude that the entries of the Mellin transform $\what{W}_{n,\sigma}(s,\mu,\delta)$ of the Whittaker function (as in \eqref{eq:WhittMellinDef}) are linear combinations (with coefficients of modulus at most 1) of at most $(\dim \sigma)^{2n-1}$ terms of the form
\begin{equation}
\label{eq:WhittMellin1}
\begin{aligned}
	& \sum_{d \in \set{0,1}^{n-1}} \int_{\Re(q)=\mathfrak{q}} \what{W}_{n-1,\sigma_{b_1}, b_2, b_3}(q,\mu',\delta') \prod_{j=1}^{n-1} A_{d_j,+}\paren{s_j-q_{j-1}+\tfrac{n-j}{n-1}\mu_n} \\
	& \times B_{\tilde{d}_j,a_j}\paren{1-s_j-s_{j+1}+q_{j-1}+q_j-\tfrac{2n-2j-1}{n-1}\mu_n,s_j-q_j-\tfrac{j}{n-1}\mu_n} \frac{dq}{(2\pi i)^{n-2}},
\end{aligned}
\end{equation}
where
\[ A_{d,\pm}(u) := (2\pi)^{-u} \Gamma(u) \cos \frac{\pi}{2}(d\pm u), \]
 $\set{0,1} \ni \tilde{d}_j \equiv d_j+d_{j+1} \pmod{2}$ using $s_n=q_0=q_{n-1}=d_n=0$.

By Stirling's formula, the beta function $B(u+v,u-v)$ has exponential decay in $\abs{\Im(v)}$ unless $\abs{\Im(v)} < \abs{\Im(u)}$, and the same holds for $B_{d,c}(u+v,u-v)$.
Suppose the Mellin transform of the lower-rank Whittaker function forces $\norm{\Im(q)} \ll_n \norm{\Im(\mu')}$, then the beta functions force $\norm{\Im(s)} \ll_n \norm{\Im(\mu)}$ and by induction this holds for all $n$ since it holds for $n=2$.

In \eqref{eq:WhittMellin1}, consider shifting the $q$ contours far to the right, starting with $q_1$ and ending with $q_{n-2}$, we see that $\what{W}_{n,\sigma}(s,\mu,\delta)$ is holomorphic in $\mu$ and $s$ for $\Re(s)$ large compared to $\Re(\mu)$, so $W^*_{n\sigma}(t,\mu,\delta)$ has super-polynomial decay as the coordinates of $t$ tend to infinity.
Furthermore, the potential poles of $\what{W}_{n,\sigma}(s,\mu,\delta)$ occur when $s=-\hat{\mu}^w-m$ for some $w \in W$, $m \in \N_0^{n-1}$ (imagine shifting the $q$ contours far to the right, picking up residues at $q_{j-1} = s_j + \mu_n-\frac{j-1}{n-1}\mu_n+m_j$ and $q_j = s_j-\frac{j}{n-1}\mu_n+m_j$, $m_j \in \N_0$), and by contour shifting this implies
\begin{thm}[Inverse Whittaker Paley-Wiener]
\label{thm:IWPW}
If $F(\mu)$ is holomorphic and has rapid decay in $\norm{\mu}$ on a tube domain containing $i\mathfrak{a}^*_0(\Lambda)$, then there exists $\eta,\eta',A,B>0$ so that
\[ \int_{i\mathfrak{a}^*_0(\Lambda)} F(\mu) W_\sigma(y,\mu,\delta) \dspec\mu \ll_{\eta,\eta',A,B} p_{(1+\eta')\rho}(y) (\dim \sigma)^B \norm{F}_{\eta,A}, \]
where
\[ \norm{F}_{\eta,A} := \sup_{\norm{\beta}<\eta} \int_{\Re(\mu)=\beta} \abs{F(\mu)} (1+\norm{\mu})^A \abs{d\mu}. \]
\end{thm}

Applying the same process to \eqref{eq:WhitInd2} yields
\begin{equation}
\label{eq:WhittMellin2}
\begin{aligned}
	& \sum_{d \in \set{0,1}^{n-2}} \sum_{d' \in \set{0,1}^{n-3}} \sum_{d'' \in \set{0,1}^{n-2}} \int_{\Re(u)=\epsilon} \what{W}_{n-2,\sigma_{b_1},b_2,b_3}(u,\mu',\delta') \int_{\Re(u')=\epsilon} \what{W}_{2,\sigma_{b'}}(u',\mu'',(\delta_{n-1},\delta_n)) \\
	& \int_{\Re(q) = \epsilon (n-3,\ldots,1)} \int_{\Re(q') = \epsilon} A_{d_1,-}\paren{s_1-q_1-u'+\tfrac{\mu_{n-1}+\mu_n}{2}} A_{d''_1,+}(s_2+\mu_{n-1}+\mu_n) \\
	& B_{\tilde{d}_{1,1},a_{1,1}}\bigl(1-s_1-s_3+q_1+q'_1+u_1+u'-\tfrac{3n-8}{2(n-2)}(\mu_{n-1}+\mu_n), s_1-u_1-\tfrac{1}{n-2}(\mu_{n-1}+\mu_n)\bigr) \\
	& B_{\tilde{d}_{2,1},a_{2,1}}\bigl(1-s_1-s_2+q_1-q'_1+u'-\tfrac{3}{2}(\mu_{n-1}+\mu_n), s_2-s_3-q_1+q'_1-u'-\tfrac{1}{2}(\mu_{n-1}+\mu_n) \bigr) \\
	& \prod_{j=1}^{n-3} A_{d_{j+1},-}(q_j-q_{j+1}) A_{d''_{j+1},+}(q'_j) A_{d'_j,+}\paren{s_{j+2}-q'_j-u_j+\tfrac{n-2-j}{n-2}(\mu_{n-1}+\mu_n)} \\
	& \prod_{j=2}^{n-2} B_{\tilde{d}_{1,j},a_{1,j}}\bigl(1-s_{j+1}-s_{j+2}-q_{j-1}+q_j+q'_{j-1}+q'_j+u_{j-1}+u_j-\tfrac{2n-3-2j}{n-2}(\mu_{n-1}+\mu_n), \\
	& \qquad \qquad \qquad s_{j+1}+q_{j-1}-q'_{j-1}-u_j+u'+\tfrac{n-2j-2}{2(n-2)}(\mu_{n-1}+\mu_n)\bigr) \\
	& B_{\tilde{d}_{2,j},a_{2,j}}\bigl(1-q_{j-1}+q_j-q'_{j-1}-q'_j, s_{j+1}-s_{j+2}-q_j+q'_j-u'-\tfrac{1}{2}(\mu_{n-1}+\mu_n)\bigr) \\
	& \frac{dq'}{(2\pi i)^{n-3}} \frac{dq}{(2\pi i)^{n-3}}\frac{du'}{2\pi i} \frac{du}{(2\pi i)^{n-3}},
\end{aligned}
\end{equation}
where $\set{0,1}\ni\tilde{d}_{1,j} \equiv d_j+d'_j+d'_{j-1} \pmod{2}$, $\set{0,1} \ni\tilde{d}_{2,j} \equiv d_j+d''_j+d''_{j+1}\pmod{2}$ using $q_{n-2}=q'_{n-2}=d'_0=d''_{n-1}=0$.
Convergence of this integral is somewhat more difficult as the exponential decay factors coming from the beta functions and Mellin transforms of the Whittaker functions are not sufficient, but the polynomial part of Stirling's formula does give enough decay.

\subsection{Jacquet integrals and limits}
For the \Asymps, we will need to evaluate the Jacquet integral of a limit of a Whittaker function, and we perform that computation here.
Throughout this section, the letters $a,b$ as subscripts are used to distinguish matrices of a similar type, not as indices.

For $y \in Y_w$, define $\alpha_w \in \R^{n-1}$ by $\psi_{\alpha_w} = \lim_{y \to 0} \psi_y$, where the limit is taken along the non-constant entries of $y$, so $\alpha_{w,i}$ is 0 if $i=\hat{r}_j$ for some $j$ and 1 otherwise.

Suppose $\Re(\mu_i) > \Re(\mu_j)$ for all $i < j$, then for $y \in Y_w$,
\begin{equation}
\label{eq:WhittYwAsymp}
\begin{aligned}
	W_\sigma(y g,\mu,\delta) =& p_{\rho+\mu^{w_l}}(y) \chi_{\delta^{w_l}}(y) W_\sigma(g,w_l,\mu,\delta,\psi_y) \\
	\sim& p_{\rho+\mu^{w_l}}(y) \chi_{\delta^{w_l}}(y) W_\sigma(g,w_l,\mu,\delta,\psi_{\alpha_w}), \qquad \text{ as } y \to 0.
\end{aligned}
\end{equation}
Then we decompose $U=\wbar{U}_{w^{-1}} U_{w^{-1}}$ so that
\begin{align*}
	W_\sigma(w g,w_l,\mu,\delta,\psi_{\alpha_w}) =& \int_{U(\R)} I_{\mu,\delta,\sigma}(w_l u w g) \wbar{\psi_{\alpha_w}(u)} du \\
	=& \int_{U_{w^{-1}}(\R)} \int_{\wbar{U}_{w^{-1}}(\R)} I_{\mu,\delta,\sigma}(w_l u_a u_b w g) \wbar{\psi_I(u_b)} du_a \, du_b,
\end{align*}
since $\psi_{\alpha_w}(u_a u_b) = \psi_I(u_b)$.
Then conjugating $w^{-1} u_b w \mapsto u_b$ with $\psi_I(w^{-1} u_b w) = \psi_I(u_b)$ gives
\begin{align*}
	W_\sigma(w g,w_l,\mu,\delta,\psi_{\alpha_w}) =& \int_{U_w(\R)} \int_{\wbar{U}_{w^{-1}}(\R)} I_{\mu,\delta,\sigma}(w_l u_a w u_b g) \wbar{\psi_I(u_b)} du_a \, du_b.
\end{align*}

Now write $w_l u_b g = u^* y^* k^*$, then
\begin{align*}
	W_\sigma(w g,w_l,\mu,\delta,\psi_{\alpha_w}) =& \int_{U_w(\R)} \int_{\wbar{U}_{w^{-1}}(\R)} I_{\mu,\delta,\sigma}(w^{-1} u_a u^* y^*) \wbar{\psi_I(u_b)} \sigma(k^*) du_a \, du_b \\
	=& \int_{U_w(\R)} W_\sigma(u^* y^*, w^{-1}, \mu, \delta, \psi_0) \wbar{\psi_I(u_b)} \sigma(k^*) du_b,
\end{align*}
using $w_l w^{-1} \wbar{U}_{w^{-1}} w w_l = \wbar{U}_{w^{-1}}$ and $w_l w w_l=w^{-1}$.
Using the known properties, e.g. \eqref{eq:WhittArgToCharFE}, of the degenerate Whittaker function,
\begin{align*}
	W_\sigma(w g,w_l, \mu,\delta,\psi_{\alpha_w}) =& W_\sigma(I, w^{-1}, \mu, \delta, \psi_0) \int_{U_w(\R)} I_{\mu^{w^{-1}},\delta^{w^{-1}},\sigma}(w_l u_b g) \wbar{\psi_I(u_b)} du_b.
\end{align*}
Since the original integral converged absolutely, we know the $u_b$ integral converges absolutely on $\Re(\mu)=\epsilon\rho$, even at $\mu^{w^{-1}}$, and by analytic continuation, this equality holds wherever the $u_b$ integral converges.
In fact, the $u_b$ integral can be expressed in terms of lower-rank, non-degenerate Whittaker functions, and this is entire in $\mu$.
As in \cref{sect:MatrixDecomps}, the fully degenerate Whittaker function can be evaluated and, up to matrix factors of the form $\sigma(w), w \in W$, is given by a product of sums of quotients of gamma functions
\[ \int_{\R} (1+x^2)^{-\frac{1+\mu_i-\mu_j}{2}} \paren{\frac{1-ix}{\sqrt{1+x^2}}}^d dx = \mathcal{W}(0,\mu_i-\mu_j,d), \]
as in \eqref{eq:classWhittDef}, and the possible poles occur at $\mu_i-\mu_j \in -\N_0$ for some $i < j$.

Hence
\begin{align}
\label{eq:LimitJacquetInt}
	\int_{\wbar{U}_w(\R)} W_\sigma(w x t,w_l, \mu,\delta,\psi_{\alpha_w}) \wbar{\psi_I(x)} dx =& W_\sigma(I, w^{-1}, \mu, \delta, \psi_0) W_\sigma(t,\mu^{w^{-1}},\delta^{w^{-1}}),
\end{align}
and this holds wherever the $x$ integral converges, in particular, on the region $\Re(\mu^{w^{-1}}_i)>\Re(\mu^{w^{-1}}_j)$ for all $i < j$.

\section{The Strong Interchange of Integrals}
\label{sect:StrongIoI}

For the \StrongIoIC, we start by taking
\[ f_1(g) = \int_{i \mathfrak{a}_0^*(\Lambda)} f(\mu) W_\sigma(yg,\mu,\delta) \dspec\mu, \]
with $f(\mu)$ as in the \IoIC and $f_2(x) = \wbar{\psi_t(x)}$ in \eqref{eq:BesselInductiveForm2}, which is effectively a choice of coordinates on $\wbar{U}_w(\R)$ where $w=w_{r_1,\ldots,r_\ell}$.
Note that compared to the \IoIC, we have conjugated $x \mapsto txt^{-1}$ and substituted $ywtw^{-1} \mapsto y$.

We can also address the convergence using \eqref{eq:BesselInductiveForm2}:
The final argument of the Whittaker function is
\[ Y^+(ywx) = y\prod_{k=1}^{\ell-1} \prod_{m=\hat{r}_k+1}^{\hat{r}_{k+1}} \prod_{i=1}^{\hat{r}_k} \hat{y}_{m,i^{\hat{w}_k}}(u_{m,i}), \qquad \hat{w}_k = w_{r_1,\ldots,r_k}, \]
and modelling $W_\sigma(y,\mu,\delta)$ as roughly $p_\rho(y)$ for small $y$ (with, say, $\Re(\mu)=0$; the general case tends to be worse) means that without the extra decay coming from the $\mu$ integral, each $u_{m,j}$ integral would converge like
\[ \int_{\R} p_\rho\paren{\hat{y}_{m,j^{\hat{w}_k}}(u_{m,j})} \paren{1+u_{m,j}^2}^{\frac{m-j^{\hat{w}_k}-1}{2}} du_{m,j} = \int_{\R} \paren{1+u_{m,j}^2}^{-\frac{1}{2}} du_{m,j} = \infty. \]
On the other hand, the slight decay given by \cref{thm:IWPW} means the integrals all converge absolutely.

Since the $\wbar{U}_w(\R)$ integral converges absolutely (outside the $\mu$ integral), we can apply a smooth, dyadic partition of unity as follows:
Let $g_1,g_2:\R\to[0,1]$ be even and smooth with $g_1$ supported on $[-4,4]$ and $g_2$ supported on $[-4,-1]\cup[1,4]$ such that
\[ g_1(x) + \sum_{i=1}^\infty g_2\paren{x/2^i} = 1 \]
identically on $\R$.
We apply the partition of unity to each $u_{m,j}$ coordinate, taking $C_{m,j}=2^i$ for some $i \in \N$ or $C_{m,j}=1$ for the $g_1$ term.
We set $\norm{C} = \max_{m,j} C_{m,j}$ and pull the (now compactly supported) $u_{m,j}$ integrals inside the $\mu$ integral.

From the matrix decompositions and in particular \cref{lem:RowIwa}, we see that $\sigma(K(wx))$ is a product of $\sigma$ at some fixed Weyl elements and $\sigma\paren{h_1\paren{\arg\frac{1-iu_{m,j}}{\sqrt{1+u_{m,j}^2}}}}$, which we have assumed are diagonalized, so that the entries are just powers $\paren{\frac{1-iu_{m,j}}{\sqrt{1+u_{m,j}^2}}}^{a_{m,j}}$, $a_{m,j} \in \Z$ with $\abs{a_{m,j}} \le a_\sigma$ for some $a_\sigma \in \N_0$, the highest weight character occurring in the restriction of $\sigma$ to $SO(2)$.
(For $n \ge 3$, we have $a_\sigma \le \dim \sigma$, but for $n=2$, $\dim \sigma \le 2$ while $a_\sigma$ can be arbitrarily large.)

Lastly, the Whittaker function has super-polynomial (actually exponential) decay in $\norm{C}$ unless each coordinate of $Y^+(wx)$ is $\ll_\mu \norm{C}^\epsilon$, with polynomial dependence on $1+\norm{\mu}$.
We record this bound on the $C_{m,j}$ variables and Mellin expand the Whittaker function.
Recalling \eqref{eq:hatyYCoords}, the $i$-th $Y$-coordinate of $\hat{y}_{m,j^{\hat{w}_k}}(u_{m,j})$ is $\paren{1+u_{m,j}^2}^{\eta_{n,i,m,j^{\hat{w}_k}}/2}$, and then
\[ b(s)_{m,j} = -\frac{1}{2} \sum_{i=1}^{n-1} \eta_{n,i,m,j^{\hat{w}_k}} s_i-\frac{1}{2} \]
is the power of $(1+u_{m,j}^2)$ resulting from the Mellin expansion of the Whittaker function.
To simplify notation we define
\[ g_{m,j,s}(u_{m,j})=\frac{1-iu_{m,j}}{\sqrt{1+u_{m,j}^2}}^{a_{m,j}} (1+u_{m,j}^2)^{b(s)_{m,j}} \times \piecewise{g_1(x) & \If C_{m,j}=1, \\ g_2(x/C_{m,j}) & \Otherwise.} \]
Note that
\[ \frac{d^k}{du_{m,j}^k} g_{m,j,s}(u_{m,j}) \ll_k \paren{a_\sigma+\norm{s}}^k C_{m,j}^{-k} \]
and $\sqrt{1+u_{m,j}^2} \asymp C_{m,j}$ on the support of $g_{m,j}(u_{m,j})$.

We inductively define $\mathcal{I}_n(r,a,s,C,y,t,\alpha)$ by $\mathcal{I}_n((n),a,s,C,y,t,\alpha) = 1$ and for $r=(r_1,\ldots,r_\ell)$ a non-trivial (i.e. $\ell > 1$) partition of $n$,
\begin{equation}
\label{eq:StrongIoI}
\begin{aligned}
	& \mathcal{I}_n(r,a,s,C,y,t,\alpha) = \\
	& \int_{\R^{n-r_1}} \e{\sum_{i=1}^{n-r_1} \alpha_i u_{n,i} \prod_{j=1}^{i-1} \sqrt{1+u_{n,j}^2}-t_{n-r_1} u_{n,1}+\delta_{r_1=1} \frac{y_{r_1} u_{n,r_2}}{(1+u_{n,r_2}^2)\prod_{i=1}^{r_2-1}\sqrt{1+u_{n,i}^2}}} \\
	& \qquad \e{-\sum_{i=2}^{\ell-1} t_{n-\hat{r}_i} \frac{u_{n,\hat{r}_i-r_1} u_{n,\hat{r}_i-r_1+1}}{\sqrt{1+u_{n,\hat{r}_i-r_1}^2}} - \sum_{i=2}^\ell \sum_{j=1}^{r_i-1} y_{\hat{r}_{i-1}+j} \frac{u_{n,\hat{r}_i-r_1-j} u_{n,\hat{r}_i-r_1-j+1}}{\sqrt{1+u_{n,\hat{r}_i-r_1-j}^2}}} \\
	& \qquad \mathcal{I}_{n-1}\paren{\wtilde{r}, a, s, C, \wtilde{y}, \wtilde{t}, \wtilde{\alpha}} \prod_{j=1}^{n-r_1} g_{n,j,s}(u_{n,j}) du_{n,j},
\end{aligned}
\end{equation}
with $\wtilde{r}, \wtilde{y}, \wtilde{t}, \wtilde{\alpha}$ as in \eqref{eq:MatDecompIndybar}-\eqref{eq:MatDecompIndr}.

Finally, we have:
\begin{conj*}[Strong Interchange of Integrals]
\hypertarget{conj:StrongIoI}{}
Fix $r$ and the values of $a_{m,j} \in \Z$ for $j=1,\ldots,n-\hat{r}_k$, $n-\hat{r}_k+1 \le m \le n-\hat{r}_{k-1}$, $k=1,\ldots,\ell$ and $\Re(s_i)$ for $i=1,\ldots,n-1$.
Then for tuples $(C_{m,j})_{m,j}$, $\log_2 C_{m,j} \in \N_0$ satisfying
\[ \prod_{k=1}^{\ell-1} \prod_{m=\hat{r}_k+1}^{\hat{r}_{k+1}} \prod_{j=1}^{\hat{r}_k} C_{m,j}^{\eta_{n,i,m,j^{\hat{w}_k}}} \ll \norm{C}^\epsilon, \qquad i=1,\ldots,n-1, \]
the integral $\mathcal{I}_n(r,a,s,C,y,t,0)$ has super-polynomial decay in $\norm{C}$ for $\norm{C}$ larger than some fixed power of the norm of $(y,t,y^{-1},t^{-1},s, a_\sigma)$.
\end{conj*}
It is important to note that we are not conjecturing that the inner integral $\mathcal{I}\paren{\wtilde{r}, a, s, C, \wtilde{y}, \wtilde{t}, \wtilde{\alpha}}$ has super-polynomial decay in the corresponding coordinates of $C$; having nonzero $\alpha$ will interfere.

\begin{proof}[Proof of \cref{prop:StrongIoIConsequences}]
The \IoIC follows by applying the decomposition (of the $\wbar{U}_w(\R)$ integral into linear combinations of $\mathcal{I}(r,a,s,C,y,t,\alpha)$, where $w=w_{r_1,\ldots,r_\ell}$) described above to get something of the form
\begin{align*}
	&p_{\rho-\rho^w}(t) \wtilde{K}_w(ywt^{-1}w^{-1},t,\Lambda,\mu,\delta,\sigma)_{i,j} = \\
	& \lim_{R \to \infty} \sum_{\norm{C} \le R} \sum_{a,b,i',j'} d_{i,j,a,b,i',j'} \int_{\Re(s)=\mathfrak{s}} \what{W}_{n,\sigma}(s,\mu,\delta)_{i',j'} \hat{p}_{-s}(y) \mathcal{I}_n(r,a,s,C,y,t,0) \frac{ds}{(2\pi i)^{n-1}},
\end{align*}
for appropriate coefficients $d_{i,j,a,b,i',j'} \in \C$ and $\mathfrak{s}$ is sufficiently large compared to $\Re(\mu)$.
The \StrongIoIC implies we can freely interchange the $R$ limit with the $\mu$ in the \IoIC.
Note that differentiating $\mathcal{I}(r,a,s,C,y,t,\alpha)$ in $y$, $t$ and $s$ produces linear combinations of integrals of the same form having different values for $a$, $b$ and $d_{i,j,a,b,i',j'}$, while the trivial bound on
\[ \mathcal{I}_n(r,a,s,C,y,t,0) \]
is clearly polynomial in $s$, hence $\wtilde{K}_w(y,t,\Lambda,\mu,\delta,\sigma)$ is polynomially bounded in $\norm{\mu}$.

Furthermore, the \AnContC follows because the Whittaker function is entire in $\mu$ and our construction of $\wtilde{K}_w(y,t,\Lambda,\mu,\delta,\sigma)$ above is visibly independent of $\Lambda$.
\end{proof}

In the integral for the Bessel function, taking $y=0$ (in the complex exponential), $w=w_l$, and replacing the Whittaker function with the power function gives the Jacquet integral for the Whittaker function, hence
\begin{conj*}[Jacquet-Whittaker Direct Continuation]
\hypertarget{conj:JWDC}{}
Fix the values of $a_{m,j} \in \Z$ for $j=1,\ldots,n-\hat{r}_k$, $n-\hat{r}_k+1 \le m \le n-\hat{r}_{k-1}$, $k=1,\ldots,\ell$.
Suppose $\abs{\Re(\mu_i)} \le \delta$ for $i=1,\ldots,n-1$ and some fixed $\delta > 0$.
Then for tuples $(C_{m,j})_{m,j}$, $\log_2 C_{m,j} \in \N_0$ and some fixed $\epsilon > 0$, we have
\[ \mathcal{I}_n((1,\ldots,1),a,\hat{\mu},C,0,t,0) \ll \norm{C}^{-\epsilon} \]
for $\norm{C}$ larger than some fixed power of the norm of $(t,t^{-1},\mu, a_\sigma)$.
\end{conj*}

\section{The Asymptotic Analysis}
\label{sect:Asymptotics}
Assume the \IoI, \AnCont and \DEPS Conjectures; we want to prove the \Asymps.
The Weyl invariant of part 1 follows by \cref{prop:IoIConsequences2}, the $W_w$-invariance of $J_w(y,\mu,\delta)$ and the holomorphy of both $K_w(y,\mu,\delta)$ and $J_w(y,\mu,\delta)$; the meromorphy follows from part 2.
For the remaining parts, we first prove the informal Asymptotics Conjecture; i.e. we derive the relation between the asymptotics of the Bessel functions and the Whittaker functions.
This is primarily \cref{lem:InformalAC}, below.
Then we determine the coefficients in terms of Shahidi's local coefficients in \cref{sect:CwComp}.
We generally view $GL(n-j) \subset GL(n)$ by the upper-left embedding.

\subsection{Proof of the Informal Asymptotics Conjecture}
Let $w=w_{r_1,\ldots,r_\ell} \in W^\text{rel}$ where $r_1+\ldots+r_\ell=n$.
For convenience of notation, we extend (for this section only) by the Iwasawa decomposition for $x \in U(\R)$, $y \in Y^+$, $k \in K$, and any character $\psi$ of $U(\R)$,
\[ \psi(xyk) = \psi(x), \quad p_{S,\mu}(xyk) = p_{S,\mu}(y), \quad \hat{p}_{S^c,-s_{S^c}}(xyk) = \hat{p}_{S^c,-s_{S^c}}(y), \quad \sigma(xyk) = \sigma(k), \]
recalling $p_{S,\mu}$ and $\hat{p}_{S^c,-s_{S^c}}$ from \cref{sect:WhittakerMellin}.

Define a test function $f:M_n \to \C$ by
\[ f(\mu)=f_{\mathcal{S}}(\mu) f_{\mathcal{S}^c}(\mu), \qquad f_{\mathcal{S}^c}(\mu) := \prod_{i \in \mathcal{S}^c} f_0(\mu_{n-i}), \qquad f_{\mathcal{S}}(\mu) := \prod_{i \in \mathcal{S}} f_1(\hat{\rho}_i,\hat{\mu}_i), \]
where $\mathcal{S}=\set{\hat{r}_i \setdiv i=1,\ldots,\ell}$,
\[ f_0(t) = 2\sqrt{\pi} R_0 e^{R_0^2 t^2}, \qquad f_1(u,t) = \frac{R_1 e^{t^2}}{u-R_1 t}, \qquad R_0 > 1, \qquad R_1 > 4+\frac{1}{\epsilon}. \]

Let $\beta \in M_n^0$, then we apply the \IoIC to the test function $f(\mu-\beta)$, we have
\begin{align*}
	\mathcal{I} :=& \int_{\wbar{U}_w(\R)} \int_{M_n^0} f(\mu) W_\sigma(y w x y', \mu+\beta,\delta) \frac{\dspec\mu}{(2\pi i)^{n-1}} \wbar{\psi_I(x)} dx \\
	&= \int_{M_n^0} f(\mu) K_w(y,\mu+\beta,\delta) W_\sigma(y', \mu+\beta,\delta) \frac{\dspec\mu}{(2\pi i)^{n-1}}.
\end{align*}
We will compare asymptotics as $R_1 \to \infty$ with
\[ y_j = \sgn(y_j) e^{-R_1^2}, \qquad j \in \mathcal{S} \]
in the two expressions for $\mathcal{I}$, and then take the limit as $R_0 \to \infty$.

\begin{lem}
\label{lem:InformalAC}
Let
\[ T = p_{(1+R_1^{-1})\rho}(y) \prod_{i\in\mathcal{S}} y_i^{\hat{\beta}_i} \sgn(y_i)^{\hat{\delta}_i}, \]
the leading term as $R_1 \to \infty$ is proportional to $T$.
More precisely,
\begin{align*}
	T^{-1} \mathcal{I} =& \int_{M_{\mathcal{S}}^0} f_{\mathcal{S}^c}(\mu) C_w^*(\mu+\beta,\delta) W_\sigma(y', \mu+\beta,\delta) \frac{\dspec\mu}{(2\pi i)^{n-\ell}} + \BigO{R_1^{-1}} \\
	=& \int_{M_{\mathcal{S}}^0} f_{\mathcal{S}^c}(\mu) \wtilde{C}_w(\mu+\beta, \delta) W_\sigma(y', \mu+\beta, \delta) \frac{\dspec\mu}{(2\pi i)^{n-\ell}}+\BigO{R_1^{-1}},
\end{align*}
where $\wtilde{C}_w(\beta, \delta)$ is the scalar determined by
\begin{align}
\label{eq:CtildeDef}
	W_\sigma(I, w^{-1}, \mu^{w_l}, \delta^{w_l},\psi_0) = \wtilde{C}_w(\mu, \delta) M_\sigma(w^{-1},\mu^{w_l}, \delta^{w_l}).
\end{align}
\end{lem}
The lemma implies the asymptotics as follows:
Equating the main terms, we may then send $R_0 \to \infty$ and divide by $W_\sigma(y', \beta, \delta)$ (which is not identically zero for $y'$ sufficiently small) to conclude
\begin{align*}
	C^*_w(\beta,\delta) = \wtilde{C}_w(\beta, \delta),
\end{align*}
for all $\delta$ and by analytic continuation all $\beta \in M_n$.

We will now proceed to prove the lemma, but first a simple lemma:
\begin{lem}
\label{lem:IwaBounds}
Let $x \in \wbar{U}_w(\R)$.
If we write the Iwasawa decomposition of $w x$ as $wx=x^* y^* k^*$ then
\begin{enumerate}
\item for $j=1,\ldots,n-1$, we have $x^*_{j,j+1} \ll_n p_{-\rho}(y^*)$,
\item for any $\alpha \in [-1,1]^{n-1}$, we have $y_1^{* \alpha_1} \cdots y_{n-1}^{* \alpha_{n-1}} \le p_{-2\rho}(y^*)$.
\end{enumerate}
\end{lem}
\begin{proof}
We may use the coordinates on $x$ defined by the substitutions throughout \cref{sect:MatrixDecomps}, and we note that in those coordinates $y^*$ is a product of the matrices $\hat{y}_{m,i}(x)$, $i<m\le n$ of \cref{lem:RowIwa}.
So it's enough to prove part 2 for such a $\hat{y}_{m,i}(x)$, and we see that
\[ p_{-\rho}\paren{\hat{y}_{m,i}(x)} = \paren{1+x^2}^{\frac{\rho_i-\rho_m}{2}} = \paren{1+x^2}^{\frac{m-i}{2}} \ge \sqrt{1+x^2}, \]
while (recall \eqref{eq:hatyYCoords}) the $j$-th $Y$-coordinate of $\hat{y}_{m,i}(x)$ is $\paren{1+x^2}^{\eta_{n,j,m,i}/2}$ and $\abs{\eta_{n,j,m,i}} \le 2$.

Similarly, $x^*_{j,j+1}$ is a sum of terms appearing in the complex exponential in \eqref{eq:BesselInductiveForm2} (or possibly deeper in the inductive definition).
Since each $x$ coordinate $u_{n,i}$ appears in a numerator as either $u_{n,i}$ or $\sqrt{1+u_{n,i}^2}$ (possibly deeper in the inductive definition through $\wtilde{y},\wtilde{t}$ or $\wtilde{\alpha}$) and only ever once in a single summand, we immediately have part 1.
\end{proof}

\subsubsection{The integral transform side}
For $S \subset [n-1]$, define $\rho_S \in M_n$ and $\hat{\rho}_{S^c}$ as follows:
If $i_1,i_2 \in S\cup\set{0,n}$ with $i_1 < i_2$ and $i_1 < i_3 < i_2 \Rightarrow i_3 \notin S$, define
\[ \rho_{S, i_3}=\frac{\hat{\rho}_{i_2}-\hat{\rho}_{i_1}}{i_2-i_1} = \frac{n-i_1-i_2}{2} \]
for all $i_1 < i_3 \le i_2$; in particular, we have $\hat{\rho}_{S,\hat{r}_i} = \hat{\rho}_{\hat{r}_i}$ for all $i =1, \ldots, \ell-1$.
Define $\hat{\rho}_{S^c}=(\hat{\rho}_{i'_1},\ldots,\hat{\rho}_{i'_{n-1-q}})$, where $S^c=\set{i'_1,\ldots,i'_{n-1-q}}$ in increasing order.
Note that if $\Re(\mu)=\alpha \rho_S$, then $\Re(\hat{\mu}_i) = \alpha \hat{\rho}_i$ for all $i \in S$ and any $\alpha \in \R$, and if also $\Re(s_{S^c})=-\alpha\hat{\rho}_{S^c}$, then
\begin{align}
\label{eq:TotalYPower}
	\abs{p_{S,\mu}(y) \hat{p}_{S^c,-s_{S^c}}(y)} = p_{\alpha\rho}(y).
\end{align}

Applying \eqref{eq:WhittContShift}, we have
\begin{align*}
	\mathcal{I} =& \sum_{S \subset [n-1]} \sum_{w_b \in W/W_S} \int_{\wbar{U}_w(\R)} \psi_y(w x y') p_\rho(y w x y') \int\limits_{\substack{M_n\\ \Re(\mu)=\alpha\rho_S}} f(\mu^{w_b^{-1}}) p_{S,\mu+\beta^{w_b}}(y w x y')\\
	& \qquad \int\limits_{\Re(s_{S^c})=-\alpha\hat{\rho}_{S^c}} \what{W}_{n,\sigma}(s_{S^c},\mu+\beta^{w_b},\delta) \hat{p}_{S^c,-s_{S^c}}(y w x y') \frac{ds_{S^c}}{(2\pi i)^{\abs{S^c}}} \\
	& \qquad \sigma(y w x y') \frac{\dspec\mu}{(2\pi i)^{n-1}} \wbar{\psi_I(x)} dx,
\end{align*}
where $\alpha=n^{-1} R_1^{-1}$.
(We will use $\alpha$ as a parameter in our contour shifting arguments.)

\subsubsection{Contour shifting}
We pause to consider the contour shifting implicit in the above.
Starting from $\Re(\mu)=0$ and the contours $\Re(s_{S^c})=-\epsilon$ in \eqref{eq:WhittContShift}, we have shifted first $\Re(\mu)$ leftward to $-\alpha\rho_S$ and then $\Re(s_{S^c})$ rightward to $-\alpha\hat{\rho}_{S^c}$.
The poles of $\what{W}_{n,\sigma}(s,\mu,\delta)$ occur whenever $-s=\hat{\mu}^{w_c}+m$ for some $m \in \N_0^{n-1}$ and $w_c \in W$, and our contour shifting in $\mu$ is certainly small enough that we won't cross any poles, so when we start shifting the $s_{S^c}$ contours rightward, we have $\Re(s_i+\hat{\mu}^{w_c}_i) \in (-1,0)$ for all $i \in S^c$, $w_c \in W_S$ and $\what{W}_{n,\sigma}(s_{S^c},\mu,\delta)$ is holomorphic on this region.
Note that this condition forms a convex subset of the parameter space $(\Re(s_{S^c}), \Re(\mu))$.
Therefore, if the endpoint of the contour shifting also satisfies this condition, the contour shifting may be performed by a (linear) substitution on the $s_{S^c}$ and a contour shift on a single $s_i$ integral (which is trivial to accomplish by the rapid decay of $\what{W}_{n,\sigma}(s_{S^c},\mu,\delta)$), while remaining in the region of holomorphy of $\what{W}_{n,\sigma}(s_{S^c},\mu,\delta)$.
Again, it is sufficient to check that $\Re(s_i+\hat{\mu}^{w_c}_i) < 0$, since the shifts are small and in fact we are shifting the $s_i$ contours in the positive direction.

Now the coordinates of $\mu^{w_c}$ for an arbitrary $w_c \in W_S$ are determined as follows:
if $i_1,i_2\in S\cup\set{0,n}$ with $i_1<i_2$ and $i_1 < i_3 < i_2 \Rightarrow i_3 \notin S$, then $\set{\mu^{w_c}_{i_1+1}, \ldots, \mu^{w_c}_{i_2}} = \set{\mu_{i_1+1},\ldots,\mu_{i_2}}$, since the powers of $y_{i_1}$ and $y_{i_2}$ in the restricted power function $p_{S,\mu^{w_c}}(y)$ must agree with $p_{S,\mu}(y)$.
We see that
\begin{align*}
	\alpha^{-1}\Re(\hat{\mu}^{w_c}_{i_3}) =& \hat{\rho}_{i_1}+\rho_{S,i_1+1}+\ldots+\rho_{S,i_3} \\
	=& \hat{\rho}_{i_1}+(i_3-i_1)\frac{n-i_1-i_2}{2} \\
	<& \hat{\rho}_{i_1}+(i_3-i_1)\frac{n-i_1-i_3}{2} \\
	=& \hat{\rho}_{i_1}+\rho_{i_1+1}+\ldots+\rho_{i_3} \\
	=& \hat{\rho}_{i_3} \\
	=& -\alpha^{-1}\Re(s_{i_3}),
\end{align*}
for all $i_1 < i_3 < i_2$, and in particular, $\Re(s_{i_3}+\hat{\mu}^{w_c}_{i_3})<0$ holds whenever $i_3 \in S^c$.

\subsubsection{Returning to the main argument}
For $j \in \mathcal{S}$, the total exponent on $y_j$ in our previous expression for $\mathcal{I}$ has real part $(1+\alpha) \hat{\rho}_j$ and as above, we have
\[ \frac{1}{\alpha} \Re(\hat{\mu}^{w_b^{-1}}_j) \le \hat{\rho}^{w_b^{-1}}_j \le \hat{\rho}_j. \]
Here the first inequality is strict when $\hat{\mu}^{w_b^{-1}}_j=\hat{\mu}_i$ with $i \in S^c$, and in that case
\[ \frac{1}{\alpha} \Re(\hat{\mu}^{w_b^{-1}}_j) \le \hat{\rho}_j-\frac{1}{2} \]
by the computation of the previous section.
The second inequality is strict when $w_b$ fails to fix the set $\set{\mu_1,\ldots,\mu_j}$ and in that case
\[ \frac{1}{\alpha} \Re(\hat{\mu}^{w_b^{-1}}_j) \le \hat{\rho}_j-2. \]

So we shift $\alpha \mapsto \paren{1+n^{-3}} R_1^{-1}$, possibly encountering poles at $\hat{\mu}^{w_b^{-1}}_j = R_1^{-1} \hat{\rho}_j$.
By the above reasoning and recalling \eqref{eq:TotalYPower}, the least power of $y$ (that is, the least power of $e^{-R_1^2}$) occurring in the shift is $p_{(1+R_1^{-1})\rho}(y)$ and this happens when $\mathcal{S}\subset S$ and $w_b \in W_w$.
(One might be confused, as the author sometimes is, why the minimum power is not $p_{\rho+R_1^{-1}\rho^{w_l}}(y) = p_{(1-R^{-1})\rho}(y)$ at $w_b=w_l$; the answer is simply that when $w_b=w_l$, shifting $\alpha$ in the positive direction means heading away from the poles, so the contour shifting actually ends at power $p_{(1+(1+n^{-3})R_1^{-1})\rho}(y)$ without encountering a single pole.)
For the other terms, we gain at least a power $\frac{1}{n^3 R_1}$ of some $y_j$, $j\in\mathcal{S}$, and we have
\[ \abs{p_\rho(w x y') p_{S,\mu+\beta^{w_b}}(w x y') \hat{p}_{S^c,-s_{S^c}}(w x y')} \ll_{y'} p_{\paren{1+\paren{1-n^{-2}}R_1^{-1}}\rho}(w x y'), \]
by \cref{lem:IwaBounds}.
So the $x$ integral converges with a pole of order at most $\frac{n(n-1)}{2}$ at $R_1=\infty$.
We handle the residues by substituting
\[ \mu^{w_b^{-1}}=\mu_{\mathcal{S}}+\sum_{i=1}^{\ell-1} t_i (e_{\hat{r}_i}-e_{\hat{r}_i+1}), \]
where $\mu_{\mathcal{S}} \in M_{\mathcal{S}}$ and then taking the residues at each
\[ t_i = R_1^{-1} \hat{\rho}_{\hat{r}_i}. \]
Therefore,
\begin{align*}
	T^{-1} \mathcal{I} =& \paren{\prod_{i \in \mathcal{S}} e^{\hat{\rho}_i^2/R_1^2}} \sum_{\mathcal{S} \subset S \subset [n-1]} \sum_{w_b \in W_w/W_S} \int_{\wbar{U}_w(\R)} \psi_y(w x y') p_\rho(w x y') \\
	& \qquad \int\limits_{\substack{M_{\mathcal{S}}\\ \Re(\mu+R_1^{-1}\rho^*)=R_1^{-1}\rho^{w_b^{-1}}_S}} f_{\mathcal{S}^c}^*(\mu+R_1^{-1}\rho^*) p_{S,(\mu+R_1^{-1}\rho^*+\beta)^{w_b}}(w x y') \\
	& \qquad \int\limits_{\Re(s_{S^c})=-R_1^{-1}\hat{\rho}_{S^c}} \what{W}_{n,\sigma}(s_{S^c},(\mu+R_1^{-1}\rho^*+\beta)^{w_b},\delta) \hat{p}_{S^c,-s_{S^c}}(w x y') \frac{ds_{S^c}}{(2\pi i)^{\abs{S^c}}} \\
	& \qquad \sigma(w x y') \frac{\dspec\mu}{(2\pi i)^{n-\ell}} \wbar{\psi_I(x)} dx+\BigO{R_1^{n^2} e^{-n^{-3} R_1}},
\end{align*}
where
\[ \rho^* = \sum_{i=1}^{\ell-1} \hat{\rho}_{\hat{r}_i} (e_{\hat{r}_i}-e_{\hat{r}_i+1}), \]
using the standard basis elements $e_i$.

A reality check for the $\mu$ contour:
To determine if this contour is possible, it is sufficient to check that $\rho^{w_b^{-1}}_S-\rho^* \in M_\mathcal{S}$, or in other words, that
\[ \hat{\rho}^{w_b^{-1}}_{S,\hat{r}_i} = \hat{\rho}_{S,\hat{r}_i} = \hat{\rho}_{\hat{r}_i} = \hat{\rho}^*_{\hat{r}_i}, \qquad i=1,\ldots,\ell-1, \]
and this is true since $w_b^{-1} \in W_w$ and $\mathcal{S} \subset S$ and by construction of $\rho^*$.

Now we may remove the factor $\psi_y(w x y')$ by applying
\[ \abs{\e{y_j x^*_{j,j+1}}-1} \ll \abs{y_j x^*_{j,j+1}}^{\frac{1}{n^3 R_1}} \ll_{n,y'} \abs{y_j}^{\frac{1}{n^3 R_1}} p_{-\frac{1}{n^3 R_1}\rho}(w x y'), \]
for each $j \in\mathcal{S}$, by \cref{lem:IwaBounds}, where $x^* y^* k^* = w x y'$.
Thus we have
\begin{align*}
	T^{-1} \mathcal{I} =& \paren{\prod_{i \in \mathcal{S}} e^{\hat{\rho}_i^2/R_1^2}} \int_{\wbar{U}_w(\R)} \mathcal{I}_1(x) \wbar{\psi_I(x)} dx+\BigO{R_1^{n^2} e^{-n^{-3} R_1}}, \\
	\mathcal{I}_1(x) :=& \int_{M_{\mathcal{S}}^0} f_{\mathcal{S}^c}^*(\mu+R_1^{-1}\rho^*) \mathcal{I}_2(x,\mu+R_1^{-1}\rho^*+\beta,\delta) \frac{\dspec\mu}{(2\pi i)^{n-\ell}}\\
	\mathcal{I}_2(x,\mu,\delta) :=& \sum_{\mathcal{S} \subset S \subset [n-1]} \sum_{w_b \in W_w/W_S} p_\rho(w x y')  p_{S,\mu^{w_b}}(w x y') \int\limits_{\Re(s_{S^c})=-\epsilon} \what{W}_{n,\sigma}(s_{S^c},\mu^{w_b},\delta) \\
	& \qquad \hat{p}_{S^c, -s_{S^c}}(w x y') \frac{ds_{S^c}}{(2\pi i)^{\abs{S^c}}} \sigma(w x y'),
\end{align*}
after shifting the $\mu$ and $s_{S^c}$ contours in the interated integral (as we may).

\subsubsection{Asymptotics and functional equations of the Whittaker function}
Suppose that
\[ 0<\Re(\mu_j-\mu_i)<\frac{\epsilon}{n^2}, \qquad \forall i<j, \]
then by apply the exact same reasoning as the previous section directly to \eqref{eq:WhittContShift}, we see that $\mathcal{I}_2$ is the first term asymptotic of 
\begin{align*}
	W_\sigma(y w x y', \mu, \delta) =& p_{\rho+\mu}(y) \chi_\delta(y) M_\sigma(w_l,\mu,\delta) W_\sigma(w x y', w_l, \mu^{w_l}, \delta^{w_l},\psi_y) \\
	\sim& p_{\rho+\mu}(y) \chi_\delta(y) M_\sigma(w_l,\mu,\delta) W_\sigma(w x y', w_l, \mu^{w_l}, \delta^{w_l},\psi_{\alpha_w})
\end{align*}
as in \eqref{eq:WhittYwAsymp}.
So we compute that from the Jacquet integral and then apply analytic continuation in $\mu$.
Note: The reversal from the region $\Re(\mu_i-\mu_j) > 0,\forall i<j$ where we did the work of the previous section to the region $\Re(\mu_j-\mu_i) > 0,\forall i<j$ used above has to do with the contour shifting and the pole of $f_1$ -- when $\Re(\mu)$ is fixed, then the lead asymptotic in $y_j$ is the one encountered first while shifting the $s$ contours to the left.

From \eqref{eq:LimitJacquetInt}, we have
\begin{equation}
\label{eq:I2IntEval}
\begin{aligned}
	& \int_{\wbar{U}_w(\R)} \mathcal{I}_2(x,\mu,\delta) \wbar{\psi_I(x)} dx = \\
	& \qquad p_{\rho+\mu}(y) \chi_\delta(y) M_\sigma(w_l,\mu,\delta) W_\sigma(I, w^{-1}, \mu^{w_l}, \delta^{w_l}, \psi_0) W_\sigma(y',\mu^{w_l w^{-1}},\delta^{w_l w^{-1}}),
\end{aligned}
\end{equation}
provided the $x$ integral converges, and $\Re(\mu) = R_1^{-1} \rho^{w w_l}$ is sufficient for that.

Now we would like to shift the $\mu$ integral in $\mathcal{I}_1(x)$ to $\Re(\mu)=R_1^{-1} (\rho^{w w_l}-\rho^*)$ in order to accommodate that condition.
To determine if this is possible, it is sufficient to check that $\rho^{w w_l}-\rho^* \in M_\mathcal{S}$, or in other words, that
\[ \hat{\rho}_{\hat{r}_i}^{w w_l} = \hat{\rho}^*_{\hat{r}_i} = \hat{\rho}_{\hat{r}_i}, \qquad i=1,\ldots,\ell-1, \]
and since $w w_l = \diag(w_{l,r_1},\ldots,w_{l,r_\ell})$, it is trivial that $w w_l$ fixes the required blocks of coordinates.
Then we may apply \eqref{eq:I2IntEval} and the end result is
\begin{align*}
	T^{-1} \mathcal{I} =& \paren{\prod_{i \in \mathcal{S}} e^{\hat{\rho}_i^2/R_1^2}} \int\limits_{\substack{M_{\mathcal{S}}\\ \Re(\mu)=R_1^{-1} (\rho^{w w_l}-\rho^*)}} f_{\mathcal{S}^c}(\mu+R_1^{-1}\rho^*) M_\sigma(w_l,\mu+R_1^{-1}\rho^*+\beta,\delta) \\
	& \qquad W_\sigma(I, w^{-1}, (\mu+R_1^{-1}\rho^*+\beta)^{w_l}, \delta^{w_l},\psi_0) \\
	& \qquad W_\sigma(y', (\mu+R_1^{-1}\rho^*+\beta)^{w_l w^{-1}}, \delta^{w_l w^{-1}}) \frac{\dspec\mu}{(2\pi i)^{n-\ell}}+\BigO{R_1^{n^2} e^{-n^{-3} R_1}} \\
	=& \int_{M_{\mathcal{S}}^0} f_{\mathcal{S}^c}(\mu) M_\sigma(w_l,\mu+\beta,\delta) W_\sigma(I, w^{-1}, (\mu+\beta)^{w_l}, \delta^{w_l},\psi_0) \\
	& \qquad W_\sigma(y', (\mu+\beta)^{w_l w^{-1}}, \delta^{w_l w^{-1}}) \frac{\dspec\mu}{(2\pi i)^{n-\ell}}+\BigO{R_1^{-1}},
\end{align*}
by the Mean Value Theorem.

Applying \eqref{eq:CtildeDef}, we have
\begin{align*}
	T^{-1} \mathcal{I} =& \int_{M_{\mathcal{S}}^0} f_{\mathcal{S}^c}(\mu) \wtilde{C}_w(\mu+\beta, \delta) W_\sigma(y', \mu+\beta, \delta) \frac{\dspec\mu}{(2\pi i)^{n-\ell}}+\BigO{R_1^{-1}}.
\end{align*}

\subsubsection{The power series side}
From \eqref{eq:KwPS}, we have
\begin{align*}
	\mathcal{I} =& \sum_{w_b\in W/W_w} \int_{M_n^0} f(\mu^{w_b^{-1}}) C_w^*(\mu+\beta^{w_b},\delta^{w_b}) \\
	& \qquad J_w^*(y,\mu+\beta^{w_b},\delta^{w_b}) W_\sigma(y', \mu^{w_b^{-1}}+\beta^{w_a},\delta^{w_a}) \frac{\dspec\mu}{(2\pi i)^{n-1}}.
\end{align*}

As before, we shift to $\Re(\mu) = \epsilon \rho$ and the leading term occurs when $w_b \in W_w$, so
\begin{align*}
	\mathcal{I} =& \int\limits_{\substack{M_{\mathcal{S}}\\ \Re(\mu)=R_1^{-1}(\rho-\rho^*)}} f_{\mathcal{S}^c}(\mu) C_w^*(\mu+R_1^{-1}\rho^*+\beta,\delta) J_w^*(y,\mu+R_1^{-1}\rho^*+\beta,\delta) \\
	& \qquad W_\sigma(y', \mu+R_1^{-1}\rho^*+\beta,\delta) \frac{\dspec\mu}{(2\pi i)^{n-\ell}} + \BigO{T e^{-R_1}}.
\end{align*}
Note that the presumed bound on the coefficients in part 6 of the \DEPSC implies that $J_w^*(y,\mu,\delta) \ll_{\mu,n} p_{\rho+\Re(\mu)}(y)$ on the region with all $y_i \ll 1$.

Applying the power series expansion of $J_w^*$, we have
\begin{align*}
	T^{-1} \mathcal{I} =& \int\limits_{\substack{M_{\mathcal{S}}\\ \Re(\mu)=R_1^{-1}(\rho-\rho^*)}} f_{\mathcal{S}^c}(\mu+R_1^{-1}\rho^*) C_w^*(\mu+R_1^{-1}\rho^*+\beta,\delta) \\
	& \qquad W_\sigma(y', \mu+R_1^{-1}\rho^*+\beta,\delta) \frac{\dspec\mu}{(2\pi i)^{n-\ell}} + \BigO{e^{-R_1}} \\
	=& \int_{M_{\mathcal{S}}^0} f_{\mathcal{S}^c}(\mu) C_w^*(\mu+\beta,\delta) W_\sigma(y', \mu+\beta,\delta) \frac{\dspec\mu}{(2\pi i)^{n-\ell}} + \BigO{R_1^{-1}}.
\end{align*}

\subsection{\texorpdfstring{Computing the $C^*_w(\beta,\delta)$ coefficients}{Computing the C*w(beta,delta) coefficients}}
\label{sect:CwComp}

Since the degenerate Whittaker function is also the intertwining operator and induces a map between Whittaker models, multiplicity one implies it acts as a scalar multiple of the functional equation \eqref{eq:WhitFE}.
So we can take $C^\dagger_w(\beta, \delta)$ to be the scalar determined by
\begin{align}
\label{eq:CdaggerDef}
	W_\sigma(I, \mu, \delta,w,\psi_0) = C^\dagger_w(\mu, \delta) M_\sigma(w,\mu, \delta).
\end{align}
This definition relates to $\wtilde{C}_w(\mu,\delta)$ as in \cref{lem:InformalAC} by
\begin{align}
\label{eq:localCoefRel}
	C^\dagger_w(\mu,\delta) =& \wtilde{C}_{w^{-1}}(\mu^{w_l}, \delta^{w_l}) = C^*_{w^{-1}}(\mu^{w_l}, \delta^{w_l}).
\end{align}
These are the local coefficients in Shahidi's chapter 5 \cite{Shah01}.

Using \cite[Theorem 4.2.2]{Shah01} and the definition of $M_\sigma(w,\mu, \delta)$, we have the factorizations
\begin{align*}
	W_\sigma(I, w_a w_b, \mu, \delta,\psi_0) =& W_\sigma(I, w_a, \mu, \delta,\psi_0) W_\sigma(I, w_b, \mu^{w_a}, \delta^{w_a},\psi_0), \\
	M_\sigma(w_a w_b,\mu, \delta) =& M_\sigma(w_a,\mu, \delta) M_\sigma(w_b,\mu^{w_a}, \delta^{w_a})
\end{align*}
and it follows that
\[ C^\dagger_{w_a w_b}(\mu, \delta) = C^\dagger_{w_a}(\mu, \delta) C^\dagger_{w_b}(\mu^{w_a}, \delta^{w_a}). \]
In \cite[Theorem 4.2.2]{Shah01}, the notation is rather different and there are conditions on $\mu$, but with analytic continuation, they are sufficient to factor any given $w$ into a product of transpositions of adjacent indices.

By this factorization of the functional equations, it's enough to consider $w$ a transposition of adjacent indices when computing $C^\dagger_w(\beta,\delta)$.
By conjugating $\sigma$ as necessary (using the inductive formula for the Whittaker function \eqref{eq:WhittDecomp}), it's enough to consider the Whittaker function on the upper left copy of $GL(2,\R)$.
Comparing \eqref{eq:GL2WhittFE} to \eqref{eq:GL2Whitt} at $y=0$ using \eqref{eq:classWhittDef}, we have
\[ C^\dagger_w(\mu,\delta) = (-1)^{\delta_1} \pi^{\frac{1}{2}-\mu_1+\mu_2} i^m \frac{\Gamma\paren{\frac{m+\mu_1-\mu_2}{2}}}{\Gamma\paren{\frac{1+m-\mu_1+\mu_2}{2}}}, \]
where $\set{0,1} \ni m \equiv \delta_1+\delta_2 \pmod{2}$.

Part 2 of the \Asymps follows by induction over the factorization $w = w_{(w(n)\,w(n)+1)} w'$ where now $w'(n)=w(n)+1$ and we use the fact that $w^{-1}(j) = j^w$.

To prove the last piece of the \Asymps, it's enough to show four things:
For $u \in \C$, $\delta \in \set{0,1}^n$ and $\set{0,1} \ni m \equiv \delta_1+\delta_2 \pmod{2}$,
\[ \sinmu(\mu^w,\delta^w) = \sinmu(\mu,\delta) \sgn(w), \]
\[ \pi^{\frac{1}{2}+u} \frac{\Gamma\paren{\frac{m-u}{2}}}{\Gamma\paren{\frac{1+m+u}{2}}} \times (2\pi)^{-u} \Gamma\paren{1+u} \times \sin \frac{\pi}{2}\paren{m+u} = (-1)^m (-\pi), \]
\[ (-1)^{\delta_2} i^m \times \frac{\sin \frac{\pi}{2}\paren{u+\delta_1-\delta_2}}{\sin \frac{\pi}{2}\paren{m+u}} \times (-1)^m = (-i)^{\delta_1+\delta_2} \]
\[ \prod_{j<k} (-i)^{\delta_j+\delta_k} = i^{-(n-1)(\delta_1+\ldots+\delta_n)}. \]
The last one is $\#\set{j:1 \le j < q}+\#\set{k:q<k\le n}=n-1$, and the rest are fairly straightforward.

\section{The long Weyl element}
\label{sect:LE}
\subsection{The Differential Equations}
\label{sect:LEDiffEqs}
We assume the \IoI and \AnCont Conjectures and prove the \DEPSC for the long Weyl element, outside of the exceptional cases $\mu_i - \mu_j \in \Z$, $i \ne j$.
Consider the function
\[ f(g) := p_{-\rho}(g) K_{w_l}(g \trans{g}, \mu, \delta), \]
which satisfies
\[ f(xyk) = \psi_I^2(x) f(y). \]

We wish to show that $\Delta f = \lambda_{2\mu}(\Delta) f$ for every $\Delta$ in the center of the universal enveloping algebra (i.e. for the Casimir operators), where
\[ \Delta p_{\rho+\mu} = \lambda_{\mu}(\Delta) p_{\rho+\mu}. \]
Let $\trans{(x^*)} y^* k^* = w_lxt$ in \eqref{eq:KwDef}, then
\[ \Delta p_{-\rho}(g) W_\sigma(g \trans{g} w_lxt,\mu,\delta) = \wbar{\psi_I(x^*)} p_{-\mu^{w_l}}(y^*) \Delta p_{-\rho}(y^* x^* g) W_\sigma(y^* x^* g \trans{g} \trans{(x^*)} y^*,w_l,\mu,\delta,\psi_{(y^*)^{-1}}), \]
and by translation invariance, it suffices to determine
\[ \Delta p_{-\rho}(g) W_\sigma(g \trans{g},w_l,\mu,\delta,\psi). \]

Now we assume $\mu$ is in the region of absolute convergence of the Whittaker function, so we need to compute
\[ \Delta p_{-\rho}(g) I_{\mu,\sigma}(w_l u g \trans{g}), \]
and again we let $x^\dagger y^\dagger k^\dagger = w_l u$ so that the above may be written
\[ p_{\rho+\mu}(y^\dagger) \Delta p_{-\rho}(g) I_{\mu,\sigma}(k^\dagger g \trans{g} \trans{(k^\dagger)}) \sigma(k^\dagger). \]
Because $\Delta$ may be written as a linear combination of compositions of first-order derivatives (coming from the Lie algebra), each of which is left-translation invariant, it is sufficient to determine $\Delta f^*$ where $f^*(g) = p_{-\rho}(g) p_{\rho+\mu}(g \trans{g})$.
Now $\Delta_G f^*(g) = \Delta_{G/K} f^*(z)$ by $K$-invariance (i.e. using $k \exp(tX) = \exp(tkXk^{-1})k$), and it is easy to see that $\Delta_{G/K} f^*(z) = \lambda_{2\mu}(\Delta) f^*(z)$ since the equivalent statement holds for every element of the Lie algebra.
(For $u = E_{i,j} \in \trans{\mathfrak{n}}$, i.e. $i > j$, write $u = \trans{u}+(u-\trans{u})$ and note that $u-\trans{u} \in \mathfrak{k}$.)

We denote by $2 \in Y^+$ the matrix such that $\psi_y^2(x) = \psi_I(2y x (2y)^{-1})$.
We now believe that $f(g)$ is the spherical Whittaker function; that is,
\[ K_{w_l}(y^2, \mu, \delta) = C(\mu) p_\rho(y) W_1(2y,2\mu, 0) \]
for $y\in Y^+$ and some constant $C(\mu)$.
We check this by comparing asymptotics as $y\to 0$:
By the functional equations of the spherical Whittaker function,
\begin{align*}
	\Lambda_\text{JW}(2\mu) W_1(2y,2\mu, 0) \sim& \sum_{w\in W} p_{\rho+2\mu^{ww_l}}(2y) \Lambda_\text{JW}(2\mu^w) W_1(I,w_l,2\mu^w, 0, \psi_0), \\
	\Lambda_\text{JW}(\mu) :=& \prod_{1 \le i < j \le n} \pi^{-\frac{1+\mu_i-\mu_j}{2}}\Gamma\paren{\tfrac{1+\mu_i-\mu_j}{2}}, \\
	W_1(I,w_l,\mu, 0, \psi_0) =& \prod_{1 \le i < j \le n} \pi^{\frac{1}{2}} \frac{\Gamma\paren{\frac{\mu_i-\mu_j}{2}}}{\Gamma\paren{\frac{1+\mu_i-\mu_j}{2}}}
\end{align*}
and from the \Asymps (for $y \in Y^+$),
\begin{align}
\label{eq:KwlAsymp}
	\sinmu(\mu,\delta) K_{w_l}(y,\mu,\delta) \sim& (-\pi)^{\frac{n(n-1)}{2}}i^{-(n-1)(\delta_1+\ldots+\delta_n)} \sum_{w\in W} \frac{\sgn(w)}{\Lambda_{w_l}(\mu^w)} p_{\rho+\mu^w}(y).
\end{align}
The ratio
\[ \frac{\Lambda_\text{JW}(2\mu) \sgn(w)}{\Lambda_{w_l}(\mu^w) p_{\rho+2\mu^w}(2) \Lambda_\text{JW}(2\mu^{ww_l}) W(I,2\mu^{ww_l}, \psi_0)} \]
is invariant under $w$ (consider the case of a transposition of consecutive indices), as required, and we have
\[ C(\mu) = \frac{i^{-(n-1)(\delta_1+\ldots+\delta_n)} \sinmu(2\mu,0) \Lambda_\text{JW}(2\mu)}{p_\rho(2) \sinmu(\mu,\delta)}, \]
after some simplification.
Note:  It might be possible to prove instead that $p_{-\rho}(y) K_{w_l}(y^2, \mu,\delta)$ has moderate growth in $y$, but this isn't entirely obvious, a priori.

Michihiko Hashizume \cite{Hashi} shows that the space of spherical Whittaker functions (not necessarily of moderate growth in $y$) has a spanning set given by power series of the form $J^*_{w_l}(y^2,(\mu/2)^w)$, $w\in W$, for some coefficients $a_{w_l,m}(\mu)$ determined recursively by the action of the Casimir operators.
It follows that for positive $y$, $K_{w_l}(y,\mu,\delta)$ has the form \eqref{eq:KwPS}.
Now the restriction of the Casimir operators to the space of long-element Bessel functions (i.e. functions that transform by $\psi_I$ on both the left and right) is composed of polynomials in $y_i$ and $y_i \partial_{y_i}$, $i=1,\ldots,n-1$ (for such a function $f$, consider the Lie algebra action of $E_{i,j}$ on $f(y)$, then $i>j$ multiplies by constants, $j>i$ multiplies by coordinates of $y$ and $2\pi i$, and $j=i$ have the form in $y_i \partial_{y_i}$); these operators act on $p_{\rho+\mu}(y) \hat{p}_m(y)$ by
\[ y_i^{k_1} (y_i \partial_{y_i})^{k_2} p_{\rho+\mu}(y) \hat{p}_m(y) = \paren{\hat{\rho}_i + \hat{\mu}_i+m_i}^{k_2} p_{\rho+\mu}(y) \hat{p}_{m+k_1 e_i}(y), \]
(using the standard basis element $e_i$), and this holds independent of the sign of $y$.
In particular, \eqref{eq:KwPS} holds for every sign (at $w=w_l$), since the asymptotics of $K_{w_l}(y,\mu,\delta)$ have the correct dependence on $\sgn(y)$ (that is, the coefficient of $p_{\rho+\mu^w}(y)$ transforms by $\chi_{\delta^w}(y)$).

\subsection{Integral Representations of the Long-Element Bessel Function}
\subsubsection{Ishii's Theorem}
Ishii \cite[Theorem 4]{Ishii} gave a useful recurrence relation for the coefficients of the power series solutions to the Whittaker differential equations on $GL(n)$, and we record it here as it applies to the long-element Bessel functions, anticipating its utility in the construction of integral representations.

For $k \in \N_0^{n-1}$ and $\mu \in \C^n$, set
\[ S(k) = \set{m \in \N_0^{n-3} \setdiv 0 \le m_i \le k_{i+1}}, \]
\[ P_{n,k}(\mu) = \prod_{j=1}^{n-1} \frac{\Poch{\mu_1-\mu_n+1}{k_j+k_{j+1}}}{\Poch{\mu_1-\mu_n+1}{k_j} \Poch{\mu_1-\mu_{j+1}+1}{k_j} \Poch{\mu_{j+1}-\mu_n+1}{k_{j+1}}}, \]
for $m \in S(k)$,
\begin{align*}
	& Q_{n,k,m}(\mu) = \\
	& \prod_{j=1}^{n-1} \frac{\Poch{\mu_{j+1}-\mu_n+k_{j+1}+1-m_{j-1}}{m_{j-1}} \Poch{\mu_1-\mu_n+k_j+1-m_{j-1}}{m_{j-1}}}{(k_j-m_{j-1})! \, \Poch{\mu_1-\mu_n+k_j+k_{j+1}+1-m_{j-1}-m_j}{m_{j-1}+m_j}} \\
	& \qquad \times \Poch{\mu_1-\mu_{j+1}+k_j+1-m_j}{m_j},
\end{align*}
\[ R_{n,k}(\mu) = \sum_{j=1}^{n-1} \paren{(k_j-k_{j-1})^2+2 k_j (\mu_j-\mu_{j+1})}, \]
using $m_j = 0$ for $j \notin \set{1,\ldots, n-3}$ and $k_j=0$ for $j \notin \set{1,\ldots,n-1}$.
Define $G_{n,k}(\mu)$ by
\[ G_{n,0}(\mu)=1, \qquad R_{n,k}(\mu) G_{n,k}(\mu) = \sum_{j=1}^{n-1} G_{n,k-e_j}(\mu). \]

Let
\[ \Omega_n = \set{\mu \in M_n \setdiv R_{n,k}(\mu^w) \ne 0, \forall k \in \N_0^{n-1}\setminus\set{0} \forall w \in W, \mu_i - \mu_j \notin \Z, i \ne j.} \]
Then for $\mu \in \Omega_n$, define $\mu' \in \C^{n-2}$ by $\mu'_j = \mu_{j+1}+\frac{\mu_1+\mu_n}{n-2}$ and suppose $\mu' \in \Omega_{n-2}$.
Then Ishii's recurrence relation is
\begin{align}
\label{eq:IshiiRecRel}
	G_{n,k}(\mu) = P_{n,k}(\mu) \sum_{m \in S(k)} G_{n-2,m}(\mu') Q_{n,k,m}(\mu),
\end{align}
and these compare to the coefficients $a_{w_l,m}(\mu)$ in the \DEPSC as
\[ a_{w_l,m}(\mu) = G_{n,m}(\mu) \prod_{j=1}^{n-1} \paren{4 \pi^2}^{m_j}. \]

Then \cite[Theorem 3]{Ishii} for $\Re(\mu_1-\mu_n) > -2$, and $y_j > 0$, becomes
\begin{align*}
	& J_{n,w_l}^\dagger(I, y, \mu) = \\
	& \prod_{j=1}^{n-1} y_j^{\frac{n-2j}{2(n-2)}(\mu_1+\mu_n)} \frac{1}{(2\pi i)^{n-2}} \int_{\abs{u_1}=1} \cdots \int_{\abs{u_{n-2}}=1} \prod_{j=1}^{n-1} I_{\mu_1-\mu_n}\paren{4 \pi\sqrt{y_j(1+u_{j-1})(1+1/u_j)}} \\
	& \qquad J_{n-2,w_l}^\dagger\paren{I, \paren{y_2 u_1/u_2,\ldots,y_{n-2} u_{n-3}/u_{n-2}}, \mu'} \prod_{j=1}^{n-2} u_j^{-\frac{n}{2(n-2)} (\mu_1+\mu_n)} \frac{du_j}{u_j},
\end{align*}
taking $u_0 = 1/u_{n-1}=0$.
Note: We've expressed Ishii's result in terms of the Bessel function by replacing his $y$ with $2\sqrt{y}$ and $\mu \mapsto 2\mu$.
To avoid confusion about branch cuts and convergence, it's best to imagine $\Im(\mu)=0$ with $\mu_1 > \mu_2 > \ldots > \mu_n$ and use analytic continuation on the final result.

\subsubsection{The other signs}
We now prove \cref{prop:IshiiRepns}.

The result of the analytic continuation of $J_{n,w_l}(I, y,\mu)$ from $\sgn(y)=I$ to $\sgn(y)=v$ by rotating each coordinate $y_j$ (having $v_j=-1$) in the counter-clockwise direction is given by
\[ J_{n,w_l}^\dagger(v, (\abs{y_1},\ldots,\abs{y_{n-1}}),\mu) \exp \paren{i \pi \sum_{v_j=-1} \hat{\mu}_j}. \]
Meanwhile, we have $I_\alpha(ix) = e^{i\frac{\pi}{2} \alpha} J_\alpha(x)$, so the analytic continuation of the right-hand side of Ishii's integral becomes
\begin{align*}
	& \prod_{j=1}^{n-1} \abs{y_j}^{\frac{n-2j}{2(n-2)}(\mu_1+\mu_n)} \frac{1}{(2\pi i)^{n-2}} \int_{\abs{u_1}=1} \cdots \int_{\abs{u_{n-2}}=1} \prod_{j=1}^{n-1} Z^{v_j}_{\mu_1-\mu_n}\paren{4\pi\sqrt{\abs{y_j}(1+u_{j-1})(1+1/u_j)}} \\
	& \times J_{n-2,w_l}^\dagger\paren{v', \paren{\abs{y_2} u_1/u_2,\ldots,\abs{y_{n-2}} u_{n-3}/u_{n-2}},\mu'} \prod_{j=1}^{n-2} u_j^{-\frac{n}{2(n-2)} (\mu_1+\mu_n)} \frac{du_j}{u_j} \\
	& \times \exp \paren{i\frac{\pi}{2}\paren{\sum_{v_j=-1}\paren{\frac{n-2j}{(n-2)}(\mu_1+\mu_n)+\mu_1-\mu_n}+2\sum_{\substack{2 \le j \le n-2\\v_j=-1}} \sum_{k=2}^j \paren{\mu_j+\frac{\mu_1+\mu_n}{n-2}}}},
\end{align*}
and the complex exponential is the same on both sides.

\section{Acknowledgments}
The author would like to thank Tyrone Crisp and Peter Humphries for helpful discussions on representation theory.

\appendix
\section{Matrix decompositions in Mathematica}
\label{sect:AppMD}
Here we give an implementation of the decompositions of \cref{sect:MatrixDecomps} in Mathematica.
The two main functions below are NiceIwaxb and NiceIwawx.
Both take a vector of parameters $rs=\set{r_1,\ldots,r_\ell}$ and return a tuple $\set{x, x^*, y^*, kss}$ where $kss$ is a vector of matrices and we take $k^*=$ Dot@@$kss$ to be the product.
For NiceIwaxb, $x$ is a choice of coordinates on $w_b \wbar{U}_{w_b}(\R) w_b^{-1}$ and the Iwasawa decomposition is $x=x^* y^* k^*$.
For NiceIwawx, $x$ is a choice of coordinates on $\wbar{U}_w(\R)$ and the Iwasawa decomposition is $wx=x^* y^* k^*$.\\[0.15in]

\begingroup
\ttfamily
\setlength{\parindent}{0pt}
DoSymbol[name\char`_, index\char`_] := Symbol[name \textless\textgreater ToString[index]]\\
DoSymbolList[name\char`_, length\char`_] := Module[\{i\}, \mlb Table[DoSymbol[name, i], \{i, 1, length\}]]\\
xmat[n\char`_, name\char`_] := Module[\{vars = DoSymbolList[name, (n (n - 1))/2], i, j\}, \mlb Table[If[i == j, 1, If[j \textgreater i, vars[[i + (j - 1) (j - 2)/2]], 0]], \mlb \{i, 1, n\}, \{j, 1, n\}]]\\
wl[rs\char`_] := Module[\{n = Length[rs], i, j\}, ArrayFlatten[Table[\mlb If[i + j == n + 1, IdentityMatrix[rs[[i]]], \mlb ConstantArray[0, \{rs[[i]], rs[[n + 1 - j]]\}]], \{i, 1, n\}, \{j, 1, n\}]]]\\
Embed[n\char`_, mat\char`_] := If[Length[mat] == n, mat, ArrayFlatten[\{\{mat, \mlb ConstantArray[0, \{Length[mat], n - Length[mat]\}]\}, \mlb \{ConstantArray[0, \{n - Length[mat], Length[mat]\}],\mlb  IdentityMatrix[n - Length[mat]]\}\}]]\\
getws[w\char`_] := Module[\{row\}, Table[Position[row, 1][[1, 1]], \{row, w\}]]\\
GetUw[w\char`_, x\char`_] := Module[\{n = Length[w], ws = getws[Transpose[w]]\}, \mlb Table[If[j \textless i, 0, If[j == i, 1, If[ws[[i]] \textless ws[[j]], x[[i, j]], 0]]], \mlb \{i, 1, n\}, \{j, 1, n\}]]\\

MDhatr[rs\char`_, i\char`_] := Total[rs[[;; i]]]\\
MDhatr[rs\char`_] := Table[MDhatr[rs, i], \{i, 1, Length[rs]\}]\\
MDnp[rs\char`_, m\char`_] := MDhatr[rs, Length[rs] - m]\\
MDwp[rs\char`_, m\char`_] := If[m \textless Length[rs], wl[rs[[;; -m - 1]]], Null]\\
MDwpp[rs\char`_, m\char`_] := wl[rs[[-m ;;]]]\\
MDwa[rs\char`_, m\char`_] := ArrayFlatten[\{\{MDwp[rs, m], 0\}, \{0, MDwpp[rs, m]\}\}]\\
MDwb[rs\char`_, m\char`_] := If[m == Length[rs], IdentityMatrix[Total[rs]], \mlb ArrayFlatten[\{\{0, IdentityMatrix[MDhatr[rs, Length[rs] - m]]\}, \mlb \{IdentityMatrix[Total[rs] - MDhatr[rs, Length[rs] - m]], 0\}\}]]\\

MDDecomps[rs\char`_, m\char`_, xa\char`_, xbs\char`_] := Module[\{n = Total[rs], wa = MDwa[rs, m], \mlbb wb = MDwb[rs, m], np = MDnp[rs, m], A, B, C, D, E, F, H\},\mlb
  A = xa[[;; np, ;; np]];\mlb
  B = xbs[[np + 1 ;;, np + 1 ;;]];\mlb
  C = xbs[[;; np, ;; np]];\mlb
  D = xbs[[;; np, np + 1 ;;]];\mlb
  E = xbs[[np + 1 ;;, np + 1 ;;]];\mlb
  F = GetUw[MDwp[rs, m], C];\mlb
  H = GetUw[MDwpp[rs, m], E];\mlb
  \{A, B, C, D, E, F, H\}]\\
  
MDhaty[n\char`_, m\char`_, i\char`_, x\char`_] := DiagonalMatrix[Join[ConstantArray[1, i - 1], \mlb \{(1 + x\^{}2)\^{}(-1/2)\}, ConstantArray[1, m - i - 1], \{(1 + x\^{}2)\^{}(1/2)\}, \mlb ConstantArray[1, n - m]]]\\
RevEmbed[n\char`_, s\char`_] := If[Length[s] \textgreater= n, s, \mlb ArrayFlatten[\{\{IdentityMatrix[n - Length[s]], 0\}, \{0, s\}\}]]\\
MDhatk2[n\char`_, x\char`_] := If[n \textless= 2, \{\{(1 + x\^{}2)\^{}(-1/2), -x (1 + x\^{}2)\^{}(-1/2)\}, \mlb \{x (1 + x\^{}2)\^{}(-1/2), (1 + x\^{}2)\^{}(-1/2)\}\}, ArrayFlatten[\{\{(1 + x\^{}2)\^{}(-1/2), \mlb 0, -x (1 + x\^{}2)\^{}(-1/2)\}, \{0, IdentityMatrix[n - 2], 0\}, \mlb \{x (1 + x\^{}2)\^{}(-1/2), 0, (1 + x\^{}2)\^{}(-1/2)\}\}]]\\
MDhatk[n\char`_, m\char`_, i\char`_, x\char`_] := Embed[n, RevEmbed[m, MDhatk2[m - i + 1, x]]]\\

MDucstarentry[i\char`_, j\char`_, u\char`_] := Module[\{l\}, \mlb -u[[i]] u[[j]] Product[(1 + u[[l]]\^{}2)\^{}(-1/2), \{l, i, j - 1\}]]\\
MDucstar[rs\char`_, m\char`_, u\char`_] := Module[\{n = Total[rs], np = MDnp[rs, m], i, j\}, \mlb Table[If[j \textless i, 0, If[j == i, 1, MDucstarentry[i, j, u]]], \{i, 1, np\}, \mlb \{j, 1, np\}]]\\
MDucstar[rs\char`_, m\char`_] := MDucstar[rs, m, DoSymbolList["u", Total[rs]]]\\
MDudstar[rs\char`_, m\char`_, u\char`_] := Module[\{n = Total[rs], np = MDnp[rs, m], i, j\},\mlb Transpose[\{Table[u[[i]] Product[(1 + u[[j]]\^{}2)\^{}(-1/2), \{j, 1, np\}]\mlb Product[(1 + u[[j]]\^{}2)\^{}(-1/2), \{j, i, np\}], \{i, 1, np\}]\}]]\\
MDudstar[rs\char`_, m\char`_] := MDudstar[rs, m, DoSymbolList["u", Total[rs]]]\\
MDuctildeentry[i\char`_, j\char`_] := Module[\{u = DoSymbolList["u", j], l\}, u[[i]] u[[j]] \mlb (1 + u[[i]]\^{}2)\^{}(-1/2) Product[(1 + u[[l]]\^{}2)\^{}(1/2), \{l, i + 1, j - 1\}]]\\
MDuctilde[rs\char`_, m\char`_] := Module[\{n = Total[rs], np = MDnp[rs, m], i, j\}, \mlb Table[If[j \textless i, 0, If[j == i, 1, MDuctildeentry[i, j]]], \{i, 1, np\}, \mlb \{j, 1, np\}]]\\
MDudtilde[rs\char`_, m\char`_] := Module[\{n = Total[rs], np = MDnp[rs, m], u, i, j\},\mlb u = DoSymbolList["u", n];Transpose[\{Table[-u[[i]] (1 + u[[i]]\^{}2)\^{}(-1) \mlb Product[(1 + u[[j]]\^{}2)\^{}(-1/2), \{j, 1, i - 1\}], \{i, 1, np\}]\}]]\\
MDxdag[rs\char`_, m\char`_] := Module[\{n = Total[rs], np = MDnp[rs, m]\}, If[n - np \textless= 1, \mlb ArrayFlatten[\{\{MDucstar[rs, m], MDudstar[rs, m]\}, \{0, 1\}\}], \mlb ArrayFlatten[\{\{MDucstar[rs, m], 0, MDudstar[rs, m]\}, \mlb \{0, IdentityMatrix[n - np - 1], 0\}, \{0, 0, 1\}\}]]]\\
MDxdaginv[rs\char`_, m\char`_] := Module[\{n = Total[rs], np = MDnp[rs, m]\}, \mlb If[n - np \textless= 1, ArrayFlatten[\{\{MDuctilde[rs, m], MDudtilde[rs, m]\}, \{0, 1\}\}], \mlb ArrayFlatten[\{\{MDuctilde[rs, m], 0, MDudtilde[rs, m]\}, \mlb \{0, IdentityMatrix[n - np - 1], 0\}, \{0, 0, 1\}\}]]]\\
MDydag[rs\char`_, m\char`_, u\char`_] := Module[\{n = Total[rs], np = MDnp[rs, m], i\},\mlb Dot @@ Table[MDhaty[n, n, i, u[[i]]], \{i, 1, np\}]]\\
MDydag[rs\char`_, m\char`_] := MDydag[rs, m, DoSymbolList["u", MDnp[rs, m]]]\\
MDkdagAll[rs\char`_, m\char`_, u\char`_] := Module[\{n = Total[rs], np = MDnp[rs, m], i\},\mlb Table[MDhatk[n, n, i, u[[i]]], \{i, np, 1, -1\}]]\\
MDxcdagentry[rs\char`_, u\char`_, xc\char`_, i\char`_] := Module[\{np = MDnp[rs, 1], j, l\}, \mlb Sum[xc[[np + i, j]] u[[j]]/Product[(1 + u[[l]]\^{}2)\^{}(1/2), \{l, 1, j\}], \mlb \{j, 1, MDnp[rs, 1]\}]]\\
MDxcdag[rs\char`_, u\char`_, xc\char`_] := Module[\{i\}, \mlb Table[\{MDxcdagentry[rs, u, xc, i]\}, \{i, 1, Total[rs] - MDnp[rs, 1] - 1\}]]\\

NiceIwaxb[rs\char`_] := NiceIwaxb[rs, xmat[Total[rs], "x"]]\\
NiceIwaxb[rs\char`_, origx\char`_] := NiceIwaxb[rs, Last[rs], origx]\\
NiceIwaxb[rs\char`_, 0, origx\char`_] := Module[\{n = Total[rs]\}, \mlb \{IdentityMatrix[n], IdentityMatrix[n], IdentityMatrix[n], \{\}\}]\\
NiceIwaxb[rs\char`_, ind\char`_, origx\char`_] := Module[\{n = Total[rs], np = Total[rs[[;; -2]]],\mlbb  rl = rs[[-1]], rps, ydag, wb, kdag, u, umat, xc, xcs, ycs, kcs, ucs, \mlbb uds, xcdag, mat1, mat2, mat3, i, j, l\},\mlb
  \{xc, xcs, ycs, kcs\} = NiceIwaxb[rs, ind - 1, origx];\mlb
  rps = Append[rs[[;; -2]], ind];\mlb
  xc = Embed[n, xc];\mlb
  xcs = Embed[n, xcs];\mlb
  ycs = Embed[n, ycs];\mlb
  u = origx[[ind, rl + 1 ;;]];\mlb
  ucs = MDucstar[rps, 1, u];\mlb
  uds = MDudstar[rps, 1, u];\mlb
  xcdag = MDxcdag[rps, u, xc];\mlb
  ydag = Embed[n, MDydag[rps, 1, u]];\mlb
  wb = MDwb[rs, 1];\mlb
  kdag = Map[Embed[n, \#] \&, MDkdagAll[rps, 1, u]];\mlb
  mat1 = Embed[n, If[ind == 1, ArrayFlatten[\{\{ucs, uds\}, \{0, 1\}\}], \mlbb ArrayFlatten[\{\{ucs, 0, uds\}, \{0, IdentityMatrix[ind - 1], xcdag\}, \mlbb \{0, 0, \{\{1\}\}\}\}]]];\mlb
  mat2 = Embed[n, ucs];\mlb
  umat = Table[If[i == ind, u[[j]] Product[(1 + u[[l]]\^{}2)\^{}(1/2), \mlbb \{l, 1, j - 1\}], 0], \{i, 1, rl\}, \{j, 1, np\}];\mlb
  mat3 = ArrayFlatten[\{\{IdentityMatrix[np], 0\}, \{umat, IdentityMatrix[rl]\}\}];\mlbb
  \{mat3 . mat2 . ydag . xc . Inverse[ydag] . Inverse[mat2],\mlbb mat1 . ydag . xcs . Inverse[ydag], ydag . ycs, \mlbb Join[Map[Embed[n, \#] \&, kcs], kdag]\}]\\

NiceIwawx[rs\char`_] := NiceIwawx[rs, xmat[Total[rs], "x"]]\\
NiceIwawx[rs\char`_ /; Length[rs] == 1, origx\char`_] := Module[\{n = Total[rs]\}, \mlb \{IdentityMatrix[n], IdentityMatrix[n], IdentityMatrix[n], \{\}\}]\\
NiceIwawx[rs\char`_, origx\char`_] := Module[\{n = Total[rs], rl = Last[rs], \mlb wp = MDwp[rs, 1], wa = MDwa[rs, 1], wb = MDwb[rs, 1], xa, xb, xbs, ybs, \mlb kbs, ybswa, Ap, Bp, C, D, E, F, H, xp, xs, ys, ks, mat1, mat2, \mlb mat3, mat4\},\mlb
  \{xa, xs, ys, ks\} = \mlbb NiceIwawx[rs[[;; -2]], (wb.origx.Inverse[wb])[[;; n - rl, ;; n - rl]]];\mlb
  \{xb, xbs, ybs, kbs\} = NiceIwaxb[rs, origx];\mlb
  xa = Embed[n, xa];\mlb
  xb = Inverse[wb] . Embed[n, xb] . wb;\mlb
  \{Ap, Bp, C, D, E, F, H\} = MDDecomps[rs, 1, ybs . xa . Inverse[ybs], xbs];\mlb
  mat1 = ArrayFlatten[\{\{IdentityMatrix[n - rl], \mlbb wp . F . Ap . Inverse[C] . D . Inverse[E]\}, \{0, IdentityMatrix[rl]\}\}];\mlb
  mat2 = ArrayFlatten[\{\{wp . F . Inverse[wp], 0\}, \{0, E\}\}];\mlb
  mat3 = ArrayFlatten[\{\{H, 0\}, \{0, F\}\}];\mlb
  mat4 = ArrayFlatten[\{\{Inverse[E], 0\}, \{0, Inverse[C]\}\}];\mlb
  ybswa = wa . ybs . Inverse[wa];\mlb
  \{mat3 . Inverse[wb] . ybs . xa . Inverse[ybs] . wb . mat4 . xb,\mlbb mat1 . mat2 . ybswa . Embed[n, xs] . Inverse[ybswa],\mlbb ybswa . Embed[n, ys], Join[Map[Embed[n, \#] \&, ks], kbs, \{wb\}]\}]\\
\endgroup

\bibliographystyle{amsplain}

\bibliography{HigherWeight}

\end{document}